\documentclass[12pt,reqno]{amsart} 

\usepackage{float} 

\usepackage{soul}

\usepackage{hhline}

\usepackage{epsfig,amssymb,amsmath,version}
\usepackage{amssymb,version,graphicx,fancybox,mathrsfs}

\usepackage{enumitem}

\usepackage{url,hyperref}
\usepackage{subfigure}
\usepackage{color}
\usepackage{stmaryrd}
\usepackage{multirow}
\usepackage{booktabs,siunitx}
\usepackage{booktabs}
\usepackage{multicol,epstopdf}
\usepackage{xypic}
\usepackage[]{graphicx}
\usepackage[all]{xy}
\usepackage{tikz,ifthen}
\usetikzlibrary{positioning, math, decorations.markings, arrows.meta}
\usetikzlibrary{calc,shapes,cd}
\usetikzlibrary{knots} 
\usepgfmodule{decorations}
\allowdisplaybreaks
\tikzset{knotarrow/.pic={ \draw[edge, <-] (0,0) -- +(-.001,0);}}

\tikzset{edge/.style={line width=0.8}}
\tikzset{wall/.style={very thick}}

\tikzset{->-/.style n args={2}{decoration={markings, mark=at position #1 with {\arrow{#2}}}, postaction={decorate}}} 

\ExplSyntaxOn
\cs_new_eq:NN \ifstreqF \str_if_eq:nnF
\cs_new_eq:NN \ifstreqTF \str_if_eq:nnTF
\ExplSyntaxOff

\tikzset{-o-/.code 2 args={\ifstreqF{#2}{} 
{\ifstreqTF{#2}{>}
   {\pgfkeysalso{decoration={markings,mark=at position #1 with {\arrow[scale=0.8]{#2}}}
                    ,postaction={decorate}}
    }
   {\ifstreqTF{#2}{<}
       {\pgfkeysalso{decoration={markings,mark=at position #1 with {\arrow[scale=0.8]{#2}}}
                    ,postaction={decorate}}
        }
       {\pgfkeysalso{decoration={markings,
                    mark=at position #1 with
                    {\draw[black, fill={#2}] circle[radius=2pt];}}
                    ,postaction={decorate}}
        }
     }
  }}}

\allowdisplaybreaks[4]

\newtheorem{theorem}{Theorem}[section]
\newtheorem{lemma}[theorem]{Lemma}
\newtheorem{definition}[theorem]{Definition}
\newtheorem{corollary}[theorem]{Corollary}
\newtheorem{proposition}[theorem]{Proposition}
\newtheorem{remark}[theorem]{Remark}
\newtheorem{conjecture}[theorem]{Conjecture}

\newcommand{\bp}{\begin{proposition}}
\newcommand{\ep}{\end{proposition}}
\newcommand{\bpr}{\begin{proof}}
\newcommand{\epr}{\end{proof}}

\newcommand{\bt}{\begin{theorem}}
\newcommand{\et}{\end{theorem}}

\newcommand{\bl}{\begin{lemma}}
\newcommand{\el}{\end{lemma}}

\newcommand{\bcr}{\begin{corollary}}
\newcommand{\ecr}{\end{corollary}}

\newcommand{\be}{\begin{equation}}
\newcommand{\ee}{\end{equation}}

\newcommand{\bes}{\begin{equation*}}
\newcommand{\ees}{\end{equation*}}

\newcommand{\ba}{\begin{align}}
\newcommand{\ea}{\end{align}}

\newcommand{\bas}{\begin{align*}}
\newcommand{\eas}{\end{align*}}

\DeclareMathOperator{\im}{\mathrm{Im}}
\DeclareMathOperator{\kernel}{\mathrm{ker}}

\graphicspath{{./Figures/} } 

\newcommand{\vs}[0]{\vspace{2mm}}

\begin{document}
\bibliographystyle{plain}

\title[The Unicity Theorem and the center of the ${\rm SL}_3$-skein algebra]{\resizebox{155mm}{!}{The Unicity Theorem and the center of the ${\rm SL}_3$-skein algebra}}

\author[Hyun Kyu Kim]{Hyun Kyu Kim}
\address{School of Mathematics, Korea Institute for Advanced Study (KIAS), 85 Hoegi-ro, Dongdaemun-gu, Seoul 02455, Republic of Korea}
\email{hkim@kias.re.kr}

\author[Zhihao Wang]{Zhihao Wang}
\address{Zhihao Wang, School of Physical and Mathematical Sciences, Nanyang Technological University, 21 Nanyang Link Singapore 637371}
\email{ZHIHAO003@e.ntu.edu.sg}
\address{University of Groningen, Bernoulli Institute, 9700 AK Groningen, The Netherlands}
\email{wang.zhihao@rug.nl}


 \maketitle

\begin{abstract}
The ${\rm SL}_3$-skein algebra $\mathscr{S}_{\bar{q}}(\mathfrak{S})$ of a punctured oriented surface $\mathfrak{S}$ is a quantum deformation of the coordinate algebra of the ${\rm SL}_3$-character variety of $\mathfrak{S}$. When $\bar{q}$ is a root of unity, we prove the Unicity Theorem for representations of $\mathscr{S}_{\bar{q}}(\mathfrak{S})$, in particular the existence and uniqueness of a generic irreducible representation. Furthermore, we show that the center of $\mathscr{S}_{\bar{q}}(\frak{S})$ is generated by the peripheral skeins around punctures and the central elements contained in the image of the Frobenius homomorphism for $\mathscr{S}_{\bar{q}}(\frak{S})$, a surface generalization of Frobenius homomorphisms of quantum groups related to ${\rm SL}_3$. We compute the rank of $\mathscr{S}_{\bar{q}}(\mathfrak{S})$ over its center, hence the dimension of the generic irreducible representation.
\end{abstract}

\tableofcontents

\newcommand{\ca}{{\cev{a}  }}

\def\BZ{\mathbb Z}
\def\Id{\mathrm{Id}}
\def\Mat{\mathrm{Mat}}
\def\BN{\mathbb N}

\def \cb {\color{blue}}
\def \cred {\color{red}}
\def \cbf {\color{blue}\bf}
\def \credf {\color{red}\bf}
\definecolor{ligreen}{rgb}{0.0, 0.3, 0.0}
\def \cg {\color{ligreen}}
\def \cgf {\color{ligreen}\bf}
\definecolor{darkblue}{rgb}{0.0, 0.0, 0.55}
\def \dbf {\color{darkblue}\bf}
\definecolor{anti-flashwhite}{rgb}{0.55, 0.57, 0.68}
\def \afw {\color{anti-flashwhite}}
\def\cF{\mathcal F}
\def\cP{\mathcal P}
\def\embed{\hookrightarrow}
\def\pr{\mathrm{pr}}
\def\cV{\mathcal V}
\def\ot{\otimes}
\def\buu{{\mathbf u}}


\def \ri {{\rm i}}
\newcommand{\bs}[1]{\boldsymbol{#1}}
\newcommand{\cev}[1]{\reflectbox{\ensuremath{\vec{\reflectbox{\ensuremath{#1}}}}}}
\def\bS{\bar \fS}
\def\cE{\mathcal E}
\def\fB{\mathfrak B}
\def\cR{\mathcal R}
\def\cY{\mathcal Y}
\def\cS{\mathscr S}

\def\fS{\mathfrak{S}}

\def\MN {(M)}
\def\cN {\mathcal{N}}

\newcommand{\beq}{\begin{equation}}
	\newcommand{\eeq}{\end{equation}}

\section{Introduction}

For an oriented $3$-manifold $M$, an {\bf ${\rm SL}_3$-web} $W$ in $M$ is a disjoint union of an oriented link in $M$ and a directed ribbon graph in $M$, such that each vertex is $3$-valent and is a sink or source, where $W$ is equipped with a transversal vector field called its framing (Def. \ref{def-web}).

\vs

Suppose that $R$ is a commutative domain and
$$
\bar{q} = q^{1/3} \in R
$$
is a distinguished invertible element. The {\bf ${\rm SL}_3$-skein module $\mathscr{S}_{\bar{q}}(M)$ of $M$} \cite{le2021stated} (see also \cite{cautis2014webs,higgins2020triangular,sikora2005skein}) is the $R$-module defined as the quotient of the free $R$-module spanned by the set of all isotopy classes of ${\rm SL}_3$-webs in $M$, by the ${\rm SL}_3$-skein relations \eqref{w.cross}-\eqref{wzh.four}.

When $M$ is given as the thickening $\mathfrak{S} \times [-1,1]$ of an oriented surface $\mathfrak{S}$, we write
$$
\mathscr{S}_{\bar{q}}(\mathfrak{S}) := \mathscr{S}_{\bar{q}}(\mathfrak{S}\times[-1,1]),
$$
which is equipped with a product given by superposition. 
Namely, when $W_1,W_2$ are ${\rm SL}_3$-webs in $\frak{S} \times [-1,1]$ such that $W_1 \subset \mathfrak{S} \times (0,1)$ and $W_2 \subset \mathfrak{S} \times (-1,0)$, the product of the corresponding elements $W_1,W_2 \in \mathscr{S}_{\bar{q}}(\mathfrak{S})$ is given by $W_1 \cdot W_2 := W_1 \cup W_2$. We refer to the algebra $\mathscr{S}_{\bar{q}}(\mathfrak{S})$ the {\bf ${\rm SL}_3$-skein algebra} of $\mathfrak{S}$. It first appeared in the work of Kuperberg \cite{kuperberg} in the theory of invariants of representations of the quantum group $U_q(\mathfrak{sl}_3)$, concerning the case when $\mathfrak{S}$ is a closed disc minus a non-empty finite set of points on the boundary. The version for a general surface $\mathfrak{S}$ was introduced by Sikora in \cite{sikora2005skein}, as a quantization of the ${\rm SL}_3$-character variety of $\mathfrak{S}$
$$
\mathfrak{X}_{{\rm SL}_3(\mathbb{C})}(\mathfrak{S}) = {\rm Hom}(\pi_1(\mathfrak{S}), {\rm SL}_3(\mathbb{C}))//{\rm SL}_3(\mathbb{C}). 
$$

\vs

The ${\rm SL}_3$-skein algebra is a special case of the more general ${\rm SL}_n$-skein algebras $\mathscr{S}^{{\rm SL}_n}_{\bar{q}}(\mathfrak{S})$ studied by Sikora \cite{sikora2005skein}, which form higher rank generalizations of the ${\rm SL}_2$-skein algebra $\mathscr{S}^{{\rm SL}_2}_{\bar{q}}(\mathfrak{S})$, which in turn has been extensively studied in the literature under the name of Kauffman bracket skein algebra \cite{prz,turaev}. The ${\rm SL}_2$-skein algebra is generated by the isotopy classes of framed unoriented links, and no graphs are involved. Although it is expected that many of the results on the ${\rm SL}_2$-skein algebras should generalize to ${\rm SL}_n$-skein algebras, the task of proving statements for ${\rm SL}_n$ becomes much more challenging than ${\rm SL}_2$, due to the presence of graphs. 

\vs

Reviewing the literature on the ${\rm SL}_2$-skein algebra $\mathscr{S}^{{\rm SL}_2}_{\bar{q}}(\mathfrak{S})$, a very important direction of research is on its representation theory. In the perspective of physics this is something natural to consider, as the ${\rm SL}_2$-skein algebra is a quantization of a classical object, namely the ${\rm SL}_2$-character variety $\mathfrak{X}_{{\rm SL}_2(\mathbb{C})}(\mathfrak{S}) = {\rm Hom}(\pi_1(\mathfrak{S}), {\rm SL}_2(\mathbb{C}))//{\rm SL}_2(\mathbb{C})$ of $\mathfrak{S}$, hence its elements should be represented as operators on a Hilbert space. Further, when the quantum parameter $\bar q$ is a root of unity, there exist finite-dimensional representations, some of which come from the Witten-Reshetikhin-Turaev quantum field theory associated to the standard representation of $U_q(\mathfrak{sl}_2)$. Bonahon and Wong's work \cite{bonahon2016representations} initiated the systematic study of finite-dimensional irreducible representations of the ${\rm SL}_2$-skein algebra $\mathscr{S}^{{\rm SL}_2}_{\bar{q}}(\mathfrak{S})$ when $\bar q$ is a root of unity. They assigned to each irreducible representation an invariant called the {\em classical shadow} $r\in \mathfrak{X}_{{\rm SL}_2(\mathbb{C})}(\mathfrak{S})$ determined by the central character of the representation. In a sequel \cite{bonahon2019representations} they conjectured that if $r$ belongs to a Zariski open dense subset of $\mathfrak{X}_{{\rm SL}_2(\mathbb{C})}(\mathfrak{S})$, there is a unique irreducible representation of the ${\rm SL}_2$-skein algebra whose classical shadow is $r$. Frohman, Kania-Bartoszynska and L\^e made a breakthrough in \cite{frohman2019unicity} by proving this conjecture, called the Unicity Conjecture or now the Unicity Theorem for ${\rm SL}_2$-skein algebras. They characterize the center of the ${\rm SL}_2$-skein algebra and show that the ${\rm SL}_2$-skein algebra is finite dimensional over its center, in order to prove the conjecture.

\vs

The goal of the present paper is to establish the ${\rm SL}_3$ version of the results obtained in \cite{bonahon2016representations,frohman2019unicity,frohman2021dimension}, which we now describe.

\vs

Let $N'$
be a positive integer, and let
$$
\mbox{$\bar{\omega} \in \mathbb{C}\setminus\{0\}$ be a  root of unity  such that the order of $\bar\omega^2$ is $N'$}.
$$
We put $R = \mathbb{C}$ and $\bar{q} = \bar{\omega}$ and consider the corresponding ${\rm SL}_3$-skein algebra $\mathscr{S}_{\bar{\omega}}(\mathfrak{S})$ at a root of unity $\bar{\omega}$. In order to prove the sought-for Unicity Theorem for $\mathscr{S}_{\bar{\omega}}(\mathfrak{S})$, it suffices to show that $\mathscr{S}_{\bar{\omega}}(\mathfrak{S})$ is a so-called {\em affine almost Azumaya algebra}, because then one can apply the Unicity Theorem for affine almost Azumaya algebras proved in \cite{frohman2019unicity,korinman2021unicity} (see Thm.\ref{tmmm8.3}). Here, a $\mathbb{C}$-algebra $\mathcal{A}$ is said to be \ul{\em affine almost Azumaya} if
\begin{enumerate}[label={\rm (Az\arabic*)}]
\item\label{Az1} $\mathcal{A}$ is finitely generated as a $\mathbb{C}$-algebra,

\item\label{Az2} $\mathcal{A}$ is a domain,

\item\label{Az3} $\mathcal{A}$ is finitely generated as a module over its center.
\end{enumerate}
Here is our first main result:
\begin{theorem}[Thm.\ref{thm-almost}]
\label{thm.main1}
Suppose that $\mathfrak{S}$ is a punctured surface, i.e. every connected component of $\mathfrak{S}$
 is a closed surface minus a non-empty finite set of points. Then $\mathscr{S}_{\bar{\omega}}(\mathfrak{S})$ is affine almost Azumaya.
\end{theorem}

\vs

For any commutative algebra $Z$, denote by ${\rm MaxSpec}(Z)$ the set of all maximal ideals of $Z$. Suppose $\rho : \mathscr{S}_{\bar{\omega}}(\mathfrak{S}) \to {\rm End}(V)$ is a finite-dimensional irreducible representation of $\mathscr{S}_{\bar{\omega}}(\mathfrak{S})$. For each element $x$ in the center $\mathcal{Z}(\mathscr{S}_{\bar{\omega}}(\mathfrak{S}))$, one can show from the irreducibility of $\rho$ that there is a complex number $r_\rho(x)$ such that $\rho(x) = r_\rho(x) \cdot {\rm Id}_V$. We thus get an algebra homomorphism $r_\rho : \mathcal{Z}(\mathscr{S}_{\bar{\omega}}(\mathfrak{S})) \to \mathbb{C}$, a central character. This in turn determines a point ${\rm ker}(r_\rho)$ in ${\rm MaxSpec}(\mathcal{Z}(\mathscr{S}_{\bar{\omega}}(\mathfrak{S})))$. So, if we denote by ${\rm Irrep}_{\mathscr{S}_{\bar{\omega}}(\mathfrak{S})}$ the set of all finite-dimensional irreducible representations of $\mathscr{S}_{\bar{\omega}}(\mathfrak{S})$ considered up to isomorphism, we obtain a map
$$
\mathbb{X} : {\rm Irrep}_{\mathscr{S}_{\bar{\omega}}(\mathfrak{S})} \to {\rm MaxSpec}(\mathcal{Z}(\mathscr{S}_{\bar{\omega}}(\mathfrak{S}))), \quad \rho \mapsto {\rm ker}(r_\rho),
$$
which we view as one version of an ${\rm SL}_3$ analog of the classical shadow of a finite-dimensional irreducible representation of the ${\rm SL}_2$-skein algebra studied in \cite{bonahon2016representations}.

\vs

From  
Thm.\ref{thm.main1} and the Unicity Theorem for affine almost Azumaya algebras (Thm.\ref{tmmm8.3}) we deduce the following second main result:
\begin{theorem}[Unicity Theorem for ${\rm SL}_3$-skein algebras]
\label{thm.main2}
Let $\mathfrak{S}$ be as in Thm.\ref{thm.main1}. Let $K \in \mathbb{N}$ be the rank of $\mathscr{S}_{\bar{\omega}}(\mathfrak{S})$ over the center $\mathcal{Z}(\mathscr{S}_{\bar{\omega}}(\mathfrak{S}))$, as defined in Def.\ref{def-rank}. 
 Define a subset $U$ of  
${\rm MaxSpec}(\mathcal{Z}(\mathscr{S}_{\bar{\omega}}(\mathfrak{S})))$ such that $I\in U$ if and only if
there exists an irreducible representation $\rho : \mathscr{S}_{\bar{\omega}}(\mathfrak{S}) \to {\rm End}(V)$ with ${\rm dim}_{\mathbb C} V=K^{\frac{1}{2}}$
and $\mathbb{X}(\rho) = I.$
Then the following hold:
\begin{enumerate}[label={\rm (\alph*)}]
\item Any finite-dimensional irreducible representation of $\mathscr{S}_{\bar{\omega}}(\mathfrak{S})$ has dimension at most $K^{\frac{1}{2}},$

\item The classical shadow map $\mathbb{X} : {\rm Irrep}_{\mathscr{S}_{\bar{\omega}}(\mathfrak{S})} \to {\rm MaxSpec}(\mathcal{Z}(\mathscr{S}_{\bar{\omega}}(\mathfrak{S})))$ is surjective.

\item 
The subset $U \subset {\rm MaxSpec}(\mathcal{Z}(\mathscr{S}_{\bar{\omega}}(\mathfrak{S})))$ is a Zariski open dense subset.
For any two irreducible representations $\rho_i : \mathscr{S}_{\bar{\omega}}(\mathfrak{S}) \to {\rm End}(V_i)$, $i=1,2$, of $\mathscr{S}_{\bar{\omega}}(\mathfrak{S})$ satisfying $\mathbb{X}(\rho_1) = \mathbb{X}(\rho_2) \in U$, we have that $\rho_1$ and $\rho_2$ are isomorphic. 

\item Any representation of $\mathscr{S}_{\bar{\omega}}(\mathfrak{S})$ sending $\mathcal{Z}(\mathscr{S}_{\bar{\omega}}(\mathfrak{S}))$ to scalar operators and whose classical shadow lies in $U$ is semi-simple.
\end{enumerate}
\end{theorem}

 Define $$N = \frac{N'}{\gcd(N',3)},\quad \bar\eta
=\bar\omega^{N^2}.$$
Similarly to the ${\rm SL}_2$ case, what plays a crucial role in the proof of Thm.\ref{thm.main1} is the {\em Frobenius homomorphism} (Thm. \ref{Fro-surface})
$$
\mathcal{F} : \mathscr{S}_{\bar{\eta}}(\mathfrak{S}) \to \mathscr{S}_{\bar{\omega}}(\mathfrak{S})
$$
of ${\rm SL}_3$-skein algebras at roots of unity. The ${\rm SL}_3$ Frobenius homomorphism  $\mathcal{F}$ was first defined by Higgins \cite{higgins2020triangular}, and the ${\rm SL}_n$ version 
$$
\mathcal{F}^{{\rm SL}_n} : \mathscr{S}^{{\rm SL}_n}_{\bar{\eta}}(\mathfrak{S}) \to \mathscr{S}^{{\rm SL}_n}_{\bar{\omega}}(\mathfrak{S})
$$
was studied in \cite{HW,HLW,wang2023stated}; see \cite{bloomquist2020chebyshev,bonahon2016representations,frohman2019unicity} for ${\rm SL}_2$ versions. As observed e.g. in \cite{HLW}, the ${\rm SL}_n$ Frobenius homomorphism $\mathcal{F}^{{\rm SL}_n}$ can be interpreted as a natural surface generalization of the Frobenius homomorphism for the quantized coordinate algebras of ${\rm SL}_n$ at roots of unity (see \cite{PW})
$$
F^{{\rm SL}_n} : \mathcal{O}_{\bar{\eta}}({\rm SL}_n) \to \mathcal{O}_{\bar{\omega}}({\rm SL}_n),
$$
which is dual to Lusztig's Frobenius homomorphism for the quantized enveloping algebras of $\mathfrak{sl}_n$ at roots of unity \cite{GL93}.

\vs

It is shown in our upcoming joint work \cite{HLW} with Thang L\^e that the Frobenius homomorphism $\mathcal{F}^{{\rm SL}_n} : \mathscr{S}^{{\rm SL}_n}_{\bar{\eta}}(\mathfrak{S}) \to \mathscr{S}^{{\rm SL}_n}_{\bar{\omega}}(\mathfrak{S})$ for ${\rm SL}_n$-skein algebras enjoys several favorable properties, generalizing the results for ${\rm SL}_2$ proved in \cite{bloomquist2020chebyshev,bonahon2016representations} to ${\rm SL}_n$, and at the same time providing more conceptual proofs already for ${\rm SL}_2$ in terms of representation theory of quantum groups. An example of the nice properties proved in \cite{HLW} says that the image of $\mathcal{F}^{{\rm SL}_n}$ 
 provides central elements in $\mathscr{S}^{{\rm SL}_n}_{\bar{\omega}}(\mathfrak{S})$.
In particular, for ${\rm SL}_3$ we have ( 
Prop.\ref{prop-center} of the present paper)
$$
\im_{\bar\omega} \cF
\subset \mathcal{Z}(\mathscr{S}_{\bar{\omega}}(\mathfrak{S})).
$$
Here $$\im_{\bar\omega} \cF=\begin{cases}
	\cF (\mathscr{S}_{\bar{\eta}}(\mathfrak{S})) & \mbox{if } 3\nmid N'\\
	\cF (\mathscr{S}_{\bar{\eta}}(\mathfrak{S})_3) & \mbox{if } 3\mid N'.
\end{cases},$$ 
where $\mathscr{S}_{\bar{\eta}}(\mathfrak{S})_3$ is a subalgebra of $\mathscr{S}_{\bar{\eta}}(\mathfrak{S})$ defined in \cite{Wan24} (see Definition \ref{def-sualgebra3}), which is related to the `(mod 3) congruence' condition of ${\rm SL}_3$-webs studied in \cite{kim2011sl3} (Remark \ref{rem-congruence}). Similarly to the ${\rm SL}_2$ case \cite{frohman2019unicity}, these central elements of the ${\rm SL}_3$-skein algebra are crucial in our proof of Thm.\ref{thm.main1} and hence of Thm.\ref{thm.main2}, the Unicity Theorem. To make the paper more self-contained, we provide a proof of $\im_{\bar\omega} \cF
\subset \mathcal{Z}(\mathscr{S}_{\bar{\omega}}(\mathfrak{S}))$, Proposition \ref{prop-center}, in \S\ref{sec.independent_proofs},  
which is completely independent of \cite{HLW} for the case when $\gcd(N',6)=1$. 

What makes the ${\rm SL}_3$ case much more difficult than the ${\rm SL}_2$ case is the fact that the description of the image under the Frobenius map $\mathcal{F}: \mathscr{S}_{\bar{\eta}}(\fS) \to \mathscr{S}_{\bar{\omega}}(\fS)$ of a general web $l \in \mathscr{S}_{\bar{\eta}}(\fS)$ is not known. The problem of finding and proving such a description is nontrivial already for the case when $l$ is a framed knot. In this case, as conjectured in \cite{bonahon2023central} and proved in 
\cite{higgins2024miraculous,HLW}, the image $\mathcal{F}(l)$ is given by `threading' $l$ along the ${\rm SL}_3$-analog of Chebyshev polynomials used for ${\rm SL}_2$ (see \S\ref{sec3} and Thm.\ref{Fro-surface}). Nothing explicit is known for the case when $l$ has 3-valent vertices.

One main technique we employ in the present paper is about how to overcome this difficulty, which we now explain. We use the ${\rm SL}_3$ quantum trace map developed in \cite{kim2011sl3} (see \cite{le2023quantum} for ${\rm SL}_n$), which is an embedding of the ${\rm SL}_3$-skein algebra into a quantum torus algebra (Thm. \ref{thm-trace}). One pivotal property of this quantum trace map shown in \cite{kim2011sl3} is about the highest degree term of the image of each of the webs that form a basis of the skein algebra (see \eqref{highest_term_of_tr_X}). These are the so-called {\em non-elliptic webs} $\alpha$, which, 
\begin{enumerate}
    \item when isotoped so that the framing of $\alpha$ is given by the positive direction of the factor $[-1,1]$ of $\fS \times [-1,1]$, its projection onto $\fS$ has no crossing, \and
    
    \item the projection of $\alpha$ onto $\fS$ has no 2-gons or 4-gons bounding contractible regions.
\end{enumerate}
These are known to form a basis of the skein algebra $\mathscr{S}_{\bar{q}}(\fS)$ \cite{frohman20223}. For each ideal triangulation $\lambda$ of a punctured surface $\fS$ (which is a triangulation whose vertices are the punctures), special coordinate systems for the set of all non-elliptic webs are studied by Frohman and Sikora \cite{frohman20223} and by Douglas and Sun \cite{douglas2024tropical, kim2011sl3}. The coordinates are enumerated by the set of vertices $V_\lambda$ of a special quiver $\mathsf{H}_\lambda$ associated to $\lambda$ drawn on the surface (Fig.\ref{quiver}), studied by Fock and Goncharov \cite{fockgoncharov06}; each edge of $\lambda$ gives two vertices, and each triangle of $\lambda$ gives one. The Douglas-Sun coordinate map gives a bijection from the set of all non-elliptic webs and the `Knutson-Tao' cone $\Gamma_\lambda$, which is a submonoid of $\mathbb{N}^{V_\lambda}$.

The ${\rm SL}_3$ quantum trace map ${\rm tr}^X_\lambda$ gives an embedding of the ${\rm SL}_3$-skein algebra $\mathscr{S}_{\bar{q}}(\fS)$ of a triangulable punctured surface $\fS$ into the quantum torus algebra $\mathcal{X}_{\hat{q}}(\fS,\lambda)$ (in \eqref{quantum_torus_algebra}) based on the quiver $\mathsf{H}_\lambda$, whose generators are enumerated by the vertex set $V_\lambda$. The result shown in \cite{kim2011sl3} says that, for a non-elliptic web $\alpha$, the image ${\rm tr}^X_\lambda(\alpha)$ has the unique highest term, which is a Laurent monomial whose degrees are given by the Douglas-Sun coordinates of $\alpha$ with respect to $\lambda$ (see \eqref{highest_term_of_tr_X} of Thm.\ref{thm-trace}). What lets us use this result is the compatibility between the Frobenius map $\mathcal{F} : \mathscr{S}_{\bar\eta}(\fS) \to \mathscr{S}_{\bar\omega}(\fS)$ and the quantum trace maps for $\mathscr{S}_{\bar\eta}(\fS)$ and $\mathscr{S}_{\bar\omega}(\fS)$, shown in \cite{HLW} (Thm.\ref{thm-com-trace} of the present paper); there, it is shown than the Frobenius map translates to a very simple `Frobenius' map between the quantum tori, which sends each generator to its $N$-th power. See \S\ref{sec.independent_proofs} for our proof of this statement, Thm.\ref{thm-com-trace}, independent of \cite{HLW}.  Using these results, we are able to translate a natural $\mathbb{Z}^{V_\lambda}$-filtration of the quantum tori to such a filtration of the skein algebra (\S\ref{sub-trace}). This enables us to deal with the `highest term', or the leading term, of an element of the skein algebra (Def.\ref{def-lt_and_deg}), and hence to use an inductive proof methods. Also, we observe that we do have a control over the leading term of the image $\mathcal{F}(\alpha)$ under the Frobenius map $\mathcal{F}$ of a non-elliptic web $\alpha$ (Lem.\ref{lem-D}).

This idea on the degree filtration and the leading term was used in \cite{frohman2019unicity} for ${\rm SL}_2$. Our arguments can be viewed as forming an ${\rm SL}_3$ analog. However, as seen above, new ingredients must be put in at all steps. In particular, the information about a general web with 3-valent vertices is first captured by the leading-term web (by using the ${\rm SL}_3$ quantum trace map developed in \cite{kim2011sl3}) which is a crossingless non-elliptic web (\cite{frohman20223}), which can still have 3-valent vertices, and then its behavior is encoded by an integer tuple given by the Douglas-Sun coordinates \cite{douglas2024tropical,frohman20223,kim2011sl3}. Through this process, we are able to turn the difficulty of dealing with webs with 3-valent vertices into treatment of integer tuples.

To briefly summarize our proof of Thm.\ref{thm.main1}, the ${\rm SL}_3$ quantum trace embedding almost immediately implies \ref{Az2} (Cor.\ref{cor-domain}), the degree filtation is used to show \ref{Az1} (Cor.\ref{cor-finite}), and by keeping a careful track of the leading term of the images of the Frobenius map we show \ref{Az3} (Prop.\ref{prop-center1}).

\vs

As can be seen in the statement of Thm.\ref{thm.main2}, a precise computation of the rank $K$ of $\mathscr{S}_{\bar{\omega}}(\mathfrak{S})$ over its center $\mathcal{Z}(\mathscr{S}_{\bar{\omega}}(\mathfrak{S}))$ is crucial in understanding the structure of generic irreducible representations of $\mathscr{S}_{\bar{\omega}}(\mathfrak{S})$. So we delve into the study of a full description of the center $\mathcal{Z}(\mathscr{S}_{\bar{\omega}}(\mathfrak{S}))$, as well as calculation of the rank $K$. 

\vs

As already mentioned above, a result proved in \cite{HLW} applied to the ${\rm SL}_3$ setting shows that 
$\im_{\bar\omega}\cF$ is contained
in the center of $\mathscr{S}_{\bar{\omega}}(\mathfrak{S})$. We note that, similarly to the ${\rm SL}_2$ case, these are in fact somewhat unexpected elements of the center, which are present only for the case when the quantum parameter $\bar{q}$ is a root of unity $\bar{\omega}$. More typical generators of the center of the skein algebra $\mathscr{S}_{\bar{q}}(\mathfrak{S})$, which are present for generic parameter $\bar{q}$, are the framed knots obtained by constant-elevation lifts of the peripheral loops around punctures, each of which is a small simple loop in $\mathfrak{S}$ surrounding a puncture of $\mathfrak{S}$ in either direction. Let's call the corresponding elements of $\mathscr{S}_{\bar{q}}(\frak{S})$ {\em peripheral skeins}. We show the following, resolving a conjecture in \cite{HLW}, which is a modification of  a conjecture of Bonahon and Higgins \cite[Conjecture 16]{bonahon2023central} in the case of ${\rm SL}_3$-skein algebras:
\begin{theorem}[center of the ${\rm SL}_3$-skein algebra at a root of unity; Thm.\ref{thm-center}]\label{thm-center-intro} 
Let $\mathfrak{S}$ be as in Thm.\ref{thm.main1}. Then the center $\mathcal{Z}(\mathscr{S}_{\bar{\omega}}(\mathfrak{S}))$ is generated by the peripheral skeins and 
 $\im_{\bar\omega}\cF$.
\end{theorem}

We take advantage of the main technical tools developed above, regarding the degree filtration and the leading terms. Using induction and investigation on the leading term degrees, we boil down our proof to linear algebraic treatment on the Douglas-Sun coordinates of non-elliptic webs, which we develop in \S\ref{sec.center}.
As a consequence, we also obtain an explicit basis of the center $\mathcal{Z}(\mathscr{S}_{\bar{\omega}}(\mathfrak{S}))$ (see Lem.\ref{lem-cd}).  

\vs

Building on this, by using the technique developed in \cite{frohman2021dimension} we arrive at the following computation of the rank:
\begin{theorem}[rank of the skein algebra at a root of unity over its center; Thm.\ref{thm.rank}]\label{the.intro.rank}
Suppose that $\mathfrak{S}$ is a connected punctured surface of genus $g$ with $n>0$ punctures.
Then the rank of $\mathscr{S}_{\bar{\omega}}(\mathfrak{S})$ over the center $\mathcal{Z}(\mathscr{S}_{\bar{\omega}}(\mathfrak{S}))$ (Def.\ref{def-rank}) equals $K_{\bar{\omega}}$ given by
$$
K_{\bar{\omega}} = \left\{
\begin{array}{ll}
N^{16g - 16 + 6n} & \mbox{if $3 \nmid N'$}, \\
3^{2g} N^{16g - 16 + 6n} & \mbox{if $3 \mid N'$. }
\end{array}
\right.
$$
\end{theorem}

To prove the above Theorem, we show that 
the rank of $\mathscr{S}_{\bar{\omega}}(\mathfrak{S})$ over its center equals the cardinality of a special finite group involving the Knutson-Tao monoid $\Gamma_\lambda$ (Cor.\ref{cor-lower} and eq.\eqref{eq-smaller}), which we further verify to equal $K_{\bar{\omega}}$ (Lem.\ref{lem-equal}).   
The rank in the above Theorem was studied in \cite{frohman2021dimension,korinman2021unicity,yu2023center} for the {\em stated} ${\rm SL}_2$-skein algebra for a surface with empty or non-compact boundary (where `stated' indicates that there are some additional relations regarding the boundary), and in \cite{HW} for the 
stated ${\rm SL}_n$-skein algebra when the surface has non-compact boundary.

\vs

We end the introduction by mentioning a result 
of ours concerning the ${\rm SL}_3$-skein module $\mathscr{S}_{\bar{q}}(M)$ of an oriented $3$-manifold $M$.  When $3\nmid N'$, we show that there is a $\cS_{\bar\eta}(M)$-module structure on $\cS_{\bar\omega}(M)$ for an oriented 3-manifold $M$ (Prop.\ref{prop-module}), 
by using the Frobenius map $\mathscr{S}_{\bar{\eta}}(M) \to \mathscr{S}_{\bar{\omega}}(M)$ for ${\rm SL}_3$-skein modules 
which we establish in Prop.\ref{Prop-Fro-M} (which also recently appeared in \cite{higgins2024miraculous}), and by using its `transparency' property 
which we show in 
Prop.\ref{prop-tran1}--\ref{prop-relation}, which is a 3-manifold counterpart of commutativity (or, centrality).  
We show that  $\cS_{\bar\omega}(M)$ is finitely generated over $\cS_{\bar\eta}(M)$ when $M$ is a compact 3-manifold (Prop.\ref{finite}). As shown in \cite{sikora2001SLn} (Prop.\ref{chara}, Cor.\ref{cor-one}), there is a one-to-one correspondence between $\text{Hom}_{\text{Alg}}(\cS_{\bar\eta}(M),\mathbb C)$ (the set of $\mathbb C$-algebra homomorphisms from $\cS_{\bar\eta}(M)$ to $\mathbb C$) and $\mathfrak{X}_{{\rm SL}_3(\mathbb{C})}(M)$. 
Then any point $\rho\in \mathfrak{X}_{{\rm SL}_3(\mathbb{C})}(M)$ defines a $\cS_{\bar\eta}(M)$-module structure on $\mathbb C$. Motivated by the work 
\cite{frohman2023sliced}, we define the character-reduced
${\rm SL}_3$-skein module 
$$\cS_{\bar\omega}(M)_\rho = \cS_{\bar\omega}(M)\otimes_{\cS_{\bar\eta}(M)}\mathbb C.$$
When $\partial M\neq\emptyset$, the point $\rho$ induces a point $\rho_{\partial}\in \mathfrak{X}_{{\rm SL}_3(\mathbb{C})}(\partial M)$, where
$\rho_\partial$ is given by the composition 
$\pi_1(\partial M)\rightarrow\pi_1(M)\xrightarrow{\rho}{\rm SL}_3(\mathbb C)$. 
We have the following:
\begin{proposition}[Prop.\ref{prop.skein_module_twisted_over_skein_algebra}]
    Suppose that $M$ is a compact 3-manifold with $\partial M\neq\emptyset$. For any $\rho\in \mathfrak{X}_{{\rm SL}_3(\mathbb{C})}(M)$, the character-reduced ${\rm SL}_3$-skein module $\cS_{\bar\omega}(M)_\rho$ 
    has the structure of a finite dimensional representation of $\cS_{\bar\omega}(\partial M)$ whose classical shadow (Def.\ref{Def-cl-sh}) 
    coincides with $\rho_\partial.$
\end{proposition}

The classical shadow in  
Def.\ref{Def-cl-sh} is the
${\rm SL}_3$ version for the classical shadow defined by
Bonahon and Wong  \cite{bonahon2016representations}.
Note that $\mathfrak{X}_{{\rm SL}_3(\mathbb{C})}(\partial M)
=\text{MaxSpec}(\mathcal Z(\cS_{\bar\omega}(\partial M)))$ holds 
if $\mathcal Z(\cS_{\bar\omega}(\partial M))=\im\cF,$ which is conjectured to be true.
Note that, in this paper, we 
studied the center of the ${\rm SL}_3$-skein algebra $\cS_{\bar\omega}(\fS)$ only when each connected component of the surface $\fS$ contains punctures, since there are no known
 coordinates for ${\rm SL}_3$-webs in closed surfaces.

\vs

We leave the problem of generalizing our main results (Thm.\ref{thm.main1}--\ref{the.intro.rank}) to the case when $\mathfrak{S}$ is a closed surface as topics of future investigation. In particular, it would be very interesting to find a connection to the work of Ganev, Jordan and Safronov \cite{GJS24}, who proved a version of the Unicity Theorem for the `quantized character variety'.
 
\vs

It is still an open question whether there exists a `good basis' for the  ${\rm SL}_n$-skein algebra with a coordinate map like the one introduced in \S\ref{subsec.coordinate_for_webs}. Once we have these tools for  the  ${\rm SL}_n$-skein algebra, the techniques used in this paper could be applied to formulate the center and the rank of the  ${\rm SL}_n$-skein algebra at roots of unity.

\vs

\def\SL{{\rm SL}_3}

{\bf Acknowledgements:}
The research of the second author is supported by the NTU research scholarship from the Nanyang Technological University (Singapore) and the
PhD scholarship from the University of Groningen (The Netherlands). H.K. (the first author) has been supported by KIAS Individual Grant (MG047204) at Korea Institute for Advanced Study. We would like to thank Daniel C Douglas,  Thang TQ  L{\^e}, Zhe Sun, and Linhui Shen for helpful discussions.

\section{The $\SL$-skein theory}\label{sec2}

We will use $\mathbb Z$ and $\mathbb N$ to denote the set of integers and the set of non-negative integers respectively.
The 3-manifolds and surfaces mentioned in this paper are assumed to be oriented.

In this section, we review $\SL$-skein modules and algebras in \cite{le2021stated}. 

\subsection{$\SL$-skein modules}\label{sub111}

\begin{definition}\label{def-web}
	A {\bf web} $l$ in a 3-manifold $M$ is a disjoint union of oriented closed paths and a directed finite  graph embedded into $M$. We also have the following requirements:
	\begin{enumerate}[label={\rm (W\arabic*)}]\itemsep0,3em
	\item $l$ only contains  $3$-valent vertices. 
 
 \item Each $3$-valent vertex is a source or a  sink.  
	
	\item Every edge $e$ of the graph is an embedded oriented  closed interval  in $M$.

	\item $l$ is equipped with a transversal framing. 
	
	\item The set of half-edges at each $3$-valent vertex is equipped with a  cyclic order. 
	\end{enumerate}
\end{definition}

The emptyset $\emptyset$ is also considered as a web in $M$. We have the convention that $\emptyset$ is isotopic only to itself.

If each component of a web is isomorphic to a circle, we say that it consists of knots.
If a web consists of a circle, we call it a {\bf framed knot}.

Our ground ring is a commutative domain $R$ with an invertible element $\hat q$. We set $\bar q =\hat q^{6}$ and $q=\hat q^{18}$ such that $\bar q^{\frac{1}{6}}=\hat q$ and $q^{\frac{1}{18}} = \hat q$. 

We will  use $S_3$ to denote the permutation group on the set $\{1,2,3\}$.

\def\M {M,\cN}
\def\fS{\mathfrak{S}}

The $\SL$-skein module of $M$, denoted as $\cS_{\bar q}(M)$, is
the quotient module of the $R$-module freely generated by the set 
of all isotopy classes of 
webs in $M$ subject to  relations \eqref{w.cross}-\eqref{wzh.four}.

\beq\label{w.cross}
q^{-\frac{1}{3}} 
\raisebox{-.20in}{
	
	\begin{tikzpicture}
		\tikzset{->-/.style=
			
			{decoration={markings,mark=at position #1 with
					
					{\arrow{latex}}},postaction={decorate}}}
		\filldraw[draw=white,fill=gray!20] (-0,-0.2) rectangle (1, 1.2);
		\draw [line width =1pt,decoration={markings, mark=at position 0.5 with {\arrow{>}}},postaction={decorate}](0.6,0.6)--(1,1);
		\draw [line width =1pt,decoration={markings, mark=at position 0.5 with {\arrow{>}}},postaction={decorate}](0.6,0.4)--(1,0);
		\draw[line width =1pt] (0,0)--(0.4,0.4);
		\draw[line width =1pt] (0,1)--(0.4,0.6);
		\draw[line width =1pt] (0.4,0.6)--(0.6,0.4);
	\end{tikzpicture}
}
- q^{\frac {1}{3}}
\raisebox{-.20in}{
	\begin{tikzpicture}
		\tikzset{->-/.style=
			
			{decoration={markings,mark=at position #1 with
					
					{\arrow{latex}}},postaction={decorate}}}
		\filldraw[draw=white,fill=gray!20] (-0,-0.2) rectangle (1, 1.2);
		\draw [line width =1pt,decoration={markings, mark=at position 0.5 with {\arrow{>}}},postaction={decorate}](0.6,0.6)--(1,1);
		\draw [line width =1pt,decoration={markings, mark=at position 0.5 with {\arrow{>}}},postaction={decorate}](0.6,0.4)--(1,0);
		\draw[line width =1pt] (0,0)--(0.4,0.4);
		\draw[line width =1pt] (0,1)--(0.4,0.6);
		\draw[line width =1pt] (0.6,0.6)--(0.4,0.4);
	\end{tikzpicture}
}
= (q^{-1}-q)
\raisebox{-.20in}{
	
	\begin{tikzpicture}
		\tikzset{->-/.style=
			
			{decoration={markings,mark=at position #1 with
					
					{\arrow{latex}}},postaction={decorate}}}
		\filldraw[draw=white,fill=gray!20] (-0,-0.2) rectangle (1, 1.2);
		\draw [line width =1pt,decoration={markings, mark=at position 0.5 with {\arrow{>}}},postaction={decorate}](0,0.8)--(1,0.8);
		\draw [line width =1pt,decoration={markings, mark=at position 0.5 with {\arrow{>}}},postaction={decorate}](0,0.2)--(1,0.2);
	\end{tikzpicture}
},
\eeq 
\beq\label{w.twist}
\raisebox{-.15in}{
	\begin{tikzpicture}
		\tikzset{->-/.style=
			{decoration={markings,mark=at position #1 with
					{\arrow{latex}}},postaction={decorate}}}
		\filldraw[draw=white,fill=gray!20] (-1,-0.35) rectangle (0.6, 0.65);
		\draw [line width =1pt,decoration={markings, mark=at position 0.5 with {\arrow{>}}},postaction={decorate}](-1,0)--(-0.25,0);
		\draw [color = black, line width =1pt](0,0)--(0.6,0);
		\draw [color = black, line width =1pt] (0.166 ,0.08) arc (-37:270:0.2);
\end{tikzpicture}}
= q^{-\frac{8}{3}}
\raisebox{-.15in}{
	\begin{tikzpicture}
		\tikzset{->-/.style=
			{decoration={markings,mark=at position #1 with
					{\arrow{latex}}},postaction={decorate}}}
		\filldraw[draw=white,fill=gray!20] (-1,-0.5) rectangle (0.6, 0.5);
		\draw [line width =1pt,decoration={markings, mark=at position 0.5 with {\arrow{>}}},postaction={decorate}](-1,0)--(-0.25,0);
		\draw [color = black, line width =1pt](-0.25,0)--(0.6,0);
\end{tikzpicture}}
,
\eeq
\beq\label{w.unknot}
\raisebox{-.20in}{
	\begin{tikzpicture}
		\tikzset{->-/.style=
			{decoration={markings,mark=at position #1 with
					{\arrow{latex}}},postaction={decorate}}}
		\filldraw[draw=white,fill=gray!20] (0,0) rectangle (1,1);
		\draw [line width =1pt,decoration={markings, mark=at position 0.5 with {\arrow{>}}},postaction={decorate}](0.45,0.8)--(0.55,0.8);
		\draw[line width =1pt] (0.5 ,0.5) circle (0.3);
\end{tikzpicture}}
= 
(q^2+1+q^{-2}) \ 
\raisebox{-.20in}{
	\begin{tikzpicture}
		\tikzset{->-/.style=
			{decoration={markings,mark=at position #1 with
					{\arrow{latex}}},postaction={decorate}}}
		\filldraw[draw=white,fill=gray!20] (0,0) rectangle (1,1);
\end{tikzpicture}},
\eeq
\beq\label{wzh.four}
\raisebox{-.30in}{
	\begin{tikzpicture}
		\tikzset{->-/.style=
			{decoration={markings,mark=at position #1 with
					{\arrow{latex}}},postaction={decorate}}}
		\filldraw[draw=white,fill=gray!20] (-1,-1) rectangle (1.2,1);
		\draw [line width =1pt,decoration={markings, mark=at position 0.5 with {\arrow{>}}},postaction={decorate}](-1,0.7)--(0,0);
		\draw [line width =1pt,decoration={markings, mark=at position 0.5 with {\arrow{>}}},postaction={decorate}](-1,0)--(0,0);
		\draw [line width =1pt,decoration={markings, mark=at position 0.5 with {\arrow{>}}},postaction={decorate}](-1,-0.7)--(0,0);
		\draw [line width =1pt,decoration={markings, mark=at position 0.5 with {\arrow{<}}},postaction={decorate}](1.2,0.7)  --(0.2,0);
		\draw [line width =1pt,decoration={markings, mark=at position 0.5 with {\arrow{<}}},postaction={decorate}](1.2,0)  --(0.2,0);
		\draw [line width =1pt,decoration={markings, mark=at position 0.5 with {\arrow{<}}},postaction={decorate}](1.2,-0.7)--(0.2,0);
\end{tikzpicture}}=-q^{-3}\cdot \sum_{\sigma\in S_3}
(-q^{\frac{2}3})^{\ell(\sigma)} \raisebox{-.30in}{
	\begin{tikzpicture}
		\tikzset{->-/.style=
			{decoration={markings,mark=at position #1 with
					{\arrow{latex}}},postaction={decorate}}}
		\filldraw[draw=white,fill=gray!20] (-1,-1) rectangle (1.2,1);
		\draw [line width =1pt,decoration={markings, mark=at position 0.5 with {\arrow{>}}},postaction={decorate}](-1,0.7)--(0,0);
		\draw [line width =1pt,decoration={markings, mark=at position 0.5 with {\arrow{>}}},postaction={decorate}](-1,0)--(0,0);
		\draw [line width =1pt,decoration={markings, mark=at position 0.5 with {\arrow{>}}},postaction={decorate}](-1,-0.7)--(0,0);
		\draw [line width =1pt,decoration={markings, mark=at position 0.5 with {\arrow{<}}},postaction={decorate}](1.2,0.7)  --(0.2,0);
		\draw [line width =1pt,decoration={markings, mark=at position 0.5 with {\arrow{<}}},postaction={decorate}](1.2,0)  --(0.2,0);
		\draw [line width =1pt,decoration={markings, mark=at position 0.5 with {\arrow{<}}},postaction={decorate}](1.2,-0.7)--(0.2,0);
		\filldraw[draw=black,fill=gray!20,line width =1pt]  (0.1,0) ellipse (0.4 and 0.7);
		\node  at(0.1,0){$\sigma_{+}$};
\end{tikzpicture}},
\eeq
where the ellipse enclosing $\sigma_+$  is the minimum crossing positive braid representing a permutation $\sigma\in S_3$ and $\ell(\sigma)=\mid\{(i,j)\mid 1\leq i<j\leq 3, \sigma(i)>\sigma(j)\}|$ is the length of $\sigma\in S_3$.
For example, the minimum crossing positive braid corresponding to the element $(1 2 3) \in S_3$ is 
$
\raisebox{-.25in}{
\begin{tikzpicture}
	\fill[gray!20] (0,-0.1) rectangle (1.6,1.7);
	\begin{knot}
		\draw[-o-={0.9}{>}] (0,1.3) -- (1.6,0.3);
		\draw[-o-={0.9}{>}] (0,0.8) -- (1.6,1.3);
		\draw[-o-={0.9}{>}] (0,0.3) -- (1.6,0.8);
		\strand[edge]  (0,1.3) -- (1.6,0.3);
		\strand[edge] (0,0.8) -- (1.6,1.3) (0,0.3) -- (1.6,0.8);
	\end{knot}
\end{tikzpicture}}
$, where at each of the left and the right sides, the endpoints of the strands correspond to $1,2,3$ of the index set $\{1,2,3\}$, read from bottom to top.

Each shaded rectangle in the above relations is the projection of a small open cube embedded in $M$. The lines contained in the shaded rectangle represent parts of webs with framing  pointing to  readers.  For detailed explanation for the above relations, please refer to \cite{le2021stated}.

\def\cq{\cS_{\bar q}}

\subsection{The functoriality}

Let $M$, $M'$ be two   3-manifolds,
and let $f:M'\rightarrow M$ be an orientation preserving embedding.
Then $f$ induces an $R$-linear map $f_{*}: \cq(M')\rightarrow \cq(M)$ such that
$f_*(l) = f(l)$ for any web $l$ in $M'$.

\subsection{$\SL$-skein algebras}\label{subb2.4}

Suppose that $\fS$ is a surface, in the following sense:

\begin{definition}
    \label{def:surface}
    By a {\bf surface} we mean a surface obtained from a not-necessarily-connected oriented compact surface $\overline{\fS}$ possibly with boundary by removing a finite set $\mathcal{P}$ of points. The members of $\mathcal{P}$ are called {\bf punctures} of $\fS$. If $\mathcal{P}$ has a non-empty intersection with each boundary component of $\overline{\fS}$, then $\fS$ is called a {\bf pb surface}. If $\fS$ has empty boundary and if $\mathcal{P}$ has a non-empty intersection with each connected component of $\overline{\fS}$, we say that $\fS$ is a {\bf punctured surface}, which is of primary interest in the present paper. If $\mathcal{P}=\emptyset$, we say that $\fS$ is a closed surface.
\end{definition}

Define 
$$
\widetilde{\fS}=\fS\times [-1,1],
$$
and  $\cq(\fS)=\cq(\widetilde{\fS})$. We will call $\widetilde{\fS}$ {\bf the thickening} of $\fS$, or a {\bf thickened surface}.
Then  $\cq(\fS)$ admits an algebra structure. For any two  webs $\alpha_1$ and $\alpha_2$ in   $\widetilde{\fS}$, we define $\alpha_1\alpha_2\in \cq(\fS)$ to be the result of stacking $\alpha_1$ above $\alpha_2$. We then refer to $\mathscr{S}_{\bar{q}}(\mathfrak{S})$ as the {\bf skein algebra} of the surface $\mathfrak{S}$.

Any web $\alpha$ in $\widetilde{\fS}$ can be represented by a {\bf web diagram} in $\fS$. The diagram is just the projection of $\alpha$ on $\fS$. Before projecting, we isotope  $\alpha$ such that the framing is given by the positive direction of $[-1,1]$ and at each singular point there are two transversal strands with over-crossing or under-crossing information.

\subsection{The $\SL$-skein algebra of the annulus}\label{sub-skeinalgebra}

We use $\bigodot$ to denote the annulus, see Figure \ref{fig1}. 
There is an element $\alpha_1\in \cq(\bigodot)$, see Figure \ref{fig1}. 
Define $\alpha_2\in\cq(\bigodot)$ to be obtained from $\alpha_1$ by reversing its orientation. 

\begin{figure}[h]
	\centering
	\includegraphics[width=15cm]{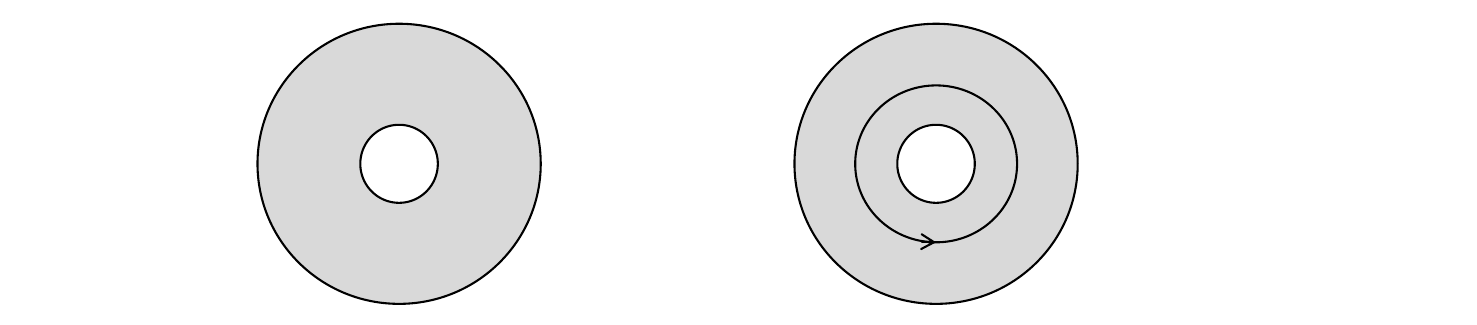}
	\caption{The left picture is $\bigodot$. The right picture is the web diagram $\alpha_1$ in $\bigodot$.}\label{fig1}
\end{figure}

The following Lemma should be well-known. However, we did not find a reference explicitly stating it.
So we give proof for the reader's convenience.
\begin{lemma}\label{alphak}
    We have $\cq(\bigodot)\cong R[\alpha_1,\alpha_2]$. 
\end{lemma}
\begin{proof}
	\cite[Theorem 2]{frohman20223} implies $\{\alpha_1^{k_1}\alpha_2^{k_2}\mid k_1,k_2\in\mathbb N\}$ is a basis of $\cq(\bigodot)$.
	Then there is an $R$-linear isomorphism from $R[\alpha_1,\alpha_2]$ to 
	$\cq(\bigodot)$ given by $\alpha_1^{k_1}\alpha_2^{k_2}\mapsto \alpha_1^{k_1}\alpha_2^{k_2}$ for $k_1,k_2\in\mathbb N$.
	Obviously, it also preserves the multiplication structure. 
\end{proof}							
											
\begin{remark}
   Please refer to 
    \cite{cremaschi2024monomial,HLW,queffelec2018sutured}
    for related works for Lemma \ref{alphak} for general ${\rm SL}_n$-skein theory. These works require $q$ 
    to be not a root of unity of small orders.
\end{remark}

\section{Threading knots using polynomials}\label{sec3}

In this section, we will introduce how to thread knots along a polynomial in two variables. The definition of this threading comes from \cite{bonahon2023central,bonahon2016representations,HLW}.

Suppose that $\gamma$ is a framed knot in a 3-manifold $M$ and $m\in\mathbb N$, 
and define $\gamma^{(m)}$ to be obtained from $\gamma$ by taking $m$ parallel copies of $\gamma$ in the framing direction. We use the convention that $\gamma^{(0)} = \emptyset$. 
Recall that  $\cev\gamma$ is obtained from $\gamma$ by reversing its orientation.

\subsection{Threading knots using polynomials}
\label{subsec.threading}

\def\SS{\cS_{\bar q}(M)}

Let $i_1,i_2$ be two non-negative integers. We use $\gamma^{(i_1,i_2)}$ to denote the 
web in $M$ obtained from $\gamma$ by the following procedures: we first take $i_1 +i_2$ parallel copies of $\gamma$ in the framing direction, and we label these copies from $1$ to $i_1 +i_2$ such that we meet them consecutively from the  one labeled by $i_1+
i_2$  to the one labeled by $1$ if we go in the framing direction;
 then we replace  
 the copies labeled from $i_1+1$ to $i_1+i_2$ with $i_2$ copies of $\cev\gamma$. Note that $\gamma^{(i_1,0)} = \gamma^{(i_1)}$ and $\gamma^{(0,i_2)}=\cev\gamma^{(i_2)}$.

Let $P =\sum_{i_1,i_2} c_{i_1,i_2}x_1^{i_1}x_2^{i_2} \in R[x_1,x_2]$. Define 
$$\gamma^{[P]}=\sum_{i_1,\cdots,i_{n-1}} c_{i_1,i_2}\gamma^{(i_1,i_2)}\in \SS.$$
Let $l = \cup_{1\leq t\leq m}l_t$ be a web in $M$ such that each
$l_t$  is a framed knot, and let $P_t = \sum_{i_{t1},i_{t2}}c_{i_{t1},i_{t2}} x_1^{i_{t1}}x_2^{i_{t2}}\in R[x_1,x_2]$ for each $1\leq t\leq m$. 
Define 
\begin{equation}\label{eq-threading}
	\cup_{1\leq t\leq m}l_t^{[P_t]}=\sum_{i_{11},i_{12},\cdots,i_{m1},i_{m2}}c_{i_{11},i_{12}}\cdots c_{i_{m1},i_{m2}} l_1^{(i_{11},i_{12})}\cup\cdots\cup l_m^{(i_{m1},i_{m2})}\in \SS.
\end{equation}
If $T$ is another web in $M$ such that $T\cap l=\emptyset$, then define 
\begin{equation}\label{eq-action}
	(\cup_{1\leq t\leq m}l_t^{[P_t]})\cup T=\sum_{i_{11},i_{12},\cdots,i_{m1},i_{m2}}c_{i_{11},i_{12}}\cdots c_{i_{m1},i_{m2}} l_1^{(i_{11},i_{12})}\cup\cdots\cup l_m^{(i_{m1},i_{m2})}\cup T\in \SS.
\end{equation}

Let $l = \cup_{1\leq t\leq m}l_t$ be a web in $M$ such that each
$l_t$  is a framed knot, and let $i_1,\cdots,i_{m}\in\{1,2\}$.
Define $l_{(i_1,\cdots,i_m)}$ to be the $n$-web obtained from $l$ by replacing  each
$l_t$ with $\cev{l_t}$ if $i_t=2$, for  $1\leq t\leq m$.

\subsection{${\rm SL}_3$-Chebyshev polynomials in two variables}\label{sub-Chebyshev}

The definition of ${\rm SL}_3$-Chebyshev polynomials in two variables comes from \cite{bonahon2023central,HLW}. 
One way to define them is as follows: for each $m\in \mathbb{N}$, let $P_{m,1}(x_1,x_2)$, $P_{m,2}(x_1,x_2) \in \mathbb{Z}[x_1,x_2]$ be polynomials satisfying
$$
{\rm tr}(A^m) = P_{m,1}({\rm tr}(A), {\rm tr}(A^{-1})), \quad
{\rm tr}(A^{-m}) = P_{m,2}({\rm tr}(A), {\rm tr}(A^{-1})), \quad \forall A\in {\rm SL}_3(\mathbb{C}).
$$
It is well known and easy to check that such polynomials are unique, and can also be characterized by the recurrence relation
$$
P_{m,1} = x_1 P_{m-1,1}-x_2 P_{m-2,1} + P_{m-3,1}, \qquad \forall m \ge 3,
$$
the first three members,
$$
P_{0,1}(x_1,x_2) = 3, \quad P_{1,1}(x_1,x_2) = x_1, \quad P_{2,1}(x_1,x_2) = x_1^2-2x_2,
$$
and the symmetry relation
$$
P_{m,2}(x_1,x_2) = P_{m,1}(x_2,x_1).
$$
For the purpose of the present paper, we will only need $P_{m,1}$, but not $P_{m,2}$.

\def\tr{\text{tr}}

\begin{lemma}\label{lem-additivity}
    Suppose that $n,m$ are positive integers. Then we have 
    $$P_{n,i}(P_{m,1}(x_1,x_2),P_{m,2}(x_1,x_2)) = P_{n+m,i}(x_1,x_2)$$
    for $i=1,2$.
\end{lemma}
\begin{proof}
    We prove the case when $i=1$.
    Recall that, for each $k\in\mathbb N$, $P_{k,1}(x_1,x_2)\in\mathbb Z[x_1,x_2]$ is the unique polynomial such that $P_{k,1}(\tr(A),\tr(A^{-1})) = \tr(A^k)$
    for all $A\in\SL(\mathbb C)$.

    For any $A\in \SL(\mathbb C)$, we have $A^m\in \SL(\mathbb C)$. Then
    \begin{align*}
        P_{n,1}(P_{m,1}(\tr(A),\tr(A^{-1})),P_{m,2}(\tr(A),\tr(A^{-1})))
        =P_{n,1}(\tr(A^m),\tr(A^{-m})) = \tr(A^{n+m}).
    \end{align*}
    This implies that $P_{n,1}(P_{m,1}(x_1,x_2),P_{m,2}(x_1,x_2)) = P_{n+m,1}(x_1,x_2)$.
\end{proof}

\section{The Frobenius map for $\SL$-skein modules}\label{sec-Fro}

When $R$ is the complex field $\mathbb C$ and $\bar q$ is a nonzero complex $\bar\omega$ with  $q=\omega = \bar\omega^{3}$, we will use $\cS_{\bar \omega}(M)$ to denote 
$\cS_{\bar q}(M)$.

In the following of this section, we will assume
$R=\mathbb C$, $\bar q=\bar\omega\in\mathbb C$ is a root of unity. We have  $\omega=\bar\omega^{3}$ such that  $\omega^{\frac{1}{3}}=\bar\omega$. 
$$
\mbox{Suppose that the order of $\omega^2$ is $N$.
Set $\bar\eta = \bar\omega^{N^2}$.}
$$
Then we have  
 $\eta=\bar \eta^{3}=\omega^{N^2} = \pm 1$.

We first review the Frobenius map for the $\SL$-skein algebras of surfaces, which is a special case for a more generalized statement in Theorem \ref{AP-thm-Fro}. 

\begin{theorem}\cite{higgins2020triangular, higgins2024miraculous,HLW}\label{Fro-surface}
	Suppose $\fS$ is a surface without boundary (Definition \ref{def:surface}). There exists a unique algebra homomorphism 
 $$
 \cF\colon \cS_{\bar\eta}(\fS)\rightarrow \cS_{\bar\omega}(\fS),
 $$
called the Frobenius map, with the following properties:
	\begin{enumerate}[label={\rm (\alph*)}]\itemsep0,3em
	\item $\cF$ is injective if  $\fS$ is a punctured surface. 
	
	\item\label{Fro-surface-b} Let $l = \cup_{1\leq t\leq m}l_t$ be a web in the thickened surface $\widetilde{\fS}=\fS \times [-1,1]$ such that each
	$l_t$  is a framed knot. Then 
 $$\mathcal{F}(l) = \cup_{1\le t \le m} l_t^{[P_{N,1}]},$$
 where the right hand side is given by the threading operation in \S\ref{subsec.threading}.
 \end{enumerate}
\end{theorem}

Note that, for any framed knot $l$ in the thickened surface $\widetilde{\fS}=\fS \times [-1,1]$, we have 
$l^{[P_{N,1}]} = \cev{l}^{[P_{N,2}]}$.

\begin{remark}\label{rem-sl3}
For a pb surface (Definition \ref{def:surface}), 
the so-called stated ${\rm SL}_n$-skein algebra is studied \cite{HW,le2021stated,le2023quantum,wang2023stated}, which involves special relations at the boundary; see \S\ref{sec.independent_proofs}.

	The Frobenius map for (stated) ${\rm SL}_2$-skein algebra is constructed in \cite{bloomquist2020chebyshev,bonahon2016representations}.
	The Frobenius map for (stated) ${\rm SL}_3$-skein algebra is constructed in \cite{higgins2020triangular}. Theorem \ref{Fro-surface}\ref{Fro-surface-b} is conjectured in \cite{bonahon2023central}, and is proved in 
 \cite{higgins2024miraculous} under the assumption that the order of $\bar{\omega}$ is coprime to $6$ (we have $\bar\eta=1$ under this assumption, in particular).
 In an upcoming joint paper of the authors and Thang L\^e \cite{HLW}, it is proved for all ${\rm SL}_n$, $n\ge 2$, with no assumption on the order of $\bar{\omega}$.

	The Frobenius map for general stated ${\rm SL}_n$-skein algebra is constructed in \cite{wang2023stated}  
 for a pb surface $\fS$ that is `essentially bordered', i.e. when every connected component of $\fS$ has non-empty boundary. 
 It is generalized to the case when $\fS$ has 
 empty boundary in \cite{HLW}.
\end{remark}

For our later use (in \S\ref{sec-3-manifolds}), we now generalize Theorem \ref{Fro-surface} to $\SL$-skein modules for 3-manifolds.

We use $H_g$ to denote the handlebody of genus $g$. 
Suppose that $M$ is a connected compact 3-manifold.
Then, for some $g$, there 
exists a collection of disjoint closed curves, denoted as $\mathcal C$, on $\partial H_g$ such that  $M$ is obtained from $H_{g}$ by attaching 2-handles along curves in $\mathcal C$ (maybe we need to fill in some sphere components with 3-dimensional balls) \cite{saito2005lecture}. 

 Suppose $\gamma$ is a closed curve in $\mathcal C$ and $D_{\gamma}$ is the 2-handle attached to $\gamma$. Let $l$ be a web in $H_g$. We can slide $l$ along 
the $2$-handle $D_\gamma$ attached to $\gamma$ 
a new web $
\rm sl_\gamma(l)$ in $H_g$. The relation
$$l=
\rm sl_\gamma(l),$$
or the corresponding difference $l - {\rm sl}_\gamma(l) \in \cS_{\bar q}(H_g)$,
is called the handle sliding relation along curve $\gamma$.

There is a natural embedding $L:H_{g}\rightarrow M$. It induces a linear map 
$$
L_{*} =L_{*,\bar{q}} :\cS_{\bar q}(H_g)\rightarrow \cS_{\bar q}(M).
$$

\begin{lemma}\cite[Proposition 2.2]{przytycki1998fundamentals}\label{lem-handle}
 The map $L_{*}:\cS_{\bar q}(H_g)\rightarrow \cS_{\bar q}(M)$ is surjective, and
	$
 \ker L_{*}$ is generated by the handle sliding relations $l - {\rm sl}_\gamma(l)$ 
 for webs $l$ in $H_g$ and curves $\gamma$ in $\mathcal C$.  
\end{lemma}

Since $H_{g}$ is isomorphic to the thickening of a surface (with non-empty boundary), we have the Frobenius map $\cF:\cS_{\bar \eta}(H_{g})\rightarrow \cS_{\bar \omega}(H_{g})$, by Theorem \ref{Fro-surface}.

\begin{proposition}\label{Prop-Fro-M}
	Suppose that $M$ is a connected compact 3-manifold.
	There exists a unique linear map $\cF\colon\cS_{\bar\eta}(M)\rightarrow \cS_{\bar\omega}(M)$ with the following properties:
	Let $l = \cup_{1\leq t\leq m}l_t$ be a web in $M$ such that each
	$l_t$  is a framed knot. Then 
	\begin{equation}\label{eq-Fro-image}
		\cF(l) = \cup_{1\leq t\leq m} l_t^{[P_{N,1}]}.
	\end{equation}
\end{proposition}
\begin{proof}
 Consider the following diagram:
	\begin{equation}\label{eq-Fro}
		\begin{tikzcd}
			\cS_{\bar\eta}(H_{g})  \arrow[r, "\cF"]
			\arrow[d, "L_{*,\bar\eta}"]  
			&  \cS_{\bar\omega}(H_{g}) \arrow[d, "L_{*,\bar\omega}"] \\
			\cS_{\bar\eta}(M)\arrow[r,dotted, "\cF"] 
			& \cS_{\bar\omega}(M) \\
		\end{tikzcd}
	\end{equation}
	To show that the diagram in equation \eqref{eq-Fro} induces a map $\cF:\cS_{\bar\eta}(M)\rightarrow \cS_{\bar\omega}(M)$, it suffices to show 
	$L_{*,
 \bar{\omega}}\circ \cF$ preserves the handle sliding relations  along curves in $\mathcal C$. 
 Suppose that $l,l'$ are webs in $H_{g}$ such that $l'$ is obtained from $l$ by
 performing the handle sliding along curves in $\mathcal C$.

	Because of relation \eqref{wzh.four}, we can assume that
	both $l$ and $l'$ consist of knots.  
	Suppose that $l = \cup_{1\leq t\leq m}l_t$ and $l' = \cup_{1\leq t\leq m}l_t'$, where $l_t$ and $l_t'$ are framed knots for $1\leq t\leq m$. We can suppose that $l$ 
 differs from $l'$ 
 by exactly 
 one knot. So we can suppose that both $l$ and $l'$ are framed knots and $l'$ is obtained from $l$ by 
 performing the handle sliding along curves in $\mathcal C$.
	Then $L_{*,\bar\omega}(\cF(l)) = L_{*,\bar\omega}(l^{[P_{N,1}]}) = L_{*,\bar{\omega}}(l)^{[P_{N,1}]}$ and $L_{*,\bar\omega}(\cF(l')) =  L_{*,\bar{\omega}}(l')^{[P_{N,1}]}$.
	Clearly $L_{*,\bar{\omega}}(l')^{[P_{N,1}]}$ is obtained from $L_{*,\bar{\omega}}(l)^{[P_{N,1}]}$ by 
 performing the handle sliding along curves in $\mathcal C$. Thus we have 
	$L_{*,\bar\omega}(\cF(l))=L_{*,\bar\omega}(\cF(l'))\in \cS_{\bar\omega}(M)$. Then there exists a linear map $\cF:\cS_{\bar\eta}(M)\rightarrow \cS_{\bar\omega}(M)$ such that the diagram in equation \eqref{eq-Fro} commutes. It is unique because 
	$L_{*,
 \bar{\eta}}:\cS_{\bar\eta}(H_{g})\rightarrow \cS_{\bar\eta}(M)$ is surjective.
	
	Theorem \ref{Fro-surface} (b) implies equation \eqref{eq-Fro-image}.
\end{proof}

\begin{remark}
    After we presented the above proof of Proposition \ref{Prop-Fro-M} in our first arXiv version, Higgins also proved it in \cite{higgins2024miraculous}, with a restriction on the order $N$, and using different techniques.
\end{remark}

\begin{remark}
	The construction in Proposition \ref{Prop-Fro-M} can be easily generalized to stated ${\rm SL}_n$-skein modules using the Frobenius map for the stated ${\rm SL}_n$-skein algebra in \cite{HLW}. 
\end{remark}

\section{The Unicity Theorem for $\SL$-skein algebras}\label{sec5}

In this section,
we will prove the Unicity Theorem  for $\SL$-skein algebras when $R=\mathbb C$ and $\hat q\in\mathbb C$ is a root of unity.
The Unicity Theorem classifies the irreducible representations of $\cS_{\bar q}(\fS)$ when $\bar q$ is a root of unity. It says that every point in a Zariski open dense subset of $\text{MaxSpec}(\mathcal Z(\cS_{\bar q}(\fS)))$ uniquely determines an irreducible representation of 
$\cS_{\bar q}(\fS)$ with a fixed dimension, which equals the square root of the $K$ in Theorem \ref{thm.main2}.
This $K$ depends on $\fS$ and the order of $\bar q^2$. We will precisely formulate this $K$ in section \ref{sec-rank}.
The work in this section generalizes Frohman,  Kania-Bartoszynska, and L{\^e}'s work 
on the Unicity Theorem of ${\rm SL}_2$-skein algebras \cite{frohman2019unicity} to ${\rm SL}_3$-version.

\subsection{An $R$-module basis for 
the ${\rm SL}_3$-skein algebra
}

If two web diagrams in a surface $\fS$ represent two isotopic webs in the thickened surface $\widetilde{\fS}=\fS\times[-1,1]$, we will consider them as the same web diagram. We will implicitly identify a web diagram in $\mathfrak{S}$ with the corresponding isotopy class of webs in $\widetilde{\fS}$.

A crossingless web diagram without trivial components (contractible loops) nor internal 2- or 4-gons is called {\bf non-elliptic}.
\begin{definition}\label{def-B_S}
Let $B_{\fS}$ be the set of all non-elliptic web diagrams in $\fS$.
\end{definition}
It is well known that $B_{\fS}$ is an $R$-module basis of $\cS_{\bar q}(\fS)$ ; see \cite[Theorem 2]{frohman20223}.

\subsection{Honeycomb and crossbar webs}
\label{subsec.honeycomb_and_crossbar}

We use $D_2$ to denote the 2-dimensional closed disk.
For each 
integer $k\ge 1$, we use $\mathbb P_k$ to denote the pb surface obtained from $D_2$ by removing $k$ points in $\partial D_2$. We call 
{$\mathbb{P}_1$, $\mathbb{P}_2$ and $\mathbb{P}_3$ a {\bf monogon}, a {\bf bigon} and a {\bf triangle}, respectively.

In the literature, the `stated' skein algebra for a pb surface (Definition \ref{def:surface}) with non-empty boundary is studied, which involves certain additional relations regarding the boundary. Large part of the current paper does not deal with such skein algebras. However, we will use webs for $\mathbb{P}_2$ and $\mathbb{P}_3$ as convenient tools. Instead of recalling the definition of a web living in the thickening of a general pb surface that is used in the definition of the stated skein algebra, we deal with webs in $\mathbb{P}_2$ and $\mathbb{P}_3$ using the following concise definition; see \S\ref{sec.independent_proofs} for a proper definition of the stated skein algebra of a general pb surface.

\begin{definition}\label{def-webPk}
Suppose $k=2,3$.
	A {\bf crossingless web} $l$ in $\mathbb P_k$ is a disjoint union of oriented closed paths and a directed finite graph properly embedded into $\mathbb P_k$. We also have the following requirements:
	\begin{enumerate}[label={\rm (\arabic*)}]
	\item $l$ only contains  $1$-valent and $3$-valent vertices. Each $1$-valent vertice is contained in $\partial \mathbb P_k$, and
 each $3$-valent vertex is a source or a sink. 
	
	\item Every edge $e$ of the graph is an embedded oriented closed interval in $\mathbb P_k$.
	\end{enumerate}
\end{definition}

A crossingless web in $\mathbb{P}_2$ or $\mathbb{P}_3$ is usually considered up to isotopy within the class of crossingless webs in $\mathbb{P}_2$ or $\mathbb{P}_3$.

A {\bf crossbar} 
in $\mathbb P_2$ is a 
crossingless web in $\mathbb{P}_2$ 
 that, when the orientations are forgotten, consists of parallel  
lines connecting the two boundary components and 
lines connecting each pair of the adjacent parallel lines. 
A crossbar web is {\bf minimal} if it is non-elliptic; see Figure \ref{crossbar} for an example.

The {\bf honeycomb} of degree $d$ (with $d\in \mathbb{Z}$), denoted as $\mathcal H_d$, is a crossingless web in $\mathbb P_3$. 
We use a picture in Figure \ref{crossbar} to illustrate $\mathcal H_3$ (we will get $\mathcal H_{-3}$ if we reverse the orientation of $\mathcal H_3$). Please refer to \cite[section 10]{frohman20223} (they call the honeycomb as the pyramid) or \cite[subsection 15.3]{le2023quantum} for the detailed definition of the honeycomb. Note that $\mathcal H_0$ is the empty web in $\mathbb P_3$. 

\begin{figure}
	\centering
	\includegraphics[width=12cm]{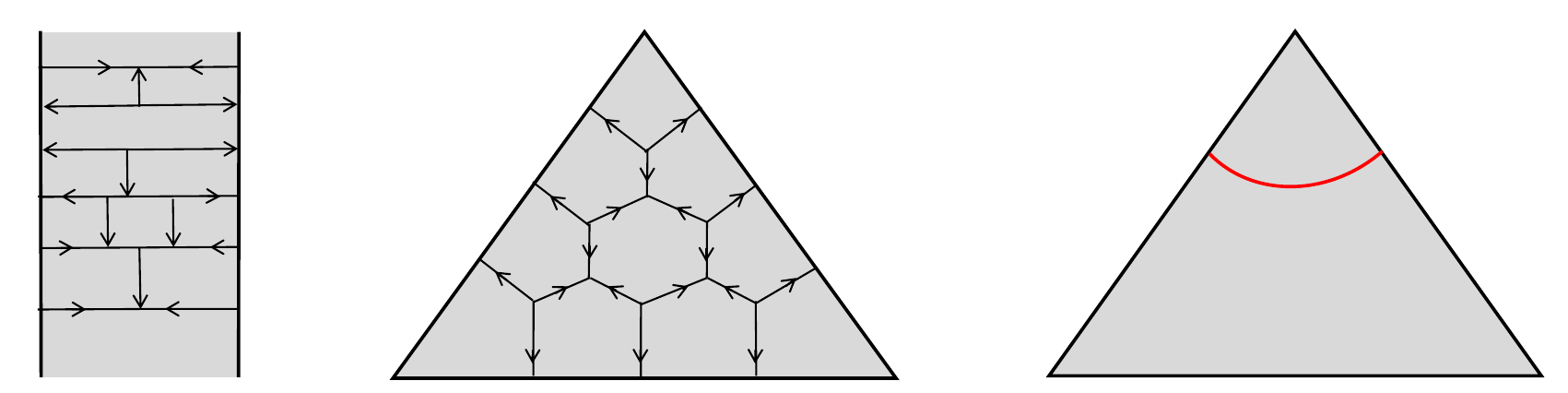}
	\caption{The left picture is an example for  
 a minimal crossbar web in $\mathbb P_2$, the the middle picture is an example for $\mathcal H_3$ in $\mathbb P_3$, the right picture is an example for a corner arc in $\mathbb P_3$ (the orientation of the red line is arbitrary).}\label{crossbar}
\end{figure}

\subsection{The coordinate for webs}
\label{subsec.coordinate_for_webs}

\def\avec{\vec}
\def\face{\mathbb F}
\def\rd{\overline}
\def\al{\alpha}
\def\nats{\mathbb N}

Suppose that $\fS$ is a 
punctured surface (Definition \ref{def:surface}).  
We say that $\fS$ is 
{\bf triangulable} if each component of $\fS$ is not 
a sphere with 
less than three punctures. 
\begin{definition}\label{def.ideal_triangulation}
    An embedding $c:(0,1)\rightarrow \fS$ is called an {\bf ideal arc} if both
$\bar c(0)$ and $\bar c(1)$ are punctures, where $\bar c\colon [0,1] \to \fS$ is the `closure' of $c$. By an ideal arc we often mean its image in $\fS$. 
An {\bf ideal triangulation}, or a {\bf triangulation} $\lambda$ of $\fS$, is a collection of mutually disjoint ideal arcs in $\fS$ with the following properties:
\begin{enumerate}[label={\rm (T\arabic*)}]
    \item any two arcs in $\lambda$ are not isotopic;

    \item $\lambda$ is maximal under condition (T1);

    \item 
    the valence of $\lambda$ at each puncture is at least two.
\end{enumerate} 
\end{definition} 
Our definition of the triangulation excludes self-folded triangles.
We will call each ideal arc in $\lambda$ an {\bf edge} of $\lambda$.
We use $\mathbb F_{\lambda}$ to denote the set of faces of $\lambda$; a member of $\mathbb{F}_\lambda$, which is called an {\bf (ideal) triangle} of $\lambda$, is the closure in $\fS$ of a connected component of $\fS\setminus \bigcup_{e\in \lambda} e$. 
It is well-known that any triangulable surface admits a triangulation. Unless otherwise stated, a triangulation is considered only up to simultaneous isotopy of its members.

In the following of this subsection, we will assume 
$\fS$ is a triangulable surface with any chosen triangulation $\lambda$. One may regard $\lambda$ as being defined only up to isotopy.

For each edge $e$ of $\lambda$, we take two disjoint copies $e'$, $e''$ of $e$ such that there is no intersection among any two edges of $\cup _{e\in\lambda} \{e',e''\}$. We write
\begin{align}
    \label{widehat_lambda}
\widehat{\lambda} := \bigcup _{e\in\lambda} \{e',e''\}
\end{align}
and call it a {\bf split} ideal triangulation of $\lambda$. 
We get a collection of bigons and ideal triangles if we cut $\fS$ along edges in 
$\widehat{\lambda}$.

A element $\beta$ of the non-elliptic web basis $B_{\fS}$ (Definition \ref{def-B_S}) is said to be in a {\bf canonical position} with respect to $\widehat{\lambda}$ if 
\begin{enumerate}[label={\rm (C\arabic*)}]
    \item\label{C1} 
    the intersection of $\beta$ with each bigon of $\widehat{\lambda}$ is a minimal crossbar, and
    
    \item\label{C2} 
    the intersection of $\beta$ with each ideal triangle of $\widehat{\lambda}$ 
    is the disjoint union of 
    corner arcs and a single honeycomb, where a corner arc is a simple arc connecting two distinct sides of the triangle; see Figure \ref{crossbar}.
    
\end{enumerate}

We define a quiver $\mathsf{H}_{\lambda}$ associated to $\lambda$ such that for any (ideal) triangle $\tau\in\mathbb F_\lambda$  the intersection $\mathsf{H}_{\lambda}\cap\tau$ looks like the picture in Figure \ref{quiver}. We use $V_{\lambda}$ to denote the set of vertices of $\mathsf{H}_{\lambda}$. In each triangle $\tau \in \mathbb{F}_\lambda$, six vertices of $\mathsf{H}_\lambda$ lie on edges of $\lambda$ (i.e. on the sides of $\tau$), and one vertex of $\mathsf{H}_\lambda$ lie in the interior of the triangle $\tau$. The vertex of $\mathsf{H}_\lambda$ lying in the interior of $\tau$ will be denoted by $v_\tau$.

\begin{figure}
	\centering
\begingroup%
  \makeatletter%
  \providecommand\color[2][]{%
    \errmessage{(Inkscape) Color is used for the text in Inkscape, but the package 'color.sty' is not loaded}%
    \renewcommand\color[2][]{}%
  }%
  \providecommand\transparent[1]{%
    \errmessage{(Inkscape) Transparency is used (non-zero) for the text in Inkscape, but the package 'transparent.sty' is not loaded}%
    \renewcommand\transparent[1]{}%
  }%
  \providecommand\rotatebox[2]{#2}%
  \newcommand*\fsize{\dimexpr\f@size pt\relax}%
  \newcommand*\lineheight[1]{\fontsize{\fsize}{#1\fsize}\selectfont}%
  \ifx\svgwidth\undefined%
    \setlength{\unitlength}{113.38582677bp}%
    \ifx\svgscale\undefined%
      \relax%
    \else%
      \setlength{\unitlength}{\unitlength * \real{\svgscale}}%
    \fi%
  \else%
    \setlength{\unitlength}{\svgwidth}%
  \fi%
  \global\let\svgwidth\undefined%
  \global\let\svgscale\undefined%
  \makeatother%
  \begin{picture}(1,0.95)%
    \lineheight{1}%
    \setlength\tabcolsep{0pt}%
    \put(0,0){\includegraphics[width=\unitlength,page=1]{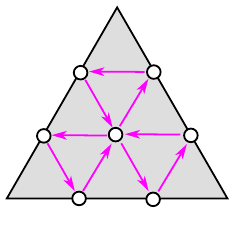}}%
    \put(0.19368421,0.68795669){\color[rgb]{0,0,0}\makebox(0,0)[lt]{\lineheight{1.25}\smash{\begin{tabular}[t]{l}$v_{22}$\end{tabular}}}}%
    \put(0.68800278,0.68459513){\color[rgb]{0,0,0}\makebox(0,0)[lt]{\lineheight{1.25}\smash{\begin{tabular}[t]{l}$v_{31}$\end{tabular}}}}%
    \put(0.84067144,0.40432272){\color[rgb]{0,0,0}\makebox(0,0)[lt]{\lineheight{1.25}\smash{\begin{tabular}[t]{l}$v_{32}$\end{tabular}}}}%
    \put(0.29835579,0.02151165){\color[rgb]{0,0,0}\makebox(0,0)[lt]{\lineheight{1.25}\smash{\begin{tabular}[t]{l}$v_{12}$\end{tabular}}}}%
    \put(0.60825052,0.01923294){\color[rgb]{0,0,0}\makebox(0,0)[lt]{\lineheight{1.25}\smash{\begin{tabular}[t]{l}$v_{11}$\end{tabular}}}}%
    \put(0.47608953,0.47040319){\color[rgb]{0,0,0}\makebox(0,0)[lt]{\lineheight{1.25}\smash{\begin{tabular}[t]{l}$v$\end{tabular}}}}%
    \put(0.02947656,0.40660133){\color[rgb]{0,0,0}\makebox(0,0)[lt]{\lineheight{1.25}\smash{\begin{tabular}[t]{l}$v_{21}$\end{tabular}}}}%
  \end{picture}%
\endgroup%

	\caption{The quiver in $\mathbb P_3$.}\label{quiver}
\end{figure}
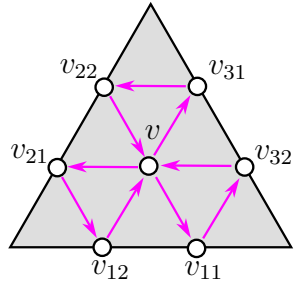

We review
the {\bf Douglas-Sun coordinates} for the elements of $B_{\fS}$ introduced in \cite{douglas2024tropical,kim2011sl3}, which form a coordinate map
\begin{align}
    \label{DS_map}
\kappa : B_{\fS} \to \mathbb{Z}^{V_\lambda}, \qquad \alpha \mapsto (k_v(\alpha))_{v\in V_\lambda}.
\end{align}

Suppose $\alpha\in B_{\fS}$ is in a canonical position. 
For $\tau\in\mathbb F_\lambda$, suppose that the quiver $\mathsf{H}_\lambda$ is as in the picture in Figure \ref{quiver}, where we denote the middle vertex by $v_\tau$. For each $i=1,2,3$, suppose the edge containing $v_{i1}$
and $v_{i2}$ is $e_i$. We define $e_{\text{out},i}(\alpha)$ (resp. $e_{\text{in},i}(\alpha)$) to be the number of occurrences of  
 $
\raisebox{-.15in}{
	\begin{tikzpicture}
		\tikzset{->-/.style=
			{decoration={markings,mark=at position #1 with
					{\arrow{latex}}},postaction={decorate}}}
		\filldraw[draw=white,fill=gray!20] (0,0) rectangle (1, 1);
		\draw [line width =1pt](0,0)--(0,1);
		\draw [line width =1pt,color=red,decoration={markings, mark=at position 0.5 with {\arrow{>}}},postaction={decorate}](1,0.5)--(0,0.5);
\end{tikzpicture}}
$
(resp.  $
\raisebox{-.15in}{
	\begin{tikzpicture}
		\tikzset{->-/.style=
			{decoration={markings,mark=at position #1 with
					{\arrow{latex}}},postaction={decorate}}}
		\filldraw[draw=white,fill=gray!20] (0,0) rectangle (1, 1);
		\draw [line width =1pt](0,0)--(0,1);
		\draw [line width =1pt,color=red,decoration={markings, mark=at position 0.5 with {\arrow{<}}},postaction={decorate}](1,0.5)--(0,0.5);
\end{tikzpicture}}
$). Here the black line is a part of $e_i$, the oriented red line is a part of $\alpha$, and the gray area is a part of $\tau$.
For each $i=1,2,3$, define the edge coordinates by
\begin{align} 
\label{DS_edges}
k_{v_{i1}}(\alpha)= e_{\text{out},i}(\alpha)+ 2e_{\text{in},i}(\alpha)\text{ and }
k_{v_{i2}}(\alpha) = e_{\text{in},i}(\alpha)+ 2e_{\text{out},i}(\alpha).
\end{align}
We use $\tau(\alpha)$ to denote the number of clockwise corner arcs of $\alpha$ in $\tau$, see Figure \ref{clockwise}.
Define the triangle-interior coordinate by
\begin{align}
\label{DS_triangle}
k_{v_\tau}(\alpha)=\left(\sum_{i=1,2,3} (e_{\text{out},i}(\alpha)+ e_{\text{in},i}(\alpha))\right) - \tau(\alpha).
\end{align}

\begin{figure}
	\centering
	\includegraphics[width=5cm]{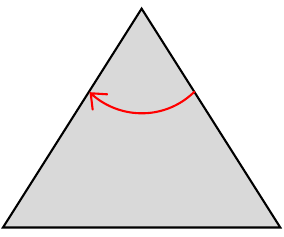}
	\caption{Clockwise corner arc.}\label{clockwise}
\end{figure}

\def\B{B_{\fS}}

The coordinate map $\kappa : B_{\fS} \to \mathbb{Z}^{V_\lambda}, \alpha \mapsto (k_v(\alpha))_{v\in V_\lambda}$ constructed this way is a well-defined injection  \cite{douglas2024tropical}.

\vspace{2mm}

\begin{remark}\label{rem-DS}
    The computation of the coordinates $k_v(\alpha)$ uses only the part of the web $\alpha$ over triangles of the split ideal triangles $\widehat{\lambda}$, and not the part over the bigons.
\end{remark}

\def\s{\text{sum}}

We use $V_{\mathbb P_3}$ to denote the 
vertex set of the quiver in Figure \ref{quiver}, for an ideal triangle $\mathbb{P}_3$. We review the local cone $\Gamma_{\mathbb P_3}\subset\mathbb Z^{V_{\mathbb P_3}}$ defined in \cite{douglas2024tropical} and the balanced part $\mathsf{B}_{\mathbb P_3}\subset \mathbb Z^{V_{\mathbb P_3}}$ introduced in \cite{kim2011sl3,le2023quantum}.
For any $${\bf k} = (k_{v_{11}},k_{v_{12}},k_{v_{21}},k_{v_{22}},k_{v_{31}},k_{v_{32}},k_{v})\in \mathbb Z^{V_{\mathbb P_3}}$$
define (see \cite[\S6.2]{douglas2024tropical})
\begin{equation}\label{eq-kkkk}
\begin{split}
	&r_{12}({\bf k}) = k_v+k_{v_{32}}-k_{v_{11}}-k_{v_{31}},\quad 
	r_{11}({\bf k}) = k_{v_{22}} + k_{v_{31}} - k_v,\\
	&r_{13}({\bf k}) = k_v+k_{v_{21}}-k_{v_{12}}-k_{v_{22}};\\
	&r_{22}({\bf k}) = k_v+k_{v_{12}}-k_{v_{21}}-k_{v_{11}},\quad 
	r_{21}({\bf k}) = k_{v_{32}} + k_{v_{11}} - k_v,\\
	&r_{23}({\bf k}) = k_v+k_{v_{31}}-k_{v_{22}}-k_{v_{32}};\\
	&r_{32}({\bf k}) = k_v+k_{v_{22}}-k_{v_{31}}-k_{v_{21}},\quad 
	r_{31}({\bf k}) = k_{v_{12}} + k_{v_{21}} - k_v,\\
	&r_{33}({\bf k}) = k_v+k_{v_{11}}-k_{v_{32}}-k_{v_{12}}.\\
 \end{split}
\end{equation}
Set $$r({\bf k}) = (r_{11}({\bf k}),r_{12}({\bf k}),r_{13}({\bf k}),r_{21}({\bf k}),r_{22}({\bf k}),r_{23}({\bf k}),r_{31}({\bf k}),r_{32}({\bf k}),r_{33}({\bf k})).$$
Define
$$\Gamma_{\mathbb P_3}=\{{\bf k}\in \mathbb Z^{V_{\mathbb P_3}}\mid 
r({\bf k})\in (3\mathbb N)^9\}\text{ and }\mathsf{B}_{\mathbb P_3}=\{{\bf k}\in \mathbb Z^{V_{\mathbb P_3}}\mid 
r({\bf k})\in (3\mathbb Z)^9\}.$$

\begin{definition}\label{def-Gamma_lambda_and_B_lambda}
Suppose that $\fS$ is a triangulable punctured surface with a triangulation $\lambda$.
For ${\bf k}\in \mathbb Z^{V_{\lambda}}$ and $\tau\in \mathbb F_\lambda$, we use 
${\bf k}|_{\tau}$ to denote the restriction of ${\bf k}$ on $V_{\tau}$.
Define
\begin{align*}
	\Gamma_{\lambda} & =\{{\bf k}\in \mathbb Z^{V_{\lambda}}  
 ~:~
	{\bf k}|_{\tau}\in \Gamma_{\tau}\text{ for any }\tau\in\mathbb F_\lambda\}, \\ 
	\mathsf B_{\lambda} & =\{{\bf k}\in \mathbb Z^{V_{\lambda}}  
 ~:~
	{\bf k}|_{\tau}\in \mathsf B_{\tau}\text{ for any }\tau\in\mathbb F_\lambda\}.
\end{align*}
\end{definition}
Then
$\Gamma_{\lambda}$ is a submonoid of $\nats^{{V}_\lambda}$, called the Knutson-Tao cone in \cite{douglas2024tropical}, and $\mathsf{B}_{\lambda}$ is a subgroup of $\mathbb Z^{{V}_\lambda}$ \cite{kim2011sl3}.

\begin{theorem}\cite[Theorem 7.1]{douglas2024tropical}\label{thm-DS_bijection}
	Suppose that $\fS$ is a triangulable punctured surface with a triangulation $\lambda$. Then
	the Douglas-Sun coordinate map $\kappa$ defines a bijection from $B_{\fS}$ to $\Gamma_{\lambda}$.
\end{theorem}

Note that $\kappa(\emptyset) = {\bf 0}\in\Gamma_{\lambda}$, where all entries of ${\bf 0}$ are zeros. 

\begin{lemma}
    \label{lem-finite}
    $\Gamma_{\lambda}$ is a finitely generated monoid.
\end{lemma}

\begin{proof}
    It follows from \cite{gordan1873ueber}, as $\Gamma_{\lambda}$ is a submonoid of $\nats^{{V}_\lambda}$ and contains ${\bf 0}$.
\end{proof}

The coordinate map $\kappa$ partially preserves the monoid structure in the following sense:
\begin{lemma}[e.g. {\cite[Lem.3.32]{kim2011sl3}}]\label{kappa_additivity}
    If $\alpha,\beta \in B_{\fS}$ can be represented as web diagrams that are disjoint, then $\alpha \beta \in B_{\fS}$ and $\kappa(\alpha\beta) = \kappa(\alpha) + \kappa(\beta)$.
\end{lemma}

We use $\bar\Gamma_{\lambda}$ to denote the subgroup of $\mathbb Z^{V_{\lambda}}$ generated by $\Gamma_{\lambda}$.

\begin{lemma}\label{lem-Gamma-bal}
	Suppose that $\fS$ is a triangulable punctured surface with a triangulation $\lambda$. Then $\bar\Gamma_{\lambda}=\mathsf{B}_{\lambda}$.
\end{lemma}
\begin{proof}
	Define an element ${\bf k} = (k_v)_{v\in V_\lambda} \in 
 \mathbb{N}^{V_\lambda}$ such that $k_v=3$ if $v$ is 
 a vertex contained in an edge of $\lambda$ and $k_v=6$ otherwise; one can easily check that ${\bf k} \in \Gamma_\lambda$.
	For any element ${\bf a}\in \mathsf{B}_{\lambda}$, we have $m{\bf k}+{\bf a}\in \Gamma_{\lambda}$ when $m\in\mathbb N$ is big enough. Then we have 
	${\bf a}\in \bar\Gamma_{\lambda}$, hence $\mathsf{B}_\lambda \subset \bar{\Gamma}_\lambda$. Obviously, we also have $\bar\Gamma_{\lambda}\subset\mathsf{B}_{\lambda}.$
\end{proof}

\subsection{Quantum trace maps for $\SL$-skein algebras}\label{sub-trace}

Suppose that $\fS$ is a triangulable punctured surface with a triangulation $\lambda$. Recall the quiver $\mathsf{H}_{\lambda}$ associated to $\lambda$ (Figure \ref{quiver}).
Define an anti-symmetric integral matrix $ Q_{\lambda} : V_{\lambda}\times V_{\lambda}\rightarrow \mathbb Z$ 
by 
\begin{align}
\label{Q_lambda}
	\begin{split}
 Q_{\lambda}(v,u) = &2\times(\text{the number of arrows in $\mathsf{H}_{\lambda}$ from $v$ to $u$})-\\&
	2\times(\text{the number of arrows in $\mathsf{H}_{\lambda}$ from $u$ to $v$}),
 \end{split}
\end{align}
i.e. 2 times the signed adjacency matrix of $\mathsf{H}_\lambda$.

\def\Xbl{\mathcal{X}_{\hat{q}}^{\rm bl}(\fS,\lambda)}
\def\Xl{\mathcal{X}_{\hat{q}}(\fS,\lambda)}

Define the {\bf quantum torus} algebra 
\begin{align}
\label{quantum_torus_algebra}
\Xl
= R\langle x_{v}^{\pm 1}, v\in V_{\lambda} \rangle/(x_vx_{v'} =\hat q^{2  Q_{\lambda}(v,v')} x_{v'}x_{v},\; v,v'\in  {V}_{\lambda}).
\end{align}

For any $k_1,\cdots,k_m\in \mathbb Z$ and $v_1,\cdots,v_m\in {V}_{\lambda}$, define 
\begin{align}
\label{Weyl-ordering}
    [x_{v_1}^{k_1}\cdots x_{v_m}^{k_m}]_{\text{norm}} = \hat q^{-\sum_{1\leq i<j\leq m}k_ik_j {Q}_{\lambda} (v_i,v_j)}   x_{v_1}^{k_1}\cdots x_{v_m}^{k_m},
\end{align}
often referred to as the Weyl-ordered normalized product. For any ${\bf k}=(k_v)_{v\in V_\lambda}\in \mathbb Z^{{V}_{\lambda}}$, define 
\begin{align}
    \label{x_k}
    x^{\bf k} = [{\textstyle \prod}_{v\in {V}_{\lambda}} 
x_v^{k_v}]_{\text{norm}}.
\end{align}
Recall from Definition \ref{def-Gamma_lambda_and_B_lambda} that $\mathsf{B}_{\lambda}$ is a subgroup of $\mathbb Z^{V_\lambda}$.
Define the {\bf balanced} subalgebra of $\Xl$ by
$$
\Xbl= R\text{-span}\{x^{{\bf k}}\mid {\bf k}\in \mathsf{B}_\lambda\}.$$
Then $\Xbl$ is an $R$-subalgebra of $\Xl$.

For any ${\bf a},{\bf b}\in \mathbb Z^{V_\lambda}$ viewed as row vectors, we use 
\begin{align}
    \label{inner_prod_Q}
    \langle{\bf a}, {\bf b}\rangle_{Q_\lambda}:={\bf a} Q_\lambda {\bf b}^T \in 2\mathbb{Z},
\end{align}
where $T$ in the superscript means transpose. Then we have 
\begin{align}
\label{quantum_torus_relations_vectors}
x^{{\bf a}} x^{{\bf b}} = \hat q^{\langle{\bf a}, {\bf b}\rangle_{Q_\lambda}}
x^{{\bf a}+{\bf b}}
,\quad x^{{\bf a}} x^{{\bf b}} = \hat q^{2\langle{\bf a}, {\bf b}\rangle_{Q_\lambda}} x^{{\bf b}} x^{{\bf a}}.    
\end{align}

We choose a linear order on ${V}_{\lambda}$.
For any ${\bf k}\in \mathbb Z^{{V}_{\lambda}}$, we write ${\bf k}$ as 
$(k_1,k_2,\cdots,k_{|{V}_{\lambda}|})$ and define 
\begin{align}
\label{sum-k}
    \s({\bf k}) = k_1+k_2+\cdots+k_{|{V}_{\lambda}|}.
\end{align} 
For any ${\bf a}$, ${\bf b}\in\mathbb Z^{{V}_{\lambda}}$, define 
\begin{align}\label{eq-linear-order-V}
    \text{${\bf a}\leq {\bf b}$ if $(\s({\bf a}),{\bf a})\leq_{\text{lex}} (\s({\bf b}),{\bf b})$,}
\end{align}
 where $\leq_{\text{lex}}$ is the lexicographic order. 
We have ${\bf a}+{\bf a}'\leq {\bf b}+{\bf b}'$ if 
${\bf a}\leq {\bf b}$ and ${\bf a}'\leq {\bf b}'$. 

\begin{remark}
	Then the linear order 
 ``$\leq$" on $\mathbb Z^{{V}_{\lambda}}$ induces a linear order on $\Gamma_{\lambda}$. Since $\Gamma_{\lambda}\subset \mathbb N^{{V}_{\lambda}}$, 
 $\{{\bf b}\in\Gamma_{\lambda}\mid {\bf b}\leq {\bf a}\}$ is a finite set for any ${\bf a}\in\Gamma_{\lambda}$. 
\end{remark}

\def\spann{R\mbox{\rm{-span}}}

For any ${\bf a}\in \mathbb Z^{{V}_{\lambda}}$, define 
$$D_{{\bf a}} = \spann\{x^{{\bf b}}\mid {\bf b} \in \mathbb{Z}^{V_\lambda},\, {\bf b}\leq {\bf a}\}\text{ and }
D_{<{\bf a}} = \spann\{x^{{\bf b}}\mid {\bf b} \in \mathbb{Z}^{V_\lambda},\, {\bf b}< {\bf a}\},$$
which are $R$-submodules of $\Xl$.
We have $D_{{\bf a}}D_{{\bf b}}\subset D_{{\bf a}+{\bf b}}$ for any
${\bf a}$, ${\bf b}\in \mathbb Z^{{V}_{\lambda}}$, so these $D_{\bf a}$ provide a natural degree filtration on $\Xl$.
Similarly,  we define 
$$\mathcal D_{{\bf a}} = \spann\{\beta\in B_{\fS}\mid \kappa(\beta)\leq {\bf a}\}\text{ and }
\mathcal D_{<{\bf a}} = \spann\{\beta\in B_{\fS}\mid \kappa(\beta)< {\bf a}\},$$
which are $R$-submodules of $\cS_{\bar q}(\fS)$.

\def\tr{{\rm tr}_{\lambda}^{X}}

The following theorem provides a crucial tool for the present paper.
\begin{theorem}[$X$-quantum trace map for ${\rm SL}_3$-skein algebras; {\cite[Proposition 5.80]{kim2011sl3}, \cite[Theorem 15.2]{le2023quantum}}]\label{thm-trace}
	There is an algebra embedding 
 $$
 {\rm tr}_{\lambda}^{X}:
	\cS_{\bar q}(\fS)\rightarrow \Xl,
 $$
 called the $X$-quantum trace map, such that 
	\begin{align}
	\label{highest_term_of_tr_X}
	\tr(\beta) 
 \in x^{\kappa(\beta)} + D_{<\kappa(\beta)}.
 \end{align}
	Furthermore, we have $\im\tr\subset\Xbl.$
\end{theorem}

Note that \eqref{highest_term_of_tr_X} says that $x^{\kappa(\beta)}$ is the unique `highest' term of $\tr(\beta)$. 
The above Theorem shows that 
\begin{equation}\label{eq-filtration}
	\tr(\mathcal D_{<{\rm a}})\subset  D_{<{\rm a}}, \;
	\tr(\mathcal D_{{\rm a}})\subset  D_{{\rm a}},\;
 (\tr)^{-1}(D_{<{\rm a}}) = \mathcal D_{<{\rm a}},\;
 (\tr)^{-1}(D_{{\rm a}}) = \mathcal D_{{\rm a}},
\end{equation} 
for any ${\bf a}\in\mathbb Z^{V_{\lambda}}$. Using this, we can see that $\mathcal{D}_{\bf a}$ would provide a degree filtration on $\cS_{\bar q}(\fS)$, which we use crucially in the proof of the main result of the present section.

\vspace{2mm}

Let us go into more details on this sought-for filtration on $\cS_{\bar q}(\fS)$. For any two basis elements $\beta,\beta'\in B_{\fS}$, define 
$\beta\diamond\beta'$ to be $\kappa^{-1}(\kappa(\beta)+\kappa(\beta'))\in B_{\fS}$. Note that $\beta\diamond\beta' = \beta'\diamond\beta.$ The following lemma will turn out to be useful.

\def\qq{\overset{(\hat q)}{=}}
\def\ee{\overset{(\hat \eta)}{=}}
\def\zz{\overset{(\hat \omega)}{=}}
\def\zzin{\overset{(\hat \omega)}{\in}}

\begin{lemma}\label{lem-fil}
	For any $\beta,\beta'\in B_{\fS}$, we have 
	$$\beta\beta' 
 \in \hat q^{\langle \kappa(\beta),\kappa(\beta')\rangle_{Q_\lambda}} \beta\diamond\beta' + \mathcal D_{<(\kappa(\beta\diamond\beta'))}.$$
\end{lemma}
\begin{proof}
	Theorem \ref{thm-trace} implies that $$\tr(\beta) 
 \in x^{\kappa(\beta)} + D_{<\kappa(\beta)}\text{ and }
	\tr(\beta') 
 \in x^{\kappa(\beta')} + D_{<\kappa(\beta')}.$$
	Then we have 
	\begin{align*}
		\tr(\beta\beta') =   
		\tr(\beta)   
		\tr(\beta') 
  \in & x^{\kappa(\beta)}x^{\kappa(\beta')} + D_{<\kappa(\beta)+\kappa(\beta')}\\
		&=\hat q^{\langle \kappa(\beta),\kappa(\beta')\rangle_{Q_\lambda}}
		 x^{\kappa(\beta)+\kappa(\beta')} + D_{<\kappa(\beta)+\kappa(\beta')}.
	\end{align*}
	We also have 
	$$\tr(\beta\diamond\beta')
 \in x^{\kappa(\beta)+\kappa(\beta')} + D_{<\kappa(\beta)+\kappa(\beta')}.$$
 Thus we have $\tr(\beta\beta' -\hat q^{\langle \kappa(\beta),\kappa(\beta')\rangle_{Q_\lambda}} \beta\diamond\beta')\in D_{<(\kappa(\beta\diamond\beta'))}.$
 Then equation \eqref{eq-filtration} completes the proof.
\end{proof}

\def\Sq{\cS_{\bar q}}
\def\lt{{\rm lt}}
\def\deg{{\rm deg}}

We have
$$
\Sq(\fS) = \cup_{{\bf a}\in\Gamma_\lambda} \mathcal D_{{\bf a}}.
$$
Lemma \ref{lem-fil} implies $\mathcal D_{{\bf a}}\mathcal D_{{\bf b}}\subset \mathcal D_{{\bf a}+{\bf b}}$ for any ${\bf a},{\bf b}\in\Gamma_\lambda.$ So indeed $\mathcal{D}_{\bf a}$ provide a degree filtration on $\Sq(\fS)$. We define it convenient to have the following notation.

\begin{definition}\label{def-lt_and_deg}
For any nonzero element $l\in \Sq(\fS)$, there exists $\alpha\in 
B_{\fS}$ and $c\in R\setminus\{0\}$ such that 
$l 
\in c\alpha + \mathcal D_{<\kappa(\alpha)}$. Then define the leading term and the degree of $l$ as
$$
{\rm lt}(l) = c\alpha \quad\mbox{and}\quad {\rm deg}(l) = \kappa(\alpha)\in \Gamma_\lambda \subset \mathbb{N}^{V_\lambda}.
$$
\end{definition}

\begin{lemma}\label{lem-leadingterm}
	For any $X,Y\in \Sq(\fS)$, we have the following:
	\begin{enumerate}[label={\rm (\alph*)}]\itemsep0,3em
 \item\label{lem_on_deg} ${\rm deg}(XY) = {\rm deg}(X) + {\rm deg}(Y)$, 

 \item\label{lem_on_lt} ${\rm lt}(XY) = \hat q^{2\langle \deg(X),\deg(Y)\rangle_{Q_\lambda}}
	\lt(YX).$
 \end{enumerate}
\end{lemma}
\begin{proof}
	We have $X 
 \in c_X\alpha + \mathcal D_{<\kappa(\alpha)}$ and $Y 
 \in c_Y\beta + \mathcal D_{<\kappa(\beta)}$, where $c_X,c_Y\in R\setminus\{0\}$ and 
	$\kappa(\al)=\deg(X)$, $\kappa(\beta) = \deg(Y)$. 
	Then $XY 
 \in c_Xc_Y\alpha\beta + \mathcal D_{<\kappa(\al)+ \kappa(\beta)}$.
	Lemma \ref{lem-fil} implies that
	$$XY 
 \in c_Xc_Y \hat q^{\langle \kappa(\al),\kappa(\beta)\rangle_{Q_\lambda}}\alpha\diamond\beta + \mathcal D_{<\kappa(\al)+ \kappa(\beta)}.$$
	Similarly, we have 
	$$YX 
 \in c_Xc_Y \hat q^{\langle \kappa(\beta),\kappa(\al)\rangle_{Q_\lambda}}\beta\diamond\alpha + \mathcal D_{<\kappa(\al)+ \kappa(\beta)}.$$
	Note that $\alpha\diamond\beta=\beta\diamond\al.$ This completes the proof.
\end{proof}

Results on the degree filtrations obtained so far using the $X$-quantum trace ${\rm tr}^X_\lambda$ (Theorem \ref{thm-trace}) in the present subsection will be used throughout the paper. We end the present subsection by another important corollary of Theorem \ref{thm-trace}, which will constitute a part of our proof of the main result of the present section:

\begin{corollary}\label{cor-domain}
 Let $\fS$ 
 be a punctured surface.
 Then 
	$\cS_{\bar q}(\fS)$ is a domain.
\end{corollary}
\begin{proof}
	Case 1: Suppose that $\fS$ is triangulable with triangulation $\lambda$. Then Theorem \ref{thm-trace} implies that $\Sq(\fS)$ is a domain because $\tr$ is injective and $\Xl$ is a domain \cite{goodearl2004introduction}.
	
	Case 2: Suppose that $\fS$ is not triangulable. 
	Let $\fS_1,\cdots,\fS_n$ be connected components of $\fS$. 
	Then there exists $1\leq i_0\leq n$ such that $\fS_{i_0}$ is a once or twice punctured sphere. 
	When $\fS_{i_0}$ is once punctured sphere, we have $\Sq(\fS_{i_0})\simeq R$ \cite{higgins2020triangular}. When $\fS_{i_0}$ is twice punctured sphere, we have $\Sq(\fS_{i_0})\simeq R[x_1,x_2]$  
 by Lemma \ref{alphak}. 
	Without loss of generality, we suppose that $\fS_1$ is the once punctured sphere and $\fS_2$ is the twice punctured sphere, and other $\fS_i$ are triangulable. Suppose $\fS'=\cup_{
 3 \leq i\leq n}\fS_i$, and $\lambda$ is a triangulation for $\fS'$. Since $\cS_{\bar q}(\fS_1\cup\fS_2)\simeq R[x_1,x_2]$ is a flat $R$-module, then $\tr\colon\cS_{\bar q}(\fS')\rightarrow \mathcal{X}_{\hat{q}}(\fS',\lambda)$ induces an injective algebra homomorphism 
	$$\text{Id}\otimes_R \tr\colon \cS_{\bar q}(\fS_1\cup\fS_2)\otimes_R\cS_{\bar q}(\fS')\rightarrow \cS_{\bar q}(\fS_1\cup\fS_2)\otimes_R \mathcal{X}_{\hat{q}}(\fS',\lambda).$$
	Note that $\cS_{\bar q}(\fS_1\cup\fS_2)\otimes_R \mathcal{X}_{\hat{q}}(\fS',\lambda) \simeq R[x_1,x_2]\otimes_R  \mathcal{X}_{\hat{q}}(\fS',\lambda)$ is a subalgebra of a quantum torus over $R$. Thus $\cS_{\bar q}(\fS_1\cup\fS_2)\otimes_R \mathcal{X}_{\hat{q}}(\fS',\lambda)$ is a domain. 
	The injectivity of $\text{Id}\otimes_R \tr$ shows
	 $\cS_{\bar q}(\fS)\simeq \cS_{\bar q}(\fS_1\cup\fS_2)\otimes_R\cS_{\bar q}(\fS')$ is also a domain.
\end{proof}

\subsection{Orderly 
finite generating set for the {$\SL$}-skein algebra}

Suppose that
$\fS$ is a triangulable 
punctured surface with a triangulation $\lambda$.

From 
Lemma \ref{lem-finite}, we know there exist ${\bf a}_1,\cdots, {\bf a}_m\in\Gamma_{\lambda}$ such that the monoid $\Gamma_{\lambda}\subset \mathbb{N}^{V_\lambda}$ is generated by ${\bf a}_1,\cdots, {\bf a}_m$. For each $1\leq i\leq m$, suppose 
that $\beta_i$ is an element of $B_{\fS}$ such that
\begin{align}
    \label{beta_i}
\kappa(\beta_i) = {\bf a}_i.
\end{align}

For any two elements $l,l'\in \cS_{\bar q}(\fS)$, we denote $l\overset{(\hat q)}{=} l'$ if $l = \hat q ^{m} l$ for some integer $m$. For an element $l$ of $\cS_{\bar q}(\fS)$ and a subset $S$ of $\cS_{\bar q}(\fS)$, we write $l \overset{(\hat q)}{\in} S$ if $l \in \hat{q}^m S$ for some integer $m$. In case $S = \hat{q}^m S$ holds for all integers $m$, then we have the following symmetry: for $l,l' \in \cS_{\bar q}(\fS)$,
\begin{align}
    \label{q_in_symmetry}
    l \overset{(\hat q)}{\in} l' + S ~ \Leftrightarrow l' \overset{(\hat q)}{\in} l + S.
\end{align}
\begin{proposition}\label{prop-finite}
	Let $\fS$ 
 be a triangulable punctured surface. Then
	the algebra $\Sq(\fS)$ is 
 orderly generated by $\beta_1,\cdots,\beta_m$, that is, 
 $$\Sq(\fS)=\spann\{\beta_1^{k_1}\cdots\beta_m^{k_m}\mid k_i\in\mathbb N\text{ for }1\leq i\leq m\}.$$
\end{proposition}
\begin{proof}
	We use $U$ to denote $\spann\{\beta_1^{k_1}\cdots\beta_m^{k_m}\mid k_i\in\mathbb N\text{ for }1\leq i\leq m\}$. It suffices to show $B_{\fS}\subset U$. We prove this using induction on $\kappa(\gamma)$ for $\gamma\in B_{\fS}$ (we use the linear order 
 ``$\leq$" on $\Gamma_{\lambda}$ defined in subsection \ref{sub-trace}).

 For the base case, we see that the empty web $\emptyset$ belongs to $U$.
 For a fixed $\gamma \in B_{\fS}$, assume that $\gamma' \in U$ holds for every $\gamma' \in B_{\fS}$ satisfying $\kappa(\gamma') < \kappa(\gamma)$.
	There exist $s_1,\cdots, s_m\in \mathbb N$ such that $\kappa(\gamma) = s_1\kappa(\beta_1)+\cdots+s_m\kappa(\beta_m)$. Lemma \ref{lem-fil} implies that
	$\beta_1^{s_1}\cdots \beta_m^{s_m} 
 \overset{(\hat{q})}{\in} \gamma+\mathcal D_{<\kappa(\gamma)}$, hence by \eqref{q_in_symmetry} we have $\gamma \overset{(\hat{q})}{\in} \beta_1^{s_1}\cdots \beta_m^{s_m} + \mathcal{D}_{<\kappa(\gamma)}$. From the 
 induction hypothesis, we have $\mathcal D_{<\kappa(\gamma)}\subset U$. Then we have $\gamma\in U$. 
\end{proof}

Note that the choice of $\beta_i$ in Proposition \ref{prop-finite} is independent of the ground ring $R$ and $\hat q$.

\begin{corollary}\label{cor-finite}
 Let $\fS$ 
 be a surface without boundary. Then the algebra $\Sq(\fS)$ is orderly finitely generated. 
\end{corollary}
\begin{proof}
We can suppose that $\fS$ is connected.	

If $\fS$ is a triangulable punctured surface, then Proposition \ref{prop-finite} implies the Corollary.

Suppose now that $\fS$ is not a triangulable punctured surface; so $\fS$ is either a sphere with less than three punctures, or a connected closed surface. Let $p_1,p_2,p_3$ be three different points in $\fS$, and let $\fS'=\fS\setminus\{p_1,p_2,p_3\}$. Then $\fS'$ is a triangulable punctured surface and the embedding from $\fS'$ to $\fS$ induces a surjective algebra homomorphism from $\Sq(\fS')$ to $\Sq(\fS)$. 
 Proposition \ref{prop-finite} implies that the algebra $\Sq(\fS')$ is orderly finitely generated. 
 Therefore so is $\Sq(\fS)$.
\end{proof}

\subsection{The Unicity Theorem for affine almost Azumaya algebras}
All algebras mentioned in this subsection are over $\mathbb C$.

A commutative algebra $Z$ is called {\bf affine} if it is a domain and it is finitely generated as an algebra.

For a commutative algebra $Z$, we will use $\text{MaxSpec}(Z)$ to denote the set of all maximal ideals of $Z$. 
Then $\text{MaxSpec}(Z)$ is an affine algebraic variety and $Z$ is the coordinate ring of $\text{MaxSpec}(Z)$ if $Z$ is affine \cite{milne2012algebraic}. 

\def\AA{\mathcal{A}}

Let $\AA$ be an algebra, and $Z$ be an affine subalgebra of the center of $\AA$. We use $\text{Irrep}_{\AA}$ to denote the set of all finite dimensional irreducible  representations of $\AA$ considered up to isomorphism (that is, two irreducible representations are considered the same if they are isomorphic).
Then there is a map
$$
\mathbb{X} :\text{Irrep}_\AA\rightarrow \text{MaxSpec}(Z)
$$
defined as the following:
Let $\rho\colon \AA\rightarrow {\rm End}(V)$ be a finite dimensional irreducible representation of $\AA$.
Since $Z$ is contained in the center of $\AA$, for every $x\in Z$ there exists a complex number $r_{\rho}(x)$ such that $\rho(x)= r_{\rho}(x){\rm Id}_V$. 
Thus we get an algebra homomorphism 
$r_{\rho}:Z\rightarrow \mathbb{C}$, which uniquely determines a
point Ker$ (r_{\rho})$ in MaxSpec($Z$). We define 
\begin{equation}\label{eq-X}
	\mathbb X(\rho) = \kernel
 (r_\rho).
\end{equation}
We call $\mathbb X(\rho)$ the {\bf classical shadow} of the representation $(V,\rho)$ of $\AA$ if $Z$ is the center of $\AA$.

\begin{definition}[\cite{korinman2021unicity}]\label{def-Azu}
	A $\mathbb C$-algebra $\AA$ is \textbf{affine almost Azumaya} if $\AA$ satisfies the following conditions; 
	\begin{enumerate}[label={\rm (Az\arabic*)}]\itemsep0,3em
		\item\label{aaz1} $\AA$ is finitely generated as a $\mathbb C$-algebra, 
		\item\label{aaz2} $\AA$ is a domain, 
		\item\label{aaz3} $\AA$ is finitely generated as a module over its center. 
	\end{enumerate}
		We will use $\mathcal Z(\AA)$ to denote the center of $\AA$.
\end{definition}
In \cite{korinman2021unicity}, the condition \ref{aaz2} is that
$\AA$ is prime, which is weaker than $\AA$ is a domain.
We use the refined version in Definition \ref{def-Azu} because it fits well with the skein algebra.

\begin{lemma}[\cite{frohman2019unicity}]
	If the algebra $\AA$ is affine almost Azumaya, then $\mathcal Z(\AA)$ is affine.
\end{lemma}

\begin{definition}\label{def-rank}
	Suppose $\AA$ is an affine almost Azumaya algebra. We use $\widetilde{\mathcal Z(\AA)}$ to denote the fraction field of $\mathcal Z(\AA)$. Then $\widetilde{\AA}=\AA\otimes_{\mathcal Z(\AA)}\widetilde{\mathcal Z(\AA)}$ is a vector space over the field $\widetilde{\mathcal Z(\AA)}$.  
	We call the dimension of $\widetilde{\AA}$ over $\widetilde{\mathcal Z(\AA)}$ as the {\bf rank} of $\AA$ over $\mathcal Z(\AA)$.
\end{definition}

\begin{theorem}[the Unicity Theorem for affine almost Azumaya algebras;\cite{frohman2019unicity,korinman2021unicity}]\label{tmmm8.3}
	Suppose $\AA$ is an affine almost Azumaya algebra of rank $K$ over its center $\mathcal Z(\AA)$.
 Define a subset $U$ of  
${\rm MaxSpec}(\mathcal{Z}(\AA))$ such that $I\in U$ if and only if
there exists an irreducible representation $\rho : \AA\to {\rm End}(V)$ with ${\rm dim}_{\mathbb C} V=K^{\frac{1}{2}}$
and $\mathbb{X}(\rho) = I.$ Then:
	
	\begin{enumerate}[label={\rm (\alph*)}]\itemsep0,3em
	\item any irreducible representation of $\AA$ has dimension at most 
 $K^{\frac{1}{2}}$;
	
	\item The map $ \mathbb{X} :{\rm Irrep}_\AA\rightarrow {\rm MaxSpec}(\mathcal Z(\AA))$ is surjective.

	\item
 The subset $U\subset{\rm MaxSpec}(\mathcal{Z}(\AA))$ is a Zariski open dense subset.
 For any two irreducible representations $V_1$, $V_2$ of $\AA$ with $\mathbb{X}(V_1) = \mathbb{X}(V_2)\in U$, 
 we have that $V_1$ and $V_2$ are isomorphic. 
 
      \item Any representation of $\AA$ sending $\mathcal Z(\AA)$ to scalar operators and  whose classical shadow lies in $U$ is semi-simple. 
      \end{enumerate}
\end{theorem}

\subsection{The Unicity Theorem for $\SL$-skein algebras}
\label{subsec:unicity}
In the 
remainder of  this section, we will assume that
$R=\mathbb C$, and that $\hat\omega\in\mathbb C$ is a root of unity. We have $\bar\omega = \hat \omega^{6}$ and $\omega=\hat\omega^{18}$ such that $\bar\omega^{\frac{1}{6}} = \hat\omega$ and $\omega^{\frac{1}{18}}=\hat\omega$. Suppose that the order of $\omega^2$ is $N$.
Set 
$$
\hat\eta = \hat\omega^{N^2}.
$$
Then we have $\bar\eta = \hat\eta^{6}=\bar\omega^{N^2}$ and
$\eta=\hat \eta^{18}=\omega^{N^2} = \pm 1$. 

In this subsection, we will show that when $\fS$ is a punctured surface, the $\SL$-skein algebra $\cS_{\bar\omega}(\fS)$ is affine almost Azumaya, i.e. satisfies \ref{aaz1}--\ref{aaz3}. 
Combining with Theorem \ref{tmmm8.3}, we 
will be able to prove Theorem \ref{thm.main2}.

\def\tfS{\widetilde{\fS}}
\def\eS{\cS_{\bar\eta}}
\def\zS{\cS_{\bar\omega}}

\def\dd{r}

As a crucial tool for our proof, recall from Theorem \ref{Fro-surface} the algebra homomorphism
$$
\cF:\cS_{\bar\eta}(\fS)\rightarrow \zS(\fS),
$$
called the Frobenius map for ${\rm SL}_3$-skein algebras. 
Suppose that 
\begin{align}
\label{d_1_or_3}
\mbox{$\dd$ is the order of the root of unity $\omega^{\frac{2N}{3}}$.}    
\end{align}
We have 
$\dd=1$ or $3$, since $\omega^{2N} = 1$. An important property of the map $\mathcal{F}$ is as follows:

\begin{proposition}[central elements in the image of Frobenius map;\cite{HLW}]\label{prop-tran}
	Suppose that $\fS$ is a surface without boundary and $K$ is a framed  knot in the thickened surface $\tfS=\fS\times[-1,1]$.  Then $\cF(K)^\dd$ is 
in the center of $\zS(\fS)$.
\end{proposition}

\def\Xe{\mathcal{X}_{\hat{\eta}}(\fS,\lambda)}
\def\Xz{\mathcal{X}_{\hat{\omega}}(\fS,\lambda)}

\vspace{2mm}

Another pivotal gadget that we use is the $X$-quantum trace map ${\rm tr}_{\lambda}^{X}$ (Theorem \ref{thm-trace}) that embeds the skein algebra of a triangulable punctured surface $\fS$ into a quantum torus algebra associated to a triangulation $\lambda$ of $\fS$; here we will use
$$
{\rm tr}^X_\lambda  : \eS(\fS) \to \Xe\quad\mbox{and}\quad
{\rm tr}^X_\lambda  : \zS(\fS) \to \Xz
$$
with a slight abuse of notation.

It is known that the Frobenius map $\mathcal{F}\colon \eS(\fS) \to \zS(\fS)$ between the skein algebras is compatible with the Frobenius map between the quantum tori, which is an algebra embedding
$$
F: \Xe\rightarrow \Xz,
$$
defined by 
\begin{align}\label{eq-Fro-quantum-tori}
F(x_v) = x_v^{N} \quad\mbox{for all}\quad v\in V_\lambda.
\end{align}
Notice that this map $F$ is in general much easier to work with than $\mathcal{F}$.

\begin{theorem}[compatibility with Frobenius maps and $X$-quantum trace maps; \cite{HLW}]\label{thm-com-trace}
 Let $\fS$ 
 be a triangulable punctured surface with a triangulation $\lambda$.
	The following diagram commutes
		\begin{equation}
		\begin{tikzcd}
			\eS(\fS)  \arrow[r, "\cF"]
			\arrow[d, "{\rm tr}_{\lambda}^X"]  
			&  \zS(\fS) \arrow[d, "{\rm tr}_{\lambda}^X"] \\
			\Xe \arrow[r, "F"] 
			& \Xz
		\end{tikzcd}.
	\end{equation}
\end{theorem}
In \S\ref{sec.independent_proofs} we present a proof of Theorem \ref{thm-com-trace}, independently of \cite{HLW}.

Recall from \S\ref{sub-trace} that the $X$-quantum trace map ${\rm tr}^X_\lambda$ is what allows us to transplant the natural degree filtration on the quantum torus algebra to a degree filtration on the skein algebra. With the above theorem, we can now investigate the Frobenius map $\mathcal{F}$ using this degree filtration:

\begin{lemma}\label{lem-D}
 Let $\fS$
 be a triangulable punctured surface with a triangulation $\lambda$.
	For any $\beta\in B_{\fS}$, we have $\cF(\beta) 
 \in \beta^N + \mathcal D_{<N\kappa(\beta)}\in\zS(\fS)$.
\end{lemma}
\begin{proof}
	For any ${\bf a}\in\mathbb Z^{{V}_{\lambda}}$, we have $F(x^{{\bf a}})=x^{N{\bf a}}$. Thus we have $F(D_{<{\bf a}}) = D_{<N{\bf a}}$. 
	Theorems \ref{thm-trace} and \ref{thm-com-trace} 
 imply that
	$$\tr(\cF(\beta)) = F(\tr(\beta)) 
 \in F(x^{\kappa(\beta)} + D_{<\kappa(\beta)}) = 
	x^{N\kappa(\beta)} + D_{<N\kappa(\beta)}.$$
	We also have 
	$$\tr( \beta^N) =  \tr(\beta)^N   
 \in (x^{\kappa(\beta)} + D_{<\kappa(\beta)})^N  = x^{N\kappa(\beta)} + D_{<N\kappa(\beta)}.$$
	
	Thus we have $$\tr(\cF(\beta))- \tr(\beta^N) = \tr(\cF(\beta) -  \beta^N)\in D_{<N\kappa(\beta)}.$$
	Then we have $\cF(\beta) -  \beta^N\in \mathcal D_{<N\kappa(\beta)}$ because of
	equation \eqref{eq-filtration}.
	This completes the proof.
\end{proof}

Recall that  $d$ is the order of the root of unity $\omega^{\frac{2N}{3}}$.  The following statement verifies the condition \ref{aaz3} for $\zS(\fS)$.
	
\begin{proposition}\label{prop-center1}
 Let $\fS$ be a surface without boundary. Then $\zS(\fS)$ is finitely generated as a module over its center. 
\end{proposition}
\begin{proof}
	We can suppose that $\fS$ is connected. 
	
  Case 1: Suppose that $\fS$ is a triangulable punctured surface with a triangulation $\lambda$. 
  Proposition \ref{prop-finite}  implies that
   both the algebra $\zS(\fS)$ and the algebra $\eS(\fS)$ are orderly generated by $\beta_1,\cdots,\beta_m\in B_{\fS}$ (note that each $\beta_i$ is a web in $\widetilde{\fS}$, so we can regard it as an element both in $\zS(\fS)$ and in $\eS(\fS)$).
  
  Since $\eta^2=1$, in $\eS(\fS)$, the relation \eqref{w.cross} implies that
  \begin{equation}\label{eq-cross}
  	\raisebox{-.20in}{
  		
  		\begin{tikzpicture}
  			\tikzset{->-/.style=
  				
  				{decoration={markings,mark=at position #1 with
  						
  						{\arrow{latex}}},postaction={decorate}}}
  			\filldraw[draw=white,fill=gray!20] (-0,-0.2) rectangle (1, 1.2);
  			\draw [line width =1pt,decoration={markings, mark=at position 0.5 with {\arrow{>}}},postaction={decorate}](0.6,0.6)--(1,1);
  			\draw [line width =1pt,decoration={markings, mark=at position 0.5 with {\arrow{>}}},postaction={decorate}](0.6,0.4)--(1,0);
  			\draw[line width =1pt] (0,0)--(0.4,0.4);
  			\draw[line width =1pt] (0,1)--(0.4,0.6);
  			\draw[line width =1pt] (0.4,0.6)--(0.6,0.4);
  		\end{tikzpicture}
  	}
  	= \eta^{\frac {2}{3}}
  	\raisebox{-.20in}{
  		\begin{tikzpicture}
  			\tikzset{->-/.style=
  				
  				{decoration={markings,mark=at position #1 with
  						
  						{\arrow{latex}}},postaction={decorate}}}
  			\filldraw[draw=white,fill=gray!20] (-0,-0.2) rectangle (1, 1.2);
  			\draw [line width =1pt,decoration={markings, mark=at position 0.5 with {\arrow{>}}},postaction={decorate}](0.6,0.6)--(1,1);
  			\draw [line width =1pt,decoration={markings, mark=at position 0.5 with {\arrow{>}}},postaction={decorate}](0.6,0.4)--(1,0);
  			\draw[line width =1pt] (0,0)--(0.4,0.4);
  			\draw[line width =1pt] (0,1)--(0.4,0.6);
  			\draw[line width =1pt] (0.6,0.6)--(0.4,0.4);
  		\end{tikzpicture}
  	}, \quad \mbox{and} \quad
  	\raisebox{-.20in}{
  		\begin{tikzpicture}
  			\tikzset{->-/.style=
  				
  				{decoration={markings,mark=at position #1 with
  						
  						{\arrow{latex}}},postaction={decorate}}}
  			\filldraw[draw=white,fill=gray!20] (-0,-0.2) rectangle (1, 1.2);
  			\draw [line width =1pt,decoration={markings, mark=at position 0.5 with {\arrow{>}}},postaction={decorate}](0.6,0.6)--(1,1);
  			\draw [line width =1pt,decoration={markings, mark=at position 0.5 with {\arrow{>}}},postaction={decorate}](0.6,0.4)--(1,0);
  			\draw[line width =1pt] (0,0)--(0.4,0.4);
  			\draw[line width =1pt] (0,1)--(0.4,0.6);
  			\draw[line width =1pt] (0.6,0.6)--(0.4,0.4);
  		\end{tikzpicture}
    }
  	= \eta^{-\frac {2}{3}}
     	\raisebox{-.20in}{
  		
  		\begin{tikzpicture}
  			\tikzset{->-/.style=
  				
  				{decoration={markings,mark=at position #1 with
  						
  						{\arrow{latex}}},postaction={decorate}}}
  			\filldraw[draw=white,fill=gray!20] (-0,-0.2) rectangle (1, 1.2);
  			\draw [line width =1pt,decoration={markings, mark=at position 0.5 with {\arrow{>}}},postaction={decorate}](0.6,0.6)--(1,1);
  			\draw [line width =1pt,decoration={markings, mark=at position 0.5 with {\arrow{>}}},postaction={decorate}](0.6,0.4)--(1,0);
  			\draw[line width =1pt] (0,0)--(0.4,0.4);
  			\draw[line width =1pt] (0,1)--(0.4,0.6);
  			\draw[line width =1pt] (0.4,0.6)--(0.6,0.4);
  		\end{tikzpicture}
  	}  	
  \end{equation}
    From relations \eqref{wzh.four} and \eqref{eq-cross}, 
   for each $1\leq i\leq m$, we can suppose that $\beta_i$ is contained in the subalgebra of $\eS(\fS)$ generated by $K_{i1},\cdots, K_{it_{i}}$, where 
  $K_{i1},\cdots, K_{it_{i}}$ are  framed knots in $\widetilde{\fS}$. Note that, for $1\leq i,i'\leq m$, $1\leq j\leq t_i$, and $1\leq j'\leq t_{i'}$, equation \eqref{eq-cross} implies that 
  \begin{align}
  \label{K_ij_almost_commute}
      K_{ij} K_{i'j'} = \bar\eta^{k}
K_{i'j'} K_{ij}\in \cS_{\bar\eta}(\fS)
  \end{align}
  holds for an integer $k$ (which depends on $i,i',j,j'$).  Therefore, for each $1\le i \le m$ and each $y_i \in \mathbb{N}$, one has
\begin{align}
    \label{beta_i_to_y_i}
    \beta_i^{y_i} = \sum_{(c_i,c_{i1},\ldots,c_{it_i}) \in \Lambda_i \subset \mathbb{C}\times \mathbb{N}^{t_i}}  c_i \prod_{1\le j \le t_i} K_{ij}^{c_{ij}},
\end{align}
where the product is taken in any chosen order (here each $c_{ij}$ is a multiple of $y_i$, but we don't need that fact), $\Lambda_i$ is some finite subset of $\mathbb{C} \times \mathbb{N}^{t_i}$.

For convenience, define a double index set
$$
{\bf I} := \{(i,j)\mid 1\leq i\leq m, \,\, 1\leq j\leq t_i\},
$$
endowed with the lexicographic order.
  From now on, 
  whenever we deal with a product expression of the form
  $$\prod_{
  (i,j) \in {\bf I}} K_{ij}^{s_{ij}}\in \cS_{\bar\eta}(\fS)\; \quad (\text{ or } \prod_{
  (i,j) \in {\bf I}} \cF(K_{ij})^{s_{ij}}\in \zS(\fS))$$
  the product is taken 
  in the lexicographic order on 
  ${\bf I}$.
  Define 
  \begin{align*}
     W=\left\{ \big(\prod_{
     (i,j) \in {\bf I}} \cF(K_{ij})^{s_{ij}}\big) \beta_1^{k_1}\cdots \beta_{m}^{k_m} \in \zS(\fS)) \left| \begin{array}{l} 0\leq k_i\leq N-1\text{ for }1\leq i\leq m, \\ 0\leq s_{ij}\leq \dd-1 \text{ for } (i,j) \in {\bf I}  
     \end{array} \right. \right\}.
  \end{align*}
Then $W$ is a finite subset of $\zS(\fS)$. We use $V$ to denote the $\mathcal Z(\zS(\fS))$-submodule of $\zS(\fS)$ generated by $W$.
We want to show $\zS(\fS) = V$. It suffices to show $B_{\fS}\subset V.$  We prove this using 
induction on $\kappa(\gamma)$ for $\gamma\in B_{\fS}$.

For the base case, note that the empty web $\emptyset$ lies in $V$. Now, for a fixed $\gamma \in B_{\fS}$, assume that for every $\gamma' \in B_{\fS}$ such that $\kappa(\gamma')<\kappa(\gamma)$ we have $\gamma' \in V$; the goal is then to show $\gamma \in V$.
There exist $s_1,\cdots, s_m\in \mathbb N$ such that $\kappa(\gamma) = s_1\kappa(\beta_1)+\cdots+s_m\kappa(\beta_m)$ (see \eqref{beta_i} and the discussion above it).
Lemma \ref{lem-fil} implies that
\begin{equation}\label{eq1}
	\beta_1^{s_1}\cdots \beta_m^{s_m} \zzin \gamma+\mathcal D_{<\kappa(\gamma)}\in\zS(\fS).
\end{equation}
For each $1\leq i\leq m$, suppose that $$
s_i = Ny_i+r_i, 
$$
where $y_i,r_i\in\mathbb N$
and $0\leq r_i\leq N-1$. 
Lemma \ref{lem-fil} implies that
\begin{equation}\label{eq2}
\beta_1^{s_1}\cdots \beta_m^{s_m} \zzin \beta_1^{Ny_1}\cdots \beta_m^{Ny_m} \beta_1^{r_1}\cdots \beta_m^{r_m}+\mathcal D_{<\kappa(\gamma)}\in\zS(\fS).
\end{equation}
Lemma \ref{lem-D} about the Frobenius map $\mathcal{F}$ implies that
\begin{equation}\label{eq3}
	\cF(\beta_1)^{y_1}\cdots \cF(\beta_m)^{y_m} \beta_1^{r_1}\cdots \beta_m^{r_m} 
 \zzin \beta_1^{Ny_1}\cdots \beta_m^{Ny_m} \beta_1^{r_1}\cdots \beta_m^{r_m}+\mathcal D_{<\kappa(\gamma)}\in\zS(\fS).
\end{equation}
From \eqref{beta_i_to_y_i} we have
\begin{equation*}
	\beta_1^{y_1}\cdots \beta_m^{y_m} = 
 \sum_{(c,(a_{ij})_{(i,j)\in {\bf I}}) \in \Lambda_0 \subset \mathbb{C} \times \mathbb{N}^{\bf I}} c (\prod_{
 (i,j) \in {\bf I}}K_{ij}^{a_{ij}})\in \eS(\fS),
\end{equation*}
for some finite subset $\Lambda_0$ of $\mathbb{C} \times \mathbb{N}^{\bf I}$. 
For 
$(i,j) \in {\bf I}$,
suppose $a_{ij} = \dd b_{ij}+s_{ij}$, where $b_{ij},s_{ij}\in\mathbb N$ and $s_{ij}<\dd$ (note that $s_{ij}=0$ if $\dd=1$). 
Then we have 
  \begin{eqnarray*}
	\beta_1^{y_1}\cdots \beta_m^{y_m} = 
 \sum_{(c',(b_{ij}),(s_{ij})) \in \Lambda_0' \subset \mathbb{C} \times \mathbb{N}^{\bf I} \times (\mathbb{N}_{<\dd})^{\bf I}}
 c' (\prod_{
 (i,j)\in {\bf I}}K_{ij}^{\dd b_{ij}}) (\prod_{
 (i,j)\in {\bf I}}K_{ij}^{s_{ij}})\in \eS(\fS),
\end{eqnarray*}
for some finite set $\Lambda_0'$.
Proposition \ref{prop-tran} implies that $\cF(\prod_{
(i,j) \in {\bf I}}K_{ij}^{\dd b_{ij}})\in \mathcal Z(\zS(\fS))\cap\im \cF$. 
Thus, by applying the Frobenius homomorphism $\mathcal{F}$ we have 
$$\cF(\beta_1)^{y_1}\cdots \cF(\beta_m)^{y_m}= 
\sum_{(z,(s_{ij})) \in \Lambda_0''\subset (\mathcal Z(\zS(\fS))\cap\im\cF) \times (\mathbb{N}_{<\dd})^{\bf I}}
z  (\prod_{
(i,j) \in {\bf I}}\cF(K_{ij})^{s_{ij}}),$$
for some finite subset $\Lambda_0''$ of $(\mathcal Z(\zS(\fS))\cap\im\cF) \times (\mathbb{N}_{<\dd})^{\bf I}$; in particular, each coefficient $z$ in the above sum lies in $\mathcal Z(\zS(\fS))\cap\im\cF$. 
Then we have 
\begin{equation}\label{eq5}
	\cF(\beta_1)^{y_1}\cdots \cF(\beta_m)^{y_m} \beta_1^{r_1}\cdots \beta_m^{r_m}=
 \sum_{(z,(s_{ij})) \in \Lambda_0''}
 z  (\prod_{
 (i,j) \in {\bf I}} \cF(K_{ij})^{s_{ij}})\beta_1^{r_1}\cdots \beta_m^{r_m}\in V,
\end{equation}
by definition of $V$ and $W$.

Applying \eqref{q_in_symmetry} to  \eqref{eq1}, \eqref{eq2} and \eqref{eq3}, we have 
$$
\gamma \zzin \cF(\beta_1)^{y_1}\cdots \cF(\beta_m)^{y_m} \beta_1^{r_1}\cdots \beta_m^{r_m} + \mathcal{D}_{<\kappa(\gamma)}.
$$
By the induction hypothesis, we have $\mathcal D_{<\kappa(\gamma)}\subset V$. Now by \eqref{eq5}, we get $\gamma \in V$, as desired.

\vspace{2mm}

Case 2: Suppose $\fS$ is not a triangulable punctured surface; so it is either a sphere with less than three punctures, or a connected closed surface. 
 Let $p_1,p_2,p_3$ be three different points in $\fS$, and let $\fS'=\fS\setminus\{p_1,p_2,p_3\}$. Then $\fS'$ is a triangulable punctured surface and the embedding $f$ from $\fS'$ to $\fS$ induces a surjective algebra homomorphism $f_*$ from $\zS(\fS')$ to $\zS(\fS)$. We have $f_*(\mathcal Z(\zS(\fS')))\subset\mathcal Z( \zS(\fS))$ because $f_*$ is surjective. 
 
 From the proof of Case 1, we know there exist elements $v_1,\cdots,v_t\in \zS(\fS')$ such that $\zS(\fS')$ is generated by these elements as a module over $\mathcal Z(\zS(\fS'))$. 
 For any $u\in \zS(\fS)$, there exists $v\in \zS(\fS')$ such that $f_*(v) = u$. There exist $z_1,\cdots,z_t\in \mathcal Z(\zS(\fS'))$ such that $v = z_1v_1+\cdots+z_tv_t$.
 Then $u = f_*(v) = f_*(z_1v_1+\cdots+z_tv_t) = f_*(z_1) f_*(v_1)+\cdots+ f_*(z_t) f_*(v_t)$, where $f_*(z_i)\in\mathcal Z(\zS(\fS))$ for $1\leq i\leq t$. 
 Thus $\zS(\fS)$ is generated by $f_*(v_1),\cdots,f_*(v_t)$ as a module over $\mathcal Z(\zS(\fS))$.
\end{proof}

Proposition \ref{prop-tran} implies that $\im\cF\subset\mathcal Z(\cS_{\bar\omega}(\fS))$ when $\dd=1$. Since $z$ is in $\mathcal Z(\zS(\fS)) \cap\im\cF$ in equation \eqref{eq5}, 
the proof of Proposition \ref{prop-center1} implies the following. 

\begin{corollary}\label{cor-Fro}
	Suppose that $\fS$ is a surface without boundary and $\dd=1$ (as in \eqref{d_1_or_3}). Then $\cS_{\bar\omega}(\fS)$ is finitely generated over $\im\cF$. 
\end{corollary}

We can now assemble the results established so far, to obtain the following first main theorem of our paper, which is a
key step to prove the Unicity Theorem for $\SL$-skein algebras at roots of unity. 

\begin{theorem}[Theorem \ref{thm.main1}]\label{thm-almost}
 Let $\fS$ be a punctured surface. 
 Then $\cS_{\bar\omega}(\fS)$ is affine almost Azumaya. 
\end{theorem}
\begin{proof}
	Corollary \ref{cor-domain} verifies \ref{aaz2}, Corollary \ref{cor-finite} verifies \ref{aaz1}, and Proposition \ref{prop-center1} verifies \ref{aaz3}.
\end{proof}

\begin{remark}
    Among the conditions \ref{aaz1}, \ref{aaz2} and \ref{aaz3}, it is only \ref{aaz2} that we proved only for punctured surfaces but not for closed surfaces. That is, if one can extend Corollary \ref{cor-domain} to closed surfaces, then Theorem \ref{thm-almost} would hold also for closed surfaces.
\end{remark}

Then Theorems \ref{tmmm8.3} and \ref{thm-almost} imply Theorem \ref{thm.main2}, the Unicity Theorem for ${\rm SL}_3$-skein algebras, which is 
the second main theorem 
of our paper.

\section{The center of the $\SL$-skein algebra}\label{sec.center}

As we showed in section \ref{sec5}, the $\SL$-skein algebra $\cS_{\bar q}(\fS)$ is affine almost Azumaya when $\bar q$ is a root of unity. So the center is very crucial to study the representation theory of $\cS_{\bar q}(\fS)$, for example, the rank $K$ in Theorem \ref{thm.main2}.
To calculate this $K$, we 
investigate the center of the $\SL$-skein algebra in this section.
We will show that the center of the skein algebra is generated by peripheral skeins when $\bar q$ is not a root of unity (Proposition \ref{prop-generic}), and 
that it is generated by peripheral skeins and central elements in the image of the Frobenius map when $\bar q$ is a root of unity (Theorem \ref{thm-center}).

In this section, we will assume that all  
surfaces are connected.
The parallel statements can be easily generalized to non-connected surfaces.

\subsection{The $\SL$-train track}\label{subsec-SL3-train_track}
Bonahon and collaborators have used train tracks to keep track of the balancedness condition of ${\rm SL}_2$-webs \cite{bonahon1996,bonahon2017representations}.
In this subsection, we will define the $\SL$-train track, which serves as a tool to understand the monoid $\Gamma_\lambda\subset \mathbb{N}^{V_\lambda}$ (Definition \ref{def-Gamma_lambda_and_B_lambda}) associated to an ideal triangulation $\lambda$ of $\fS$, which is in bijection with the non-elliptic web basis $B_{\fS}$ (Definition \ref{def-B_S}) of $\mathscr{S}_{\bar{q}}(\mathfrak{S})$ (Theorem \ref{thm-DS_bijection}). A basic idea is that, for each web diagram $\alpha \in B_{\fS}$ put into a canonical position \ref{C1}--\ref{C2} with respect to a split ideal triangulation $\widehat{\lambda}$ (see \eqref{widehat_lambda}), we record the part of $\alpha$ living over triangles of $\widehat{\lambda}$ (see Remark \ref{rem-DS}) in a combinatorial manner.

An {\bf $\SL$-train track} of a triangle $\mathbb P_3$ is a graph in $\mathbb{P}_3$ containing $7$ components as shown in Figure \ref{traintrack}. So there are two kinds of $\SL$-train tracks of 
$\mathbb P_3$. The one containing a $3$-valent source (resp. sink) is of {\bf type $1$} (resp. {\bf type $-1$}).
Only the combinatorial structure matters, so one can regard an ${\rm SL}_3$-train track as being defined up to isotopy.
 A  {\bf weight} of an $\SL$-train track $\mathcal T$ of $\mathbb P_3$ is  a map $w: 
 \pi_0(\mathcal{T})\rightarrow\mathbb N$, which assigns a non-negative integer weight to each component of $\mathcal{T}$. 
An $\SL$-train track of $\mathbb P_3$ with a weight is called a {\bf weighted $\SL$-train track} of $\mathbb P_3$.
 Since we only deal with $\SL$, we will omit the $\SL$ for the (weighted) $\SL$-train track.
 \begin{figure}[h]
 	\centering
\begingroup%
  \makeatletter%
  \providecommand\color[2][]{%
    \errmessage{(Inkscape) Color is used for the text in Inkscape, but the package 'color.sty' is not loaded}%
    \renewcommand\color[2][]{}%
  }%
  \providecommand\transparent[1]{%
    \errmessage{(Inkscape) Transparency is used (non-zero) for the text in Inkscape, but the package 'transparent.sty' is not loaded}%
    \renewcommand\transparent[1]{}%
  }%
  \providecommand\rotatebox[2]{#2}%
  \newcommand*\fsize{\dimexpr\f@size pt\relax}%
  \newcommand*\lineheight[1]{\fontsize{\fsize}{#1\fsize}\selectfont}%
  \ifx\svgwidth\undefined%
    \setlength{\unitlength}{170.07874016bp}%
    \ifx\svgscale\undefined%
      \relax%
    \else%
      \setlength{\unitlength}{\unitlength * \real{\svgscale}}%
    \fi%
  \else%
    \setlength{\unitlength}{\svgwidth}%
  \fi%
  \global\let\svgwidth\undefined%
  \global\let\svgscale\undefined%
  \makeatother%
  \begin{picture}(1,0.83333333)%
    \lineheight{1}%
    \setlength\tabcolsep{0pt}%
    \put(0,0){\includegraphics[width=\unitlength,page=1]{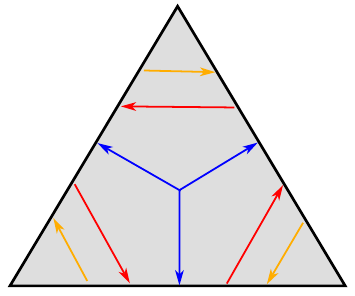}}%
    \put(0.13612251,0.0991177){\color[rgb]{0,0,0}\makebox(0,0)[lt]{\lineheight{1.25}\smash{\begin{tabular}[t]{l}$a_1$\end{tabular}}}}%
    \put(0.48375768,0.65395246){\color[rgb]{0,0,0}\makebox(0,0)[lt]{\lineheight{1.25}\smash{\begin{tabular}[t]{l}$b_1$\end{tabular}}}}%
    \put(0.29303744,0.19611752){\color[rgb]{0,0,0}\makebox(0,0)[lt]{\lineheight{1.25}\smash{\begin{tabular}[t]{l}$a_2$\end{tabular}}}}%
    \put(0.47971005,0.54871596){\color[rgb]{0,0,0}\makebox(0,0)[lt]{\lineheight{1.25}\smash{\begin{tabular}[t]{l}$b_2$\end{tabular}}}}%
    \put(0.64669221,0.17229313){\color[rgb]{0,0,0}\makebox(0,0)[lt]{\lineheight{1.25}\smash{\begin{tabular}[t]{l}$c_2$\end{tabular}}}}%
    \put(0.82015337,0.09977444){\color[rgb]{0,0,0}\makebox(0,0)[lt]{\lineheight{1.25}\smash{\begin{tabular}[t]{l}$c_1$\end{tabular}}}}%
    \put(0.49777936,0.34753189){\color[rgb]{0,0,0}\makebox(0,0)[lt]{\lineheight{1.25}\smash{\begin{tabular}[t]{l}$s$\end{tabular}}}}%
  \end{picture}%
\endgroup%
 \quad  
\begingroup%
  \makeatletter%
  \providecommand\color[2][]{%
    \errmessage{(Inkscape) Color is used for the text in Inkscape, but the package 'color.sty' is not loaded}%
    \renewcommand\color[2][]{}%
  }%
  \providecommand\transparent[1]{%
    \errmessage{(Inkscape) Transparency is used (non-zero) for the text in Inkscape, but the package 'transparent.sty' is not loaded}%
    \renewcommand\transparent[1]{}%
  }%
  \providecommand\rotatebox[2]{#2}%
  \newcommand*\fsize{\dimexpr\f@size pt\relax}%
  \newcommand*\lineheight[1]{\fontsize{\fsize}{#1\fsize}\selectfont}%
  \ifx\svgwidth\undefined%
    \setlength{\unitlength}{170.07874016bp}%
    \ifx\svgscale\undefined%
      \relax%
    \else%
      \setlength{\unitlength}{\unitlength * \real{\svgscale}}%
    \fi%
  \else%
    \setlength{\unitlength}{\svgwidth}%
  \fi%
  \global\let\svgwidth\undefined%
  \global\let\svgscale\undefined%
  \makeatother%
  \begin{picture}(1,0.83333333)%
    \lineheight{1}%
    \setlength\tabcolsep{0pt}%
    \put(0,0){\includegraphics[width=\unitlength,page=1]{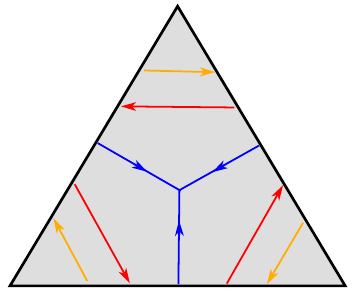}}%
    \put(0.13612251,0.0991177){\color[rgb]{0,0,0}\makebox(0,0)[lt]{\lineheight{1.25}\smash{\begin{tabular}[t]{l}$a_1$\end{tabular}}}}%
    \put(0.48375768,0.65395246){\color[rgb]{0,0,0}\makebox(0,0)[lt]{\lineheight{1.25}\smash{\begin{tabular}[t]{l}$b_1$\end{tabular}}}}%
    \put(0.29303744,0.19611752){\color[rgb]{0,0,0}\makebox(0,0)[lt]{\lineheight{1.25}\smash{\begin{tabular}[t]{l}$a_2$\end{tabular}}}}%
    \put(0.47971005,0.54871596){\color[rgb]{0,0,0}\makebox(0,0)[lt]{\lineheight{1.25}\smash{\begin{tabular}[t]{l}$b_2$\end{tabular}}}}%
    \put(0.64669221,0.17229313){\color[rgb]{0,0,0}\makebox(0,0)[lt]{\lineheight{1.25}\smash{\begin{tabular}[t]{l}$c_2$\end{tabular}}}}%
    \put(0.82015337,0.09977444){\color[rgb]{0,0,0}\makebox(0,0)[lt]{\lineheight{1.25}\smash{\begin{tabular}[t]{l}$c_1$\end{tabular}}}}%
    \put(0.49777936,0.34753189){\color[rgb]{0,0,0}\makebox(0,0)[lt]{\lineheight{1.25}\smash{\begin{tabular}[t]{l}$s$\end{tabular}}}}%
  \end{picture}%
\endgroup%

 	\caption{The train tracks in $\mathbb P_3$ with components labeled by $a_i,b_i,c_i$, for $i=1,2$, and $s$.}\label{traintrack}
 \end{figure}
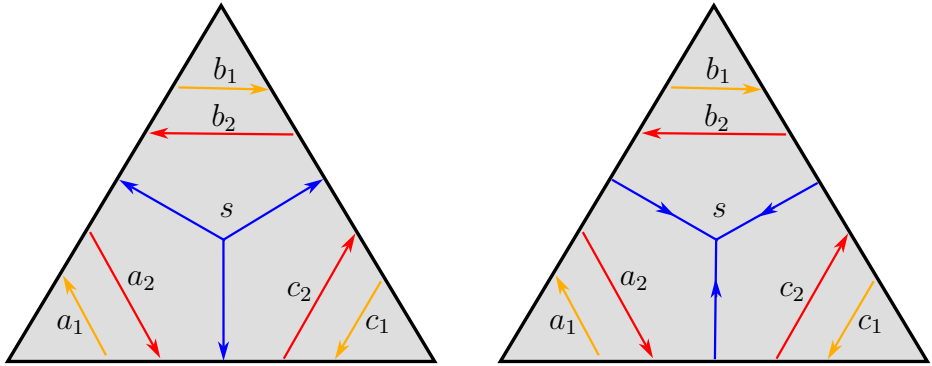

\def\iin{\text{in}}
\def\out{\text{out}}

Suppose that $\mathcal T$ is a weighted train track for $\mathbb P_3$ with weight $w$ and $e$ is an edge of $\mathbb P_3$. Define $e_{\text{in}}(\mathcal T,\mathbb P_3)$ (resp. $e_{\text{out}}(\mathcal T,\mathbb P_3)$) to be the sum of weights of components of $\mathcal T$ that point into the inside (resp. outside) of $\mathbb P_3$.
For example $e_{\text{in}}(\mathcal T,\mathbb P_3) = w(a_1)+w(c_2)$ and 
$e_{\text{out}}(\mathcal T,\mathbb P_3) = w(a_2)+w(c_1)+w(s)$ if $e$ is the bottom edge of $\mathbb P_3$ and $\mathcal T$ is the train track in the left picture in Figure \ref{traintrack}.

Suppose $\fS$ is triangulable with a triangulation $\lambda$. A {\bf train track} $\mathcal{T}$ of $\fS$ with respect to $\lambda$ is a disjoint union of train tracks of triangles $\tau \in \mathbb{F}_\lambda$ of $\lambda$. More precisely, for each triangle $\tau$, we require that $\mathcal{T} \cap \tau$ should be a train track of $\tau$. The components of $\mathcal{T}\cap \tau$ are called components of $\mathcal{T}$; the set of all components of $\mathcal{T}$ is denoted by $\pi_0(\mathcal{T})$. 
For each map $\epsilon:\mathbb F_{\lambda}\rightarrow\{-1,1\}$, we say $\mathcal T$ is of {\bf type $\epsilon$} if $\mathcal T\cap\tau$ is of type
$\epsilon(\tau)$ for each $\tau\in\mathbb F_{\lambda}$. 
Note that the type of train track totally determines the train track. There is a bijection between the set of all train tracks of $\fS$ (with respect to $\lambda$) and the set of all maps from
$\mathbb F_\lambda$ to $\{-1,1\}$. We will identify these two sets.
In particular, for a train track $\mathcal{T}$, the corresponding type map $\mathbb{F}_\lambda \to \{-1,1\}$ is denoted by $\mathcal{T}$, so that $\mathcal{T}(\tau) \in \{-1,1\}$ denotes the type of the triangle $\tau \in \mathbb{F}_\lambda$.

Any map $w:
\pi_0(\mathcal{T})\rightarrow \mathbb N$ restricts to a map
$w:
\pi_0(\mathcal{T} \cap \tau) \rightarrow \mathbb N$ for each $\tau\in\mathbb F_\lambda$.  Then $w$ induces a weight for $\mathcal T\cap \tau$.
A {\bf weight} of  $\mathcal T$ is a map $w:
\pi_0(\mathcal{T})\rightarrow \mathbb N$ with the following properties: for each edge $e$, there two different triangles $\tau,\tau'\in\mathbb F_{\lambda}$ adjacent to $e$; 
the following compatibility equations hold
\begin{align}
    \label{w-compatibility_at_e}
e_{\iin}(\mathcal T\cap \tau,\tau) = e_{\out}(\mathcal T\cap \tau',\tau')\text{ and }e_{\out}(\mathcal T\cap \tau,\tau) = e_{\iin}(\mathcal T\cap \tau',\tau'),
\end{align}
where the weights for $\mathcal T\cap \tau$ and $\mathcal T\cap \tau'$ are induced by $w$. 
A train track of $\fS$ with a weight is called a {\bf weigted train track} of $\fS$.
We will use a pair $(\mathcal T, w)$ to indicate a weighted train track or just $\mathcal T$ when there is no need to emphasize the weight.

We will use $\overline{\mathbb W}_{\lambda}$ to denote the set of all weighted train tracks of $\fS$ (with respect to $\lambda$). 
For each $\mathcal T\in\mathbb T_{\lambda}$, we use $\overline{\mathbb W}_{\lambda,\mathcal T}$ to denote  the set  of all weighted train tracks of $\fS$ whose underlying train track is $\mathcal T$.
Suppose $w_1$ and $w_2$ are two weights for $\mathcal T$. It is easy to show
$w_1+w_2$ is also a weight of $\mathcal T$. Thus $\overline{\mathbb W}_{\lambda,\mathcal T}$ is a monoid. 

We define an equivalence relation $\simeq$ on $\overline{\mathbb W}_{\lambda}$.
Suppose $(\mathcal T_1,w_1),(\mathcal T_2,w_2)\in \overline{\mathbb W}_{\lambda}$.
We define $(\mathcal T_1,w_1)\simeq (\mathcal T_2,w_2)$ if, for each $\tau\in\mathbb F_{\lambda}$ as illustrated in Figure \ref{traintrack}, we have 
$w_1(a_i) = w_2(a_i),w_1(b_i) = w_2(b_i),w_1(c_i) = w_2(c_i)$, and 
$w_1(s)\mathcal T_1(\tau) = w_2(s)\mathcal T_2(\tau)$ (recall $\mathcal{T}_i(\tau)\in\{-1,1\}$ is the type of $\mathcal{T}_i \cap \tau$). 
We use ${\mathbb W}_{\lambda}$ to denote $\overline{\mathbb W}_{\lambda}/\simeq$. We will soon see that $\mathbb{W}_\lambda$ is in bijection with $\Gamma_\lambda$.

For each train track $\mathcal T$, we use ${\mathbb W}_{\lambda,\mathcal T}$ to denote $\overline{\mathbb W}_{\lambda,\mathcal T}/\simeq$. Actually, we have 
${\mathbb W}_{\lambda,\mathcal T}= \overline{\mathbb W}_{\lambda,\mathcal T}$. 
So ${\mathbb W}_{\lambda,\mathcal T}$ also has a monoid structure. 
Note that ${\mathbb W}_{\lambda,\mathcal T_1}\cap {\mathbb W}_{\lambda,\mathcal T_2}$ may not be empty for two different train tracks $\mathcal T_1,\mathcal T_2$.

\vspace{2mm}

Take any element of $\mathbb{W}_\lambda$, represented by a weighted train track $(\mathcal{T},w) \in \overline{\mathbb{W}}_\lambda$. We will construct an element $(k_v)_{v\in V_\lambda} \in \Gamma_\lambda \subset \mathbb{N}^{V_\lambda}$. Let $\tau\in \mathbb{F}_\lambda$ be any triangle, and let the elements of $V_\lambda$, i.e. vertices of the quiver $\mathsf{H}_\lambda$, lying in $\tau$, be labeled as in Figure \ref{quiver}. We now turn the weighted train track $(\mathcal{T}\cap \tau, w|_{\mathcal{T}\cap \tau})$ of $\tau$ into a canonical web $\alpha_\tau$ in $\tau$ (as in \ref{C2}) as follows. For each of the six corner arc components $a$ of $\mathcal{T} \cap \tau$, the web $\alpha_\tau$ has $w(a)$ parallel copies of the corresponding corner arcs. For the 3-valent component $s$ of $\mathcal{T}$, if $w(s)=0$ then $\alpha_\tau$ has no honeycomb component; if $w(s)\neq 0$ then $\alpha_\tau$ has the honeycomb component of degree $w(s) \mathcal{T}(\tau)$. Now, for each of the seven vertices $v\in V_\lambda$ lying in $\tau$, let $k_v$ be the Douglas-Sun coordinates of such a web $\alpha_\tau$ in $\tau$, given by the formulas in \eqref{DS_edges}--\eqref{DS_triangle}. The following statement is one interpretation of the result proved in \cite{douglas2024tropical}.
\begin{lemma}[\cite{douglas2024tropical}]
\label{lem-iso-traintrack}
The above construction yields a well-defined bijection
$$
f : \mathbb{W}_\lambda \to \Gamma_\lambda \subset \mathbb{N}^{V_\lambda}, \qquad (\mathcal{T},w) \mapsto (k_v)_{v\in V_\lambda}.
$$
\end{lemma}

Explicit formulas expressing this map $f$ can be easily deduced from \eqref{DS_edges}--\eqref{DS_triangle}; see the previous version of the current paper for them, and see e.g. the formulas in the proof of Proposition \ref{prop-generic} in the next subsection.

Recall from Theorem \ref{thm-DS_bijection} the bijection $\kappa : B_{\fS} \to \Gamma_\lambda$, the Douglas-Sun coordinate map. We thus get a bijection
\begin{align}\label{map-g}
    g\colon \mathbb W_\lambda\xrightarrow{f}\Gamma_\lambda\xrightarrow{\kappa^{-1}} B_\fS,
\end{align}
which enables us to deal with non-elliptic basis webs in $B_{\fS}$ through corresponding weighted train tracks. The inverse map $g^{-1} :B_{\fS} \to \mathbb{W}_\lambda$ can be described as follows. For $\alpha \in B_{\fS}$, isotope it into a canonical position \ref{C1}--\ref{C2} with respect to a split ideal triangulation $\widehat{\lambda}$. For each triangle $\widehat{\tau}$ of $\widehat{\lambda}$ corresponding to a triangle $\tau$ of $\lambda$, let $d_\tau \in \mathbb{Z}$ be the degree of the honeycomb component of $\alpha \cap \widehat{\tau}$. Let $\mathcal{T}$ be any train track such that $d_\tau \mathcal{T}(\tau) \ge 0$ for each $\tau$. Construct the weight $w$ of $\mathcal{T}$ for each triangle $\tau$ by: for each corner arc component $a$ of $\mathcal{T} \cap \tau$, let $w(a)$ be the number of corner arcs of $\alpha\cap \widehat{\tau}$ 
 of the same kind, and for the 3-valent component $s$ of $\mathcal{T} \cap \tau$, let 
 \begin{align}
 \label{w_of_d_and_degree_of_honeycomb}
     w(s) := d_\tau \mathcal{T}(\tau).
 \end{align}

 Note that the degrees of the honeycomb components of the part of the web $\alpha$ over each triangle of $\widehat{\lambda}$ gives a restriction on the choice of the train track $\mathcal{T}$. One also recalls how to extract the information on these degrees from the Douglas-Sun coordinates:
 \begin{remark}[{\cite[Rem.3.28, eq.(3.5)]{kim2011sl3}}]\label{rem-deg}
Suppose that ${\bf k}=(k_v)_{v\in V_\lambda}\in\Gamma_\lambda\subset \mathbb{N}^{V_\lambda}$ and $\tau\in\mathbb F_\lambda$ is as illustrated in Figure \ref{quiver}.  
Then
\begin{align}\label{eq-deg-h}
 \sum_{i=1,2,3} 
 k_{v_{i2}}- \sum_{i=1,2,3} 
 k_{v_{i1}}
\end{align} 
equals $3$ times the degree $d_\tau \in \mathbb{Z}$ of the honeycomb of the web $\alpha=\kappa^{-1}({\bf k})\in B_\fS$  
in 
the triangle $\widehat{\tau}$ of $\widehat{\lambda}$, when $\alpha$ is put into a canonical position with respect to $\widehat{\lambda}$.
\end{remark}

\begin{lemma}\label{lem-iso-traintrack_monoids}
For any train track $\mathcal T$, define 
$$\Gamma_{\lambda,\mathcal T}=\{{\bf k}\in\Gamma_\lambda\mid 
d_\tau({\bf k})\mathcal T(\tau)\geq 0\text{ for any }\tau\in\mathbb F_\lambda\}.$$
Then the map $f : \mathbb{W}_\lambda \to \Gamma_\lambda$ of Lemma \ref{lem-iso-traintrack} restricts to an isomorphism between monoids:
$$
f|_{\mathbb W_{\lambda,\mathcal T}}: \mathbb W_{\lambda,\mathcal T}\rightarrow \Gamma_{\lambda,\mathcal T}.
$$    
\end{lemma}

\vspace{2mm}

The peripheral skeins will play a crucial role in the study of the center of the skein algebra in the present section. Below we establish useful lemmas regarding peripheral skeins, which make use of the train tracks and their weights.

Suppose that $p$ is a puncture in a punctured surface $\fS$ and $(\mathcal T,w)$ is a weighted train track of  $\fS$ with respect to a triangulation $\lambda$ of $\fS$. Suppose that $U$ is a small open regular neighborhood of $p$ and that the part of the train track $\mathcal T\cap U$ looks like the picture in Figure \ref{puncture}.
There are two peripheral skeins $\alpha_p,\cev{\alpha}_p$ going around $p$. Here 
the direction of $\alpha_p$ (resp. $\cev{\alpha}_p$) is counterclockwise (resp. clockwise).
 \begin{figure}[htbp!]
	\centering
\begingroup%
  \makeatletter%
  \providecommand\color[2][]{%
    \errmessage{(Inkscape) Color is used for the text in Inkscape, but the package 'color.sty' is not loaded}%
    \renewcommand\color[2][]{}%
  }%
  \providecommand\transparent[1]{%
    \errmessage{(Inkscape) Transparency is used (non-zero) for the text in Inkscape, but the package 'transparent.sty' is not loaded}%
    \renewcommand\transparent[1]{}%
  }%
  \providecommand\rotatebox[2]{#2}%
  \newcommand*\fsize{\dimexpr\f@size pt\relax}%
  \newcommand*\lineheight[1]{\fontsize{\fsize}{#1\fsize}\selectfont}%
  \ifx\svgwidth\undefined%
    \setlength{\unitlength}{144.56692913bp}%
    \ifx\svgscale\undefined%
      \relax%
    \else%
      \setlength{\unitlength}{\unitlength * \real{\svgscale}}%
    \fi%
  \else%
    \setlength{\unitlength}{\svgwidth}%
  \fi%
  \global\let\svgwidth\undefined%
  \global\let\svgscale\undefined%
  \makeatother%
  \begin{picture}(1,0.68627451)%
    \lineheight{1}%
    \setlength\tabcolsep{0pt}%
    \put(0,0){\includegraphics[width=\unitlength,page=1]{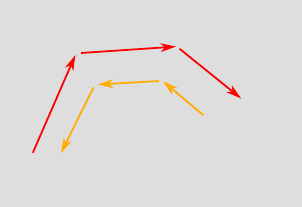}}%
    \put(0.39097281,0.55604138){\color[rgb]{0,0,0}\makebox(0,0)[lt]{\lineheight{1.25}\smash{\begin{tabular}[t]{l}$v_1$\end{tabular}}}}%
    \put(0.10836578,0.342025){\color[rgb]{0,0,0}\makebox(0,0)[lt]{\lineheight{1.25}\smash{\begin{tabular}[t]{l}$v_2$\end{tabular}}}}%
    \put(0.26664872,0.26622751){\color[rgb]{0,0,0}\makebox(0,0)[lt]{\lineheight{1.25}\smash{\begin{tabular}[t]{l}$u_2$\end{tabular}}}}%
    \put(0.42202144,0.36877704){\color[rgb]{0,0,0}\makebox(0,0)[lt]{\lineheight{1.25}\smash{\begin{tabular}[t]{l}$u_1$\end{tabular}}}}%
    \put(0.55801098,0.31081427){\color[rgb]{0,0,0}\makebox(0,0)[lt]{\lineheight{1.25}\smash{\begin{tabular}[t]{l}$u_m$\end{tabular}}}}%
    \put(0.71028682,0.46686789){\color[rgb]{0,0,0}\makebox(0,0)[lt]{\lineheight{1.25}\smash{\begin{tabular}[t]{l}$v_m$\end{tabular}}}}%
    \put(0.49575199,0.14584331){\color[rgb]{0,0,0}\makebox(0,0)[lt]{\lineheight{1.25}\smash{\begin{tabular}[t]{l}$p$\end{tabular}}}}%
    \put(0,0){\includegraphics[width=\unitlength,page=2]{train_track_around_puncture.pdf}}%
  \end{picture}%
\endgroup%

	\caption{The train track components around the puncture $p$. Here $v_i,u_i$, for $1\leq i\leq m$, are the labelings for these components.}\label{puncture}
\end{figure}

\begin{lemma}\label{lem-loop}
One has:
\begin{enumerate}[label={\rm (\alph*)}]
	\item If $u:=\text{min}\{w(u_1),\cdots,w(u_m)\}$,  
 then $g(\mathcal T,w)$ contains $u$ copies of ${\alpha}_p$ but not $u+1$ copies of $\alpha_p$.
	
	\item If $v:=\text{min}\{w(v_1),\cdots,w(v_m)\}$, 
 then $g(\mathcal T,w)$ contains $v$ copies of $\cev{\alpha}_p$ but not $v+1$ copies of $\cev{\alpha}_p$.
 \end{enumerate}
\end{lemma}
\begin{proof}
	We only prove (a). We define a new weight $w'$ for $\mathcal T$ such that
	$w'(u_i) = u$ for $1\leq i\leq m$ and $w'(a) = 0$ for any other component $a$ of $\mathcal T$. 
 One can easily see that both $w'$  and $w-w'$ are weights of $\mathcal T$.
 Observe that $g(\mathcal T,w')$ consists of $u$ parallel copies of the peripheral skein $\alpha_p$.
	Lemma 
 \ref{lem-iso-traintrack_monoids}
  implies that $f(\mathcal T,w) = f(\mathcal T,w')+f(\mathcal T,w-w')$.
	Since $g(\mathcal{T},w')$ and $g(\mathcal{T},w-w')$ can be represented by webs that are disjoint (because $g(\mathcal{T},w')$ consists only of peripheral skeins), one can apply the additivity property of $\kappa$ 
 (Lemma \ref{kappa_additivity})
 to get 
 \begin{align*}
 \kappa(g(\mathcal T,w')\cup g(\mathcal T,w-w')) & = \kappa(g(\mathcal T,w'))+ \kappa( g(\mathcal T,w-w')) \\
 & =
	f(\mathcal T,w')+ f(\mathcal T,w-w').
 \end{align*}
	Then we have 
	$$\kappa(g(\mathcal T,w)) = f(\mathcal T,w)=f(\mathcal T,w')+ f(\mathcal T,w-w')=\kappa(g(\mathcal T,w')\cup g(\mathcal T,w-w')).$$
	Since $\kappa$ is a bijection,  
 we have $g(\mathcal T,w) =g(\mathcal T,w')\cup g(\mathcal T,w-w') $.  
 So $g(\mathcal{T},w)$ contains $u$ copies of $\alpha_p$.
 
 If $g(\mathcal{T},w-w')$ contains a copy of $\alpha_p$, we have $(w-w')(u_i) \ge 1$ for all $i=1,\ldots,m$. However, by construction of $w'$, we have $\min\{(w-w')(u_1),\ldots,(w-w')(u_m)\} = u-u=0$, which is a contradiction. So $g(\mathcal{T},w-w')$ does not contain any copy of $\alpha_p$, finishing the proof. 
\end{proof}
Here is another useful lemma to be used later:

\begin{lemma}
\label{lem-only_peripherals}
Let $(\mathcal{T},w)$ be a weighted train track of a punctured surface $\mathfrak{S}$ with respect to a triangulation $\lambda$ of $\fS$. Suppose that for each puncture $p$, we have $w(u_1)=w(u_2)=\cdots=w(u_m)$ and $w(v_1)=w(v_2)=\cdots=w(v_m)$ in the notation discussed above. Suppose further that $w(a)=0$ for every 3-valent component $a$ of $\mathcal{T}$.

Then, either $g(\mathcal{T},w)$ is $\emptyset$, or $g(\mathcal{T},w)$ consists only of peripheral skeins.
\end{lemma}
\begin{proof}
Per each puncture $p$, write $m$ as $m_p$, and denote $u_1,\ldots,u_m$ by $u_{p,1},\ldots,u_{p,m_p}$, and $v_1,\ldots,v_m$ by $v_{p,1},\ldots,v_{p,m_p}$. Write $u_p = w(u_{p,i})$ and $v_p = w(v_{p,i})$. By Lemma \ref{lem-loop}, we see that $g(\mathcal{T},w)$ contains $u_p$ copies of $\alpha_p$ and $v_p$ copies of $\cev{\alpha}_p$.

Define a weight $w_p$ by setting $w_p (u_{p,i}) = u_p$, $w_p(v_{p,i})=v_p$, for all $i=1,\ldots,m_p$, and $w_p(a)=0$ for all other components $a$. Then $g(\mathcal{T},w_p)$ consists exactly of $u_p$ copies of $\alpha_p$ and $v_p$ copies of $\cev{\alpha}_p$. By applying (the proof of) Lemma \ref{lem-loop} inductively, we see that $g(\mathcal{T},w) = g(\mathcal{T}, w - \sum_p w_p) \cup \bigcup_p g(\mathcal{T}, w_p)$, where $\sum_p$ and $\bigcup_p$ run over all punctures $p$. Since $w - \sum_p w_p=0$ in our case, we have $g(\mathcal{T},w) = \bigcup_p g(\mathcal{T},w_p)$, which consists only of peripheral skeins, if it's nonempty. 
\end{proof}

\def\ZZ{\mathcal Z( \cS_{\bar q}(\fS))}
\def\Sp{\cS_{\bar q}^{\circ}(\fS)}

\subsection{On the case when the quantum parameter is not a root of unity}

Recall that,
for an $R$-algebra $\AA$, we  use $\mathcal Z(\AA)$ to denote its center.
Let
\begin{align}
\label{peripheral_subalgebra}
\cS_{\bar q}^{\circ}(\fS) := \mbox{the subalgebra of $\cS_{\bar q}(\fS)$ generated by all peripheral skeins.}
\end{align} 
It is easy to see that
\begin{align}
\label{peripheral_in_center}
    \cS_{\bar q}^{\circ}(\fS)\subset\mathcal Z( \cS_{\bar q}(\fS)).
\end{align}

Using the basis $B_{\fS}$ one can deduce the following.
\begin{lemma}\label{lem-ll}
	We use 
 $\mathcal{P}$ to denote the set of punctures of a punctured surface $\fS$.
	For each $p\in 
 \mathcal{P}$, we use $\alpha_p,\cev{\alpha}_p$ to denote the two peripheral skeins going around $p$.
	Then $$\cS_{\bar q}^{\circ}(\fS) = R[\alpha_p,\cev{\alpha}_p\mid p\in 
 \mathcal{P}].$$
\end{lemma}

In this section, we will assume $\hat q^k\neq 1$ for any  $k\in\mathbb Z$. 
Recall from Definition \ref{def-lt_and_deg} that we defined the degree for any nonzero element in $\cS_{\bar q}(\fS)$.

\begin{proposition}\label{prop-generic}
 If $\mathfrak{S}$ is a punctured surface, then 
	$\cS_{\bar q}^{\circ}(\fS)=\mathcal Z( \cS_{\bar q}(\fS)).$
\end{proposition}
\begin{proof}
	When $\fS$ is 
 a once punctured sphere, this is obvious. When $\fS$ is 
 a twice punctured sphere, Lemma \ref{alphak} implies the Proposition.
	
	Suppose now that $\fS$ is triangulable with a triangulation $\lambda$. Suppose that $0\neq X\in\ZZ$. We will use induction on $\deg(X)$ to show $X\in\Sp.$ 
 The base case is when $\deg(X)={\bf 0}$, 
 for which we have $X=c \,\emptyset\in \Sp$, where $c\in R\setminus\{0\}$. 
 Now, fix any ${\bf k}=(k_v)_{v\in V_\lambda} \in \Gamma_{\lambda}\subset \mathbb{N}^{V_\lambda}$. As the induction hypothesis, assume that 
 $X\in \Sp$ 
 holds for every $X\in\ZZ$ 
 such that $\deg(X)<{\bf k}$. 
	Let $Y$  be an element in $\ZZ$ 
 such that $\deg(Y)={\bf k}$; our goal is show that $Y\in \Sp$. 
 We begin by observing that the following holds, due to Definition \ref{def-lt_and_deg}:
 $$Y 
 \in c \kappa^{-1}({\bf k})+ \mathcal D_{<{\bf k}},$$ where $c\in R\setminus\{0\}$.

 Let $\mathsf{H}_\lambda$ be the quiver associated to $\lambda$ whose vertex set is $V_\lambda$, looking as in Figure \ref{quiver} for each triangle $\tau \in \mathbb{F}_\lambda$. Let $Q_\lambda$ be 2 times the signed adjacency matrix of $\mathsf{H}_\lambda$ \eqref{Q_lambda}.

From $Y\in \ZZ$ and Lemma \ref{lem-leadingterm}\ref{lem_on_lt}  
 it follows that
 \begin{align}
     \label{kQa_zero_on_the_nose}
 {\bf k} Q_{\lambda} {\bf a}^T=0
 \end{align}
 holds for any ${\bf a}\in\Gamma_\lambda$ (let $\beta = \kappa^{-1}({\bf a})\in \cS_{\bar q}(\fS)$, and take the leading terms ${\rm lt}$ of both sides of $Y\beta = \beta Y$). Lemma \ref{lem-Gamma-bal} implies that
 \eqref{kQa_zero_on_the_nose} holds also for any ${\bf a}\in\mathsf{B}_\lambda=\bar{\Gamma}_\lambda$. Since $3\mathbb Z^{V_\lambda}\subset \mathsf{B}_\lambda$, 
 we can deduce that
	${\bf k} Q_{\lambda}=0$. For later referral, we write this result as
 $$
 \frac{1}{2} {\bf k}Q_\lambda=0.
$$

  Consider the weighted train track $(\mathcal T,w) = f^{-1}({\bf k}),$ where $f$ is the map in Lemma \ref{lem-iso-traintrack}. 
	 
	 Suppose that $e$ is an edge in $\lambda$ and $\tau,\tau'$ are two different triangles adjacent to $e$. The intersections $\mathcal T\cap (\tau\cap\tau')$ and 
	 $\mathsf{H}_{\lambda} \cap (\tau\cap\tau')$ are as illustrated in Figure \ref{traintrack-quiver}. 
	 Since $w$ is a weight of $\mathcal T$,  
  from the edge compatibility \eqref{w-compatibility_at_e} we have 
	 \begin{align}\label{train-track-eq}
   \begin{split}
       w(b_1) + w(c_2) + [\mathcal{T}(\tau)]_+ w(s) & = w(b_1') + w(c_2') + [-\mathcal{T}(\tau')]_+ w(s'), \\
       w(b_2) + w(c_1) + [-\mathcal{T}(\tau)]_+ w(s) & = w(b_2') + w(c_1') + [\mathcal{T}(\tau')]_+ w(s'),
   \end{split}
	 \end{align}
  where for any real number $A$, the symbol $[A]_+$ stands for the positive part:
  $$
  [A]_+ := \frac{1}{2}(A+|A|) = 
  \left\{\begin{array}{cl}
      A & \mbox{if $A\ge 0$}, \\
      0 & \mbox{if $A\le 0$}.
  \end{array}\right.
$$

  From the definition of $f$ (see \S\ref{subsec-SL3-train_track}, and \eqref{DS_edges}--\eqref{DS_triangle})  
  we get (see also \cite{douglas2024tropical,kim2011sl3}):
	 \begin{equation}\label{eq-k-cor}
	 	\left. \begin{split}
    k_{v_{11}} &= 2w(a_1)+ 2w(c_2) + w(a_2) + w(c_1) +(\mathcal{T}(\tau) +3[-\mathcal{T}(\tau)]_+) w(s),\;\\
    k_{u_{11}} &= 2w(a_1')+ 2w(b_2') + w(a_2') + w(b_1') +(\mathcal{T}(\tau') + 3[-\mathcal{T}(\tau')]_+) w(s'),\\
    k_{v_{22}} &= 2w(a_1)+ 2w(b_2) + w(a_2) + w(b_1) +(2\mathcal{T}(\tau)+3[-\mathcal{T}(\tau)]_+) w(s),\;\\
    k_{u_{22}} &= 2w(a_1')+ 2w(c_2') + w(a_2') + w(c_1') +(2\mathcal{T}(\tau')+3[-\mathcal{T}(\tau')]_+) w(s'), \\
    k_v &= 2(w(a_1)+w(b_1)+w(c_1)) + w(a_2) + w(b_2) + w(c_2) +3w(s),\\
    k_u &= 2(w(a_1')+w(b_1')+w(c_1')) + w(a_2') + w(b_2') + w(c_2') +3w(s').
	 	\end{split}
   \right\}
	 \end{equation}
For example, the part $(\mathcal{T}(\tau) + 3[-\mathcal{T}(\tau)]_+) w(s)$ in $k_{v_{11}}$ equals $w(s)$ if $\mathcal{T}(\tau)=1$ and equals $2 w(s)$ if $\mathcal{T}(\tau)=-1$.
  
	 \begin{figure}
	 	\centering	 
   \scalebox{0.9}{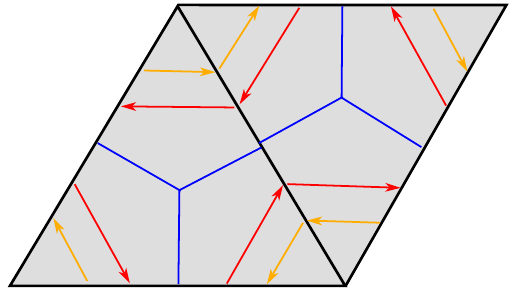} \hspace{2mm}
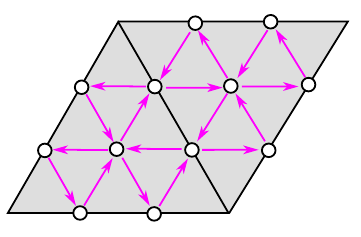
	 	\caption{In each picture, the left (resp. right) triangle is $\tau$ (resp. $\tau'$). The orientations of $s,s'$ are arbitrary.}\label{traintrack-quiver}
	 \end{figure}
	 
	 From equations \eqref{train-track-eq} and \eqref{eq-k-cor}, together with $A = [A]_+ - [-A]_+$ and $|A|=[A]_++[-A]_+$, one can verify that (we omit a detailed computation, which is straightforward)
	 \begin{equation}\label{eq-v31}
	  	\begin{split}
     \frac{1}{2} ({\bf k}Q_\lambda)_{v_{31}} &= \frac{1}{2} \sum_{u\in V_\lambda} 
     k_u 
     (Q_{\lambda})_{u,v_{31}} = 
     k_{u_{11}}+ 
     k_v - 
     k_{v_{22}}- 
     k_u
     \\ & = 
     3(w(c_1)-w(c_1')
     +[-\mathcal{T}(\tau)]_+ w(s) - [\mathcal{T}(\tau')]_+ w(s'))\\
     & =
     3(w(b_2')-w(b_2)),
	  	\end{split}
	 \end{equation}
  which is zero because $\frac{1}{2} {\bf k}Q_\lambda=0$.
	 Similarly, we have 
	 \begin{equation}\label{eq-v32}
	 	\begin{split}
    \frac{1}{2} ({\bf k}Q_\lambda)_{v_{32}} &= \frac{1}{2} \sum_{u\in V_\lambda} 
    k_u 
    (Q_\lambda)_{u,v_{32}} = 
    k_{v_{11}} + 
    k_u - 
    k_{u_{22}} - 
    k_v
    \\ &= 
    3(w(c_2)-w(c_2')) \\
    & =
    3(w(b_1')-w(b_1)-[\mathcal{T}(\tau)]_+ w(s) + [-\mathcal{T}(\tau')]_+ w(s')),
	 	\end{split}
	 \end{equation}
  which is zero because $\frac{1}{2}{\bf k}Q_\lambda=0$.
	 
 So we have $w(b_2)= w(b_2')$ and $w(c_2) = w(c_2')$.

 For each triangle $\tau \in \mathbb{F}_\lambda$ as illustrated in Figure \ref{quiver}, for the internal vertex $v$ of $\mathsf{H}_\lambda$ we have
 \begin{align}\label{eq-d}
   \frac{1}{2}({\bf k}Q_\lambda)_v = \frac{1}{2} \sum_{u\in V_\lambda} 
   k_u 
   (Q_\lambda)_{uv} & = \sum_{i=1,2,3} 
   k_{v_{i2}}- \sum_{i=1,2,3} 
   k_{v_{i1}},
	 \end{align}
  which is zero because $\frac{1}{2} {\bf k}Q_\lambda=0$. From \eqref{w_of_d_and_degree_of_honeycomb} and Remark \ref{rem-deg} we have
  \begin{align}
      \label{w_d_as_sum}
      \sum_{i=1,2,3} 
   k_{v_{i2}}- \sum_{i=1,2,3} 
   k_{v_{i1}} = 3w(s) \mathcal{T}(\tau).
  \end{align}
  for the 3-valent component $s$ of the train track $\mathcal{T}\cap \tau$. Hence we have
  \begin{align}
      \label{3_w_d_zero}
      3w(s)=0,
  \end{align}
therefore $w(s)=0$.

  Then, by putting in $w(s)=w(s')=0$ in \eqref{eq-v31}--\eqref{eq-v32}, one obtains $w(c_1)=w(c_1')$ and $w(b_1)=w(b_1')$.

 Thus one observes that Lemma \ref{lem-only_peripherals} applies, so that 
 $\kappa^{-1}({\bf k})$ is a product of peripheral skeins. Then $c\kappa^{-1}({\bf k})\in \Sp$. The induction assumption implies that $(\mathcal D_{<{\bf k}}
 \cap\ZZ)\subset \Sp$. 
 Since $Y \in c \kappa^{-1}({\bf k})+ \mathcal D_{<{\bf k}}$, we have $Y\in\Sp$.
\end{proof}

\subsection{The $\text{mod}\;3$ intersection number for webs}

For a punctured surface $\fS$, we review a subalgebra of $\cS_{\bar{q}}(\fS)$ introduced in \cite{Wan24}.
For a web diagram $\alpha$ in $\fS$, every crossing point $p$ of $\alpha$ determines an integer 
``$1$" or 
``$-1$" as illustrated in Figure \ref{fg-crossing}, denoted as $v(p)$, which add up to the `writhe' of $\alpha$.

\begin{figure}[h]  
	\centering\includegraphics[width=5cm]{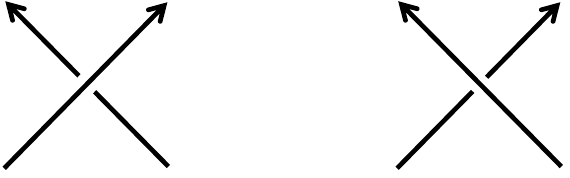} 
	\caption{The integer determined by the left (resp. the right) picture is 
 ``$1$" (resp. 
 ``$-1$").}       
	\label{fg-crossing}   
\end{figure}

Let $\alpha,\beta$ be two  $n$-web diagrams in $B_{\fS}$. Suppose $\gamma$ is the web diagram representing the product $\alpha\beta$, obtained by stacking $\alpha$ above $\beta$. Then define \begin{align}\label{i3}
i_3(\alpha,\beta) =\sum_{p}v(p) \in\mathbb Z_3
\end{align} where the sum is over all the crossing points of $\gamma$ between $\alpha$ and $\beta$.
It is known that $i_3(\alpha,\beta) $ is well-defined and $i_3(\alpha',\beta') = i_3(\alpha,\beta)$ if $\alpha$ (resp. $\beta$) and $\alpha'$ (resp. $\beta'$) represent isotopic  webs \cite{Wan24}. 

\begin{definition}[\cite{Wan24}]\label{def-sualgebra3}
	Define $\cS_{\bar q}(\fS)_3$ to be an $R$-submodule of $\cS_{\bar q}(\fS)$ spanned by
	\begin{equation}\label{condition}
			B_{\fS,3}:=\{\alpha\in B_\fS \mid i_3(\alpha,\beta) = 0\text{ for any  }\beta\in B_\fS\}.
	\end{equation}
	Then $\cS_{\bar q}(\fS)_3$ is a subalgebra of  $\cS_{\hat q}(\fS)$.
\end{definition}

In \cite{Wan24}, the author defined 
$\cS_{\bar q}(\fS)_3$ to be the $R$-submodule of $\cS_{\bar q}(\fS)$ generated by the web diagrams $\alpha$ such that $i_3(\alpha,\beta) = 0$
for any web diagram $\beta$. This definition is equivalent to Definition \ref{def-sualgebra3}, in which we used only the web diagrams in $B_\fS$ instead of all web diagrams, since the mod $3$ intersection number stays the same when we use isotopies and relations \eqref{w.cross}-\eqref{wzh.four} to express web diagrams as linear sums of basis elements.

\vspace{2mm}

This subalgebra $\cS_{\bar{q}}(\fS)_3$ will play a crucial role in the following subsection in the proof of the structure of the center of the skein algebra $\cS_{\bar{q}}(\fS)$ when $\bar{q}$ is a root of unity. Here we establish one key technical step on the characterization of the elements of the spanning set $B_{\fS,3}$ of $\cS_{\bar{q}}(\fS)_3$ among the elements of the basis $B_\fS$ of $\cS_{\bar{q}}(\fS)$, via the Douglas-Sun coordinate map $\kappa$ \eqref{DS_map}:

\begin{proposition}\label{lemgroup}
	Suppose $\fS$ is a triangulable punctured surface with a triangulation $\lambda$, and $\al\in B_\fS$.
 Then 
$$
\alpha\in B_{\fS,3} \quad\Leftrightarrow\quad
\kappa(\al) Q_\lambda\kappa(\beta)^T = 0\in\mathbb Z_{18} \mbox{ for any } \beta\in B_\fS.
$$
\end{proposition}

As a preliminary, we first show the following lemma:

\begin{lemma}\label{lem-inter3}
	Suppose that $\fS$ is a triangulable punctured surface with a 
 triangulation $\lambda$, and $\alpha
 \in B_\fS$. Then
 $$
 \kappa(\al) Q_\lambda \in (6\mathbb Z)^{V_\lambda}.
 $$
\end{lemma}

\begin{proof}[Proof of Lemma \ref{lem-inter3}.]
 Let $\kappa (\alpha)
 =(k_v)_{v\in V_\lambda} \in \Gamma_\lambda \in \mathbb{N}^{V_\lambda}$. 
 
 Suppose that the triangle $\tau\in\mathbb F_\lambda$ is as illustrated in Figure \ref{quiver}.
	Since $
 (k_v)_{v\in V_\gamma}\in \Gamma_\lambda$,  
 we have 
	$\sum_{i=1,2,3} 
 k_{v_{i2}}- \sum_{i=1,2,3} 
 k_{v_{i1}}= 0\in\mathbb Z_3$ (see Remark \ref{rem-deg}).
	Then equation \eqref{eq-d} implies that $
 (\kappa(\alpha) Q_\lambda)_v=0\in\mathbb Z_6$.

	Suppose that $e$ is an edge in $\lambda$ and $\tau,\tau'$ are two different triangles of $\lambda$ adjacent to $e$. The intersections $\mathcal T\cap (\tau\cap\tau')$ and 
	$\mathsf{H}_{\lambda} \cap (\tau\cap\tau')$ are as illustrated in Figure \ref{traintrack-quiver}. 
 From \eqref{eq-v31} and  \eqref{eq-v32} it follows that
 $
(\kappa(\alpha) Q_\lambda)_{v_{31}} = (\kappa(\alpha) Q_\lambda)_{v_{32}} =0\in\mathbb Z_6$.	
\end{proof}

\begin{proof}[Proof of Proposition \ref{lemgroup}.]
	We choose $R$ and $\hat q$ such that $\hat q^2$ is a root of unity of order $18$ (this is reasonable since we can choose $R$ to be $\mathbb C$). Then the order of $q^{\frac{2}{3}} = (\hat q^2)^6$ is $3$. Then relation \eqref{w.cross} implies
	\begin{equation}\label{eq-crossq}
		\raisebox{-.20in}{
			
			\begin{tikzpicture}
				\tikzset{->-/.style=
					
					{decoration={markings,mark=at position #1 with
							
							{\arrow{latex}}},postaction={decorate}}}
				\filldraw[draw=white,fill=gray!20] (-0,-0.2) rectangle (1, 1.2);
				\draw [line width =1pt,decoration={markings, mark=at position 0.5 with {\arrow{>}}},postaction={decorate}](0.6,0.6)--(1,1);
				\draw [line width =1pt,decoration={markings, mark=at position 0.5 with {\arrow{>}}},postaction={decorate}](0.6,0.4)--(1,0);
				\draw[line width =1pt] (0,0)--(0.4,0.4);
				\draw[line width =1pt] (0,1)--(0.4,0.6);
				\draw[line width =1pt] (0.4,0.6)--(0.6,0.4);
			\end{tikzpicture}
		}
		= q^{\frac {2}{3}}
		\raisebox{-.20in}{
			\begin{tikzpicture}
				\tikzset{->-/.style=
					
					{decoration={markings,mark=at position #1 with
							
							{\arrow{latex}}},postaction={decorate}}}
				\filldraw[draw=white,fill=gray!20] (-0,-0.2) rectangle (1, 1.2);
				\draw [line width =1pt,decoration={markings, mark=at position 0.5 with {\arrow{>}}},postaction={decorate}](0.6,0.6)--(1,1);
				\draw [line width =1pt,decoration={markings, mark=at position 0.5 with {\arrow{>}}},postaction={decorate}](0.6,0.4)--(1,0);
				\draw[line width =1pt] (0,0)--(0.4,0.4);
				\draw[line width =1pt] (0,1)--(0.4,0.6);
				\draw[line width =1pt] (0.6,0.6)--(0.4,0.4);
			\end{tikzpicture}
		},\quad
					\raisebox{-.20in}{
				\begin{tikzpicture}
					\tikzset{->-/.style=
						
						{decoration={markings,mark=at position #1 with
								
								{\arrow{latex}}},postaction={decorate}}}
					\filldraw[draw=white,fill=gray!20] (-0,-0.2) rectangle (1, 1.2);
					\draw [line width =1pt,decoration={markings, mark=at position 0.5 with {\arrow{>}}},postaction={decorate}](0.6,0.6)--(1,1);
					\draw [line width =1pt,decoration={markings, mark=at position 0.5 with {\arrow{>}}},postaction={decorate}](0.6,0.4)--(1,0);
					\draw[line width =1pt] (0,0)--(0.4,0.4);
					\draw[line width =1pt] (0,1)--(0.4,0.6);
					\draw[line width =1pt] (0.6,0.6)--(0.4,0.4);
				\end{tikzpicture}
			}	
				=q^{-\frac{2}{3}}
				\raisebox{-.20in}{
			
			\begin{tikzpicture}
				\tikzset{->-/.style=
					
					{decoration={markings,mark=at position #1 with
							
							{\arrow{latex}}},postaction={decorate}}}
				\filldraw[draw=white,fill=gray!20] (-0,-0.2) rectangle (1, 1.2);
				\draw [line width =1pt,decoration={markings, mark=at position 0.5 with {\arrow{>}}},postaction={decorate}](0.6,0.6)--(1,1);
				\draw [line width =1pt,decoration={markings, mark=at position 0.5 with {\arrow{>}}},postaction={decorate}](0.6,0.4)--(1,0);
				\draw[line width =1pt] (0,0)--(0.4,0.4);
				\draw[line width =1pt] (0,1)--(0.4,0.6);
				\draw[line width =1pt] (0.4,0.6)--(0.6,0.4);
			\end{tikzpicture}
		}.
	\end{equation}
 
 For $\alpha,\beta\in B_\fS$, we know $i_3(\alpha,\beta) =0,1,2\in\mathbb Z_3$. Since $q^{\frac{2}{3}}$ is a primitive root of unity of order $3$, then $q^{\frac{2}{3} i_3(\alpha,\beta)}$ makes sense.
	Equation \eqref{eq-crossq} implies  that
 $\alpha\beta = q^{\frac{2}{3} i_3(\alpha,\beta)}\beta\alpha\in\cS_{\bar q}(\fS)$. We know  $\cS_{\bar q}(\fS)$ is torsion-free because it is a free module over $R$. This implies that $\alpha\beta=\beta\alpha$ if and only if $q^{\frac{2}{3} i_3(\alpha,\beta)} = 1$.
 Thus, for $\alpha \in B_{\fS}$, we have
 \begin{align*}
 \alpha\in B_{\fS,3}
     \Leftrightarrow
     i_3(\alpha,\beta)=0\text{ for all }\beta\in B_\fS
     \Leftrightarrow \al\in\ZZ.
 \end{align*}

 So, it suffices to show that for $\al \in B_{\fS}$, the condition $\al \in \ZZ$ is equivalent to the condition $\kappa(\al) Q_\lambda\kappa(\beta)^T = 0\in\mathbb Z_{18}$.

 \vspace{2mm}
 
	Suppose first that $\al \in B_{\fS}$ satisfies $\al\in \ZZ$. 
 Then Lemma \ref{lem-leadingterm}\ref{lem_on_lt} implies that $\hat q^{2\kappa(\al) Q_\lambda\kappa(\beta)^T}=1$ holds for any $\beta \in B_{\fS}$. 
 So we have $\kappa(\al) Q_\lambda\kappa(\beta)^T = 0\in\mathbb Z_{18}$.

 Conversely, suppose now that $\alpha \in B_{\fS}$ satisfies $\kappa(\al) Q_\lambda\kappa(\beta)^T = 0\in\mathbb Z_{18}$ for any $\beta\in B_\fS$. 
	Theorem \ref{thm-trace} implies that
	$$\tr(\alpha) =
  x^{\kappa(\al)} + \sum_{{\bf k}\in\Lambda\subset \mathsf B_\lambda} c_{{\bf k}} x^{{\bf k}},$$
	where $\Lambda$ is a finite subset of $\mathsf{B}_\lambda$, and $c_{{\bf k}}\in R\setminus\{0\}$ for ${\bf k}\in\Lambda$.
	From the `congruence' property of the exponents as shown in  \cite[Proposition 5.76]{kim2011sl3} 
 we have ${\bf k}-\kappa(\al)\in (3\mathbb Z)^{V_\lambda}$ for all ${\bf k}\in\Lambda$. From Lemma \ref{lem-inter3} 
 we have $Q_{\lambda}\kappa(\beta)^T\in (6\mathbb Z)^{V_\lambda}$ (which follows from taking the transpose of $-\kappa(\beta) Q_\lambda \in (6\mathbb{Z})^{V_\lambda}$, as $Q_\lambda^T = -Q_\lambda$). Thus $({\bf k}-\kappa(\al))Q_{\lambda} \kappa(\beta)^T=0\in\mathbb Z_{18}$. This shows 
	${\bf k} Q_\lambda\kappa(\beta)^T = 0\in\mathbb Z_{18}$ for each ${\bf k}\in\Lambda$.
  Thus we have $\tr(\alpha) x^{\kappa(\beta)} = x^{\kappa(\beta)}  \tr(\alpha)\in \Xbl$, for every $\beta \in B_{\fS}$. So $\tr(\alpha) x^{{\bf t}} = x^{{\bf t}}  \tr(\alpha)\in \Xbl$ holds for all ${\bf t} \in \Gamma_\lambda = \kappa(B_{\fS})$, hence for all ${\bf t} \in \bar{\Gamma}_\lambda = \mathsf{B}_\lambda$ (Lemma \ref{lem-Gamma-bal}). 
	From the definition of $\Xbl$, we can conclude that $\tr(\al)\in\mathcal Z(\Xbl)$.
	Then the injectivity of the quantum trace map ${\rm tr}^X_\lambda$ shows $\alpha\in\ZZ$, as desired.
\end{proof}

\begin{remark}\label{rem-congruence}
    From Proposition \ref{lemgroup} and Lemma \ref{lem-inter3}, it follows that for $\alpha \in B_{\fS}$, we have
    $$
    \kappa(\alpha) \in (3\mathbb{Z})^{V_\lambda} ~ \Rightarrow ~ \alpha \in B_{\frak{S},3}.
   $$
That is, if a non-elliptic basis web $\alpha \in B_{\fS}$ satisfies the `congruence' condition (as in \cite[Definitions 3.38, 3.40]{kim2011sl3}), then $\alpha$ is in $B_{\fS,3}$. Thus, we may refer to the subalgebra $\cS_{\bar{q}}(\fS)_3$ as the `(mod 3) congruent' subalgebra of the skein algebra $\cS_{\bar{q}}(\fS)$. We expect that the converse of the above implication holds; since we do not need it in the present paper, we leave it to readers (see also \eqref{eq-Homology-group}).
\end{remark}

\subsection{On the case when the quantum parameter is a root of unity}
\def\Sfz{\cS_{\bar\omega} (\fS)}
\def\Spz{\cS_{\bar\omega}^{\circ} (\fS)}
\def\Zz{\mathcal Z(\cS_{\bar\omega} (\fS))}
\def\Sfe{\cS_{\bar\eta} (\fS)}
In this subsection, we will assume $R=\mathbb C$ and $\hat q=\hat \omega$ is a root of unity. 
 We have $\bar\omega = \hat \omega^{6}$ and $\omega=\hat\omega^{18}$ such that $\bar\omega^{\frac{1}{6}} = \hat\omega$ and $\omega^{\frac{1}{18}}=\hat\omega$. Suppose the order of $\hat \omega^2$ is $N''$.  
 Define $N'=\frac{N''}{\gcd(N'',6)}$ and $N= \frac{N'}{\gcd(N',3)}$.
 Then $N'$ (resp. $N$) is the order of $\bar\omega^2$ (resp. $\omega^2$). 
 Set $\hat\eta = \hat\omega^{N^2}$. Then we have $\bar\eta = \hat\eta^{6}=\bar\omega^{N^2}$ and
$\eta=\hat \eta^{18}=\omega^{N^2} = \pm 1$.  Define $$\im_{\bar\omega} \cF=\begin{cases}
	\cF (\Sfe) & \mbox{if } 3\nmid N',\\
	\cF (\Sfe_3) & \mbox{if } 3\mid N'.
\end{cases}$$ 
As shall be seen from now on, it is important to keep track of the dichotomy between the case when $3\nmid N'$ and when $3\mid N'$, which in particular justifies that in the latter case one must consider the congruent subalgebra $\Sfe_3$, not just $\Sfe$. A useful comment is that this dichotomy is directly related to the number $\dd \in \{1,3\}$ defined in \eqref{d_1_or_3} and appeared in Proposition \ref{prop-tran}, as we have
$$
\dd = {\rm gcd}(N',3).
$$

The following proposition states how the Frobenius map $\mathcal{F}$ of skein algebras provide `non-obvious' elements of the center of the skein algebra at a root of unity, in addition to the obvious central elements given by peripheral skeins (\eqref{peripheral_subalgebra}).
\begin{proposition}[\cite{HLW}]\label{prop-center}
 Let $\mathfrak{S}$ be a surface without boundary. Then we have 
	$$\im_{\bar\omega} \cF\subset \Zz.$$
\end{proposition}
Note that in Proposition \ref{prop-center}, $\fS$ can be a closed surface. In \S\ref{sec.independent_proofs} we present a proof of Proposition \ref{prop-center}, independently of \cite{HLW}.

The following third main theorem of our paper gives a description of the center of the skein algebra $\Sfz$ at a root of unity, resolving  
a conjecture in \cite{HLW} for the case of ${\rm SL}_3$-skein algebras, which we shall observe to be a modification of a conjecture of Bonahon and Higgins \cite[Conjecture 16]{bonahon2023central}.

\begin{theorem}[center of skein algebra at root of unity; Theorem \ref{thm-center-intro}]\label{thm-center} 
Let $\mathfrak{S}$ be a (connected) punctured surface. As a subalgebra of $\Sfz$, the center  $\Zz$ is generated by $\Spz$ (\eqref{peripheral_subalgebra}) and $\im_{\bar\omega} \cF$.
\end{theorem}

Part of the proof of this theorem closely follows 
our proof of Proposition \ref{prop-generic} which is about the center of the skein algebra at a generic quantum parameter.

\begin{proof}
	\def\Aa{A_{\bar\omega}}
	We use $A_{\bar\omega}$ to denote the the subalgebra of $\Sfz$ generated by $\Spz$ and $\im_{\bar\omega} \cF$.
	
	From Proposition \ref{prop-center} and \eqref{peripheral_in_center} we have $\Aa \subset \Zz$, so 
 it suffices to show that $\Zz\subset\Aa$.
		When $\fS$ is 
  a once or twice punctured sphere, this is obvious.

	Suppose that $\fS$ is triangulable with a triangulation $\lambda$. Suppose $0\neq X\in\Zz$. We will use induction on $\deg(X)\in \Gamma_\lambda \in \mathbb{N}^{V_\lambda}$ (Definition \ref{def-lt_and_deg}) to show $X\in\Aa.$ 
 For the base case when $\deg(X)={\bf 0}$, 
 we have $X=
 c \,\emptyset\in \Aa$, where $
 c \in \mathbb C\setminus\{0\}$, as desired. 
 Now, fix any ${\bf k}=(k_v)_{v\in V_\lambda}\in \Gamma_{\lambda}$. 
 As an induction hypothesis, assume that $X\in \Aa$ holds for every $X \in \Zz$ such that $\deg(X) <{\bf k}$.
	Let $Y$  be an element in $\Zz$ and $\deg(Y)={\bf k}$; our goal is to show that $Y\in \Aa$ holds. 
 From Definition \ref{def-lt_and_deg}, we can write 
 $$Y =c \kappa^{-1}({\bf k})+ \mathcal D_{<{\bf k}},$$ where $c\in \mathbb C\setminus\{0\}$.
	
From $Y\in \Zz$ and Lemma \ref{lem-leadingterm}\ref{lem_on_lt}, by a similar reasoning as we used for \eqref{kQa_zero_on_the_nose}, we get this time that
 \begin{align}
     \label{kQa_zeto}
 {\bf k} Q_{\lambda} {\bf a}^T=0\in\mathbb Z_{N''}
 \end{align}
 holds for any ${\bf a}\in\Gamma_\lambda$, and hence for any ${\bf a}\in\mathsf{B}_\lambda=\bar{\Gamma}_\lambda$.
 Since $3\mathbb Z^{V_\lambda}\subset \mathsf{B}_\lambda$, 
 we can deduce that
	$3{\bf k} Q_{\lambda}=0 \in\mathbb Z_{N''}$.
	Note that $\frac{1}{2} Q_\lambda$ is still an integer matrix. Then, in view of $N'=\frac{N''}{\gcd(N'',6)}$, we have 
	$\frac{1}{2}{\bf k} Q_{\lambda}=0 \in\mathbb Z_{N'}$.

We now apply the arguments in the proof of Proposition \ref{prop-generic}, which we developed based on the condition $\frac{1}{2} {\bf k} Q_\lambda=0$. Since we only have $\frac{1}{2}{\bf k} Q_{\lambda}=0 \in\mathbb Z_{N'}$ this time, the results should be modified accordingly.

Consider the weighted train track $(\mathcal{T},w) = f^{-1}({\bf k})$, where $f$ is the map in Lemma \ref{lem-iso-traintrack}. By using the same argument as we arrived at \eqref{3_w_d_zero}, we now have
$$
3w(s) = 0 \in \mathbb{Z}_{N'},
$$
for the 3-valent component $s$ of the train track $\mathcal{T} \cap \tau$ for each triangle $\tau \in \mathbb{F}_\lambda$.
 So, in view of $N= \frac{N'}{\gcd(N',3)}$, we have
 \begin{align}\label{eq-w-d-zero}
 w(s) = 0 \in \mathbb{Z}_N.
\end{align}

	Suppose that $e$ is an edge in $\lambda$ and $\tau,\tau'$ are two different triangles adjacent to $e$. The intersections $\mathcal T\cap (\tau\cap\tau')$ and 
	$\mathsf{H}_{\lambda} \cap (\tau\cap\tau')$ are as illustrated  in Figure \ref{traintrack-quiver}. 
 
We still have \eqref{eq-v31}--\eqref{eq-v32}, which are now $0$ as elements of $\mathbb{Z}_{N'}$. Since we also have $3w(s)=3w(s')=0 \in \mathbb{Z}_{N'}$, 
 in view of $N= \frac{N'}{\gcd(N',3)}$ 
 we get
	\begin{align}\label{eq-key}
		\text{$w(c_i)= w(c_i')\in\mathbb Z_N$ and $w(b_i) = w(b_i')\in\mathbb Z_N$ for $i=1,2$.}
	\end{align}
	 
	One can write $\kappa^{-1}({\bf k}) \in B_{\fS}$ as 
  $\kappa^{-1}({\bf k}) = \alpha P$, where $P$ is a product of peripheral skeins (or is empty) and $\alpha\in B_\fS$ contains no peripheral skeins.
  Consider the weighted train track $(\mathcal{T}',w') := g^{-1}(\alpha)$, where $g$ is as defined in \eqref{map-g}. Let ${\bf k}' = \kappa(\alpha) \in \Gamma_\lambda$, so that $(\mathcal{T}',w') = f^{-1}({\bf k}')$. It is convenient to note that $(\mathcal{T},w) = g^{-1}(\alpha P)$.  
  Then we can say that we have $\mathcal T'=\mathcal T$ since $P$ is a product in peripheral skeins, hence does not affect the degrees of honeycombs in triangles (see \eqref{w_of_d_and_degree_of_honeycomb} and the discussion after it). 
    We have ${\bf k} = {\bf k}' +\kappa(P)$ since $P$ consists of peripheral skeins and hence can be represented by web diagrams disjoint from $\alpha$; then apply the additivity property of $\kappa$ (Lemma \ref{kappa_additivity}). We also have $\kappa(P) Q_{\lambda} {\bf a}^T=0\in\mathbb Z_{N''}$ for any ${\bf a}\in\Gamma_\lambda$, which can be seen for example by the same reasoning as in \eqref{kQa_zeto}, because $P \in \Zz$ (\eqref{peripheral_in_center}).
This implies that 
\begin{align}
\label{k_prime_Q_lambda_a_zero}
{\bf k}' Q_{\lambda} {\bf a}^T=0\in\mathbb Z_{N''} \quad \mbox{holds for any ${\bf a}\in\Gamma_\lambda$.}
\end{align}
Then the above discussion for ${\bf k}$ which we developed after \eqref{kQa_zeto} applies to ${\bf k}'$. 
	 Thus equation \eqref{eq-w-d-zero} is true for ${\bf k}'$, $$w'(s) = 0\in\mathbb Z_{N},$$ and equation \eqref{eq-key} is true for $w'$.
	 
	 Suppose that $p$ is a puncture of $\fS$ and the train track $\mathcal T$ nearby 
  $p$ is as in 
  the picture in Figure \ref{puncture}. Then Lemma \ref{lem-loop} implies that
	 $w'(v_{t})=w'(u_j)=0$ holds for some $0\leq t,j\leq m$. Then equation \eqref{eq-key} for $w'$ shows 
	 $$\text{$w'(c_i)= w'(c_i')=w'(b_i) = w'(b_i')=0\in\mathbb Z_N$ for $i=1,2$.}$$
  Thus it follows that $w'(c) = 0\in\mathbb Z_N$ for any component $c$ of $\mathcal T$.
  So $w_1:=\frac{1}{N} 
  w'$ is a well-defined weight of $\mathcal{T}$. 
  Set $\alpha_1=g(\mathcal T, w_1)  \in B_\fS,\;{\bf k}_1 =\kappa(\alpha_1)=f(\mathcal{T},w_1)\in\Gamma_\lambda$. 
  Since ${\bf k}' = f(\mathcal{T},w')$, and since $f$ preserves the monoid structure (Lemma \ref{lem-iso-traintrack_monoids}), we have 
	 $N{\bf k}_1 = {\bf k}'$. In the meantime, from Lemma \ref{lem-D} 
  we have $\lt(\cF(\alpha_1)) = \lt(\alpha_1^N)$. 
  From Definition \ref{def-lt_and_deg} one can also deduce that $\lt(\cF(\alpha_1)P) = \lt(\cF(\alpha_1)) P = \lt(\alpha_1^N) P = \lt(\alpha_1^N P).$
	 Thus 
	 \begin{align}\label{eq-cent-deg}
	    \deg(\cF(\alpha_1)P) = \deg(\alpha_1^N P) = N\deg(\alpha_{1}) +\deg(P)
	    =N{\bf k}_1 +\kappa(P) = {\bf k} = \deg(Y).
	 \end{align}
  So, there exists a nonzero complex number $t$ such that 
	 $Y 
  \in t \cF(\alpha_1)P + \mathcal D_{<{\bf k}}.$

\vspace{1mm}

      If $3\nmid N'$, then we have $\cF(\alpha_1) P\in  \Aa\subset\Zz$. Thus $Y- t\cF(\alpha_1)P\in \Zz\cap \mathcal D_{<{\bf k}}\subset A_{\bar\omega},$ by the induction hypothesis.    This shows $Y\in \Aa$, as desired.

	Suppose now that $3\mid N'$. 
  From \eqref{k_prime_Q_lambda_a_zero} we have ${\bf k}' Q_{\lambda} \kappa(\beta)^T=0\in\mathbb Z_{N''}$ for any $\beta\in B_\fS$.
  Note from Lemma \ref{lem-inter3} that $\frac{1}{6}{\bf k}' Q_{\lambda} \kappa(\beta)^T$ is still an integer. Thus, in view of $N'=\frac{N''}{\gcd(N'',6)}$ it follows that
 $\frac{1}{6}{\bf k}' Q_{\lambda} \kappa(\beta)^T=0\in\mathbb Z_{N'}$ and hence $\frac{1}{6}{\bf k}_1 Q_{\lambda} \kappa(\beta)^T =0\in\mathbb Z_{3}$, because $3 \mid N'$. Now, by Proposition \ref{lemgroup}  
 we get $\al_1\in B_{\fS,3}$. This shows that $\cF(\alpha_1) P\in  \Aa\subset\Zz$, in view of the definition of $\Aa$ and $\im_{\bar\omega} \cF$. As 
 in the case $3 \nmid N'$, using the induction hypothesis we can deduce that  
 $Y\in\Aa$, as desired.
\end{proof}

In the following Remark, we compare the center in Theorem \ref{thm-center} and the center conjectured by Bonahon and Higgins 
\cite{bonahon2023central} in the case of ${\rm SL}_3$-skein algebras.

\begin{remark}

In \cite[Conjecture 16]{bonahon2023central} Bonahon and Higgins conjecture that 
    the center
$\Zz$ of $\Sfz$ coincides with the subalgebra $\mathcal{Z}'(\Sfz)$ of $\Sfz$ 
generated by peripheral skeins and $l^{[P_{N',1}]}$ for any web diagram $l$ in 
$\fS$ consisting of knots. 
When $N'=N$ (i.e. $3 \nmid N'$, so $\dd=1$), it  
is obvious that $\mathcal Z'(\Sfz)=\Zz$. 

When $N' = 3N$ (i.e. $3 \mid N'$, so $\dd=3$), we 
will show that $\mathcal Z'(\Sfz)\subset\Zz$.  
Note that $P_{3,1}(x_1,x_2) = x_1^3 - 3x_1x_2 +3$.
Lemma \ref{lem-additivity} implies that 
\begin{align*}
    P_{3,1}(P_{N,1}(x_1,x_2), P_{N,2}(x_1,x_2))
    =P_{N',1}(x_1,x_2).
\end{align*}
Suppose that $\alpha$ is a framed knot in the thickened surface $\widetilde{\fS}$. 
Note that $\alpha^{(3,0)}$ (see \S\ref{subsec.threading}) consists of three parallel copies of $\al$. So we write $\alpha^{(3,0)} = \al\cup\al\cup\al$, where the union is taken in a small regular open neighborhood of $\al$. Similarly, we write $\al^{(1,1)} = \al\cup \cev{\al}$ (see \S\ref{subsec.threading}).
We have\begin{align*}
& \cF(\al^{(3,0)}) - 3\,\cF(\al^{(1,1)}) + 3 \\
    =&\cF(\al\cup \al\cup \al) - 3\,\cF(\al\cup\cev{\al}) + 3\\
    = &\al^{[P_{N,1}]}\cup \al^{[P_{N,1}]}\cup \al^{[P_{N,1}]}
    - 3\,\al^{[P_{N,1}]}\cup\al^{[P_{N,2}]} + 3\\
    =& \al^{[P_{N,1}^3]} - \al^{[3 P_{N,1}P_{N,2}]} + \al^{[3]}\\
    =& \al^{[P_{N',1}]}.
\end{align*}
Note that we used the following straightforward fact: if $\alpha_1,\alpha_2$ are parallel copies of $\alpha$, then $\alpha_1 \cup \cev{\alpha_2} = \cev{\alpha_1} \cup \alpha_2 \in \mathscr{S}_{\bar{\omega}}(\fS)$.

It is easy to check that $\al^{(3,0)},\al^{(1,1)}\in\Sfe_3$.
This implies that $\al^{[P_{N',1}]}\in \Zz$. 
Suppose that $l =\cup_{1\leq i\leq m} l_i$, where each $l_i$ is a framed knot in $\widetilde{\fS}$.
Then we have 
$$l^{[P_{N',1}]} = \cup_{1\leq i\leq m} l_i^{[P_{N',1}]}
 = \prod_{1\leq i\leq m} l_i^{[P_{N',1}]}
 ,$$
 where the last equality 
 holds because each $l_i^{[P_{N',1}]}$ is `transparent' in $\Sfz$ (\cite[Theorem 13]{bonahon2023central}). 
 Since $l_i^{[P_{N',1}]} \in \Zz$, we have $l^{[P_{N',1}]} \in \Zz$, so
$\mathcal Z' (\Sfz)\subset \Zz$.

In the meantime, still for the case when $N' = 3N$ (i.e. $3 \mid N'$ and $r=3$), 
suppose that $\beta$ is a simple knot diagram in $\fS$ such that
$\beta$ is not a peripheral skein. Then $\beta^{(1,1)} = \beta \cev{\beta} \in \Sfe_3$, so from 
Theorem \ref{thm-center} 
we have 
$\cF(\beta\cev{\beta})
\in \Zz$.
It is not obvious 
whether the element $\mathcal{F}(\beta \cev{\beta}) = \beta^{[P_{N,1}]} {\cev{\beta}}^{[P_{N,1}]} = P_{N,1}(\beta,\cev{\beta}) P_{N,1}(\cev{\beta},\beta) = P_{N,1}(\beta,\cev{\beta}) P_{N,2}(\beta,\cev{\beta})$ belongs to $\mathcal{Z}'(\Sfz)$ or not.
\end{remark}

\section{The rank of the $\SL$-skein algebra over its center}\label{sec-rank}
The rank defined in Definition \ref{def-rank} is 
crucial to understand the representation theory of the affine almost Azumaya algebra. The square root of this rank equals the maximal dimension of irreducible representations of the corresponding algebra. The irreducible representation of this maximal dimension is uniquely determined by a point in the Maximal ideal Spectrum of the center.
In this section, we will calculate this rank for $\SL$-skein algebras when $R=\mathbb C$ and $\bar q=\bar \omega$ is a root of unity. Please refer to \cite{frohman2021dimension,korinman2021unicity,yu2023center} for related works for ${\rm SL}_2$, and 
to \cite{HW}  for general stated ${\rm SL}_n$ skein algebras of
pb surfaces that have punctured boundary.

In 
this section, we will assume
$R=\mathbb C$, $\bar q=\bar\omega\in\mathbb C$ is a root of unity. 
 We have  $\omega=\bar\omega^{3}$, so that $\omega^{\frac{1}{3}}=\bar\omega$.
 Suppose that the order of $\bar\omega^2$ is $N'$.
Define $N=\frac{N'}{\gcd(N',3)}$. Then $N$ is the order of $\omega^2$.
Set $\bar\eta = \bar\omega^{N^2}$. Then we have  
$\eta=\bar \eta^{3}=\omega^{N^2} = \pm 1$. 

We will assume that all 
surfaces are connected. 

As in the previous section, it is important to keep track of the dichotomy between the case when $3\nmid N'$ and $3\mid N'$.

\subsection{A basis for $\Zz$} 
Theorem \ref{thm-center} 
describes the center of
$\Sfz$ when $\fS$
is a punctured surface. We will 
give a basis for this center in this subsection. 
Define the following subsets of the non-elliptic web basis $B_{\fS}$ (Definition \ref{def-B_S}):
\begin{align*}
	B_\fS^*=&\{\alpha\in B_\fS\mid\alpha\text{ does not contain peripheral skeins.}\}\\
	B_\fS^{\circ}=&\{\beta\in B_\fS\mid\beta\text{ consists of peripheral skeins.}\}
\end{align*}
We start with a simple technical lemma.
\begin{lemma}\label{lem-alb}
	Suppose that $\alpha_1,\alpha_1\in 	B_\fS^*$ and $\beta_1,\beta_2\in B_\fS^{\circ}$ 
 satisfy  $$\kappa(\alpha_1)+\kappa(\beta_1) = \kappa(\alpha_2)+\kappa(\beta_2).$$ Then we have $\alpha_1=\alpha_2$ and $\beta_1=\beta_2$. In particular, we have 
$\kappa(\alpha_1)=\kappa(\alpha_2)$ and $\kappa(\beta_1)=\kappa(\beta_2)$.
\end{lemma}
\begin{proof}
	Since a peripheral skein can be isotoped to be disjoint from any given web, one can apply the additivity of the coordinate map $\kappa$ (Lemma \ref{kappa_additivity}) to get 
	$$\kappa(\alpha_1\cup\beta_1) = \kappa(\alpha_1)+\kappa(\beta_1) = \kappa(\alpha_2)+\kappa(\beta_2) = \kappa(\alpha_2\cup\beta_2).$$
	Then we have $\alpha_1\cup\beta_1 = \alpha_2\cup\beta_2\in B_\fS$. This means that 
	$\alpha_1\cup\beta_1$ and $\alpha_2\cup\beta_2$  represent isotopic webs in $\widetilde{\fS}$. This happens only when $\alpha_1=\alpha_2$ and $\beta_1=\beta_2$. 

\end{proof}

\def\Bz{B_{\bar\omega}}
\def\Gz{\Gamma_{\bar\omega}}

Define
\begin{align*}
	B_{\bar\omega}=\begin{cases}
		\{(\alpha,\beta)\mid \alpha\in B_\fS^*,\beta\in B_\fS^{\circ}\} & \mbox{if } 3\nmid N',\\
		\{(\alpha,\beta)\mid \alpha\in B_\fS^*\cap B_{\fS,3},\beta\in B_\fS^{\circ}\} & \mbox{if } 3\mid N',
	\end{cases}
\end{align*}
where $ B_{\fS,3}$ is defined in equation \eqref{condition}.
It is not hard to see that
\begin{align}
\label{B_fS_circ_in_B_fS_3}
    B_\fS^{\circ} \subset B_{\fS,3};
\end{align}
for $\alpha \in B_{\fS}^\circ$ and $\beta \in B_\fS$, since $\alpha$ can be isotoped to be disjoint from $\beta$, it follows that the mod 3 intersection number $i_3(\alpha,\beta)$ defined in \eqref{i3} is zero. 
For any $(\alpha,\beta)\in B_{\bar\omega}$, define
$$
{\bf F}(\al,\beta) := \cF(\al)\beta\in\Sfz,
$$
where $\alpha$ (resp. $\beta$) is regarded as an element in $\Sfe$ (resp. $\Sfz$), and $\mathcal{F}: \Sfe \to \Sfz$ is the Frobenius map for skein algebras (Theorem \ref{Fro-surface}).
From Theorem \ref{thm-center}  
it follows that
${\bf F}$ is a map from $\Bz$ to $\Zz$:
$$
{\bf F} : \Bz \to \Zz.
$$

Suppose that $\fS$ is triangulable, with a triangulation $\lambda$.
Define
$$
\Gamma_{\lambda,3} = \kappa(B_{\fS,3}).
$$
Proposition \ref{lemgroup}, which provides linear equations characterizing the subset $B_{\fS,3}$ of $B_\fS$, implies the following.
\begin{corollary}
	$\Gamma_{\lambda,3}$ is a submonoid of $\Gamma_\lambda$.
\end{corollary}

\def\kappaz{\kappa_{\bar{\omega}}}

Define the following subset of $\Gamma_\lambda$ which captures the degrees of the leading terms of the images of elements of $\Bz$ under the map ${\bf F}$:
\begin{align}\label{eq-GGz}
	\Gamma_{\bar\omega}=\begin{cases}
		N\Gamma_\lambda + \kappa(B_\fS^{\circ}) & \mbox{if } 3\nmid N',\\
		N\Gamma_{\lambda,3} + \kappa(B_\fS^{\circ}) & \mbox{if } 3\mid N'.
	\end{cases}
\end{align}
One can observe that $\Gamma_{\bar\omega}$ is a submonoid of $\Gamma_\lambda$.
We can establish a correspondence between $\Bz$ and $\Gz$ by the map 
$$
\kappaz: \Bz \to \Gz,
$$
 defined as $$
\kappaz (\alpha,\beta) = N\kappa(\alpha) + \kappa(\beta) \quad\mbox{for all $(\alpha,\beta) \in \Bz$}.
$$
\begin{lemma}\label{lem-dd}
	The map $
 \kappaz \colon\Bz\rightarrow\Gz$ is a bijection.
\end{lemma}
\begin{proof}
 
 Let's first show that the map $\kappaz$ is surjective.  A possibly non-trivial part is to show that $N\Gamma_{\lambda,3} \subset N \kappa(B^*_\fS \cap B_{\fS,3}) + \kappa(B^\circ_\fS)$. Any element of $N \Gamma_{\lambda,3}$ is of the form $N \kappa(\alpha)$ for some $\alpha \in B_{\fS,3}$. One can write $\alpha = \alpha_1 \beta_1$ with some $\alpha_1 \in B^*_\fS$ and $\beta_1 \in B^\circ_\fS$. Since $\alpha_1$ and $\beta_1$ are disjoint, for every $\beta \in B_\fS$ we have $i_3(\alpha,\beta) = i_3(\alpha_1,\beta) + i_3(\beta_1,\beta)$. As noted in \eqref{B_fS_circ_in_B_fS_3} we have $i_3(\beta_1,\beta)=0$. Since $i_3(\alpha,\beta)=0$ because $\alpha \in B_{\fS,3}$, it follows that $i_3(\alpha_1,\beta)=0$, which means $\alpha_1 \in B_{\fS,3}$, thus $\alpha_1 \in B^*_\fS \cap B_{\fS,3}$. Hence, by additivity of $\kappa$ (Lemma \ref{lem-additivity}), we have  $N\kappa(\alpha) = N\kappa(\alpha_1)+N\kappa(\beta_1) \in N\kappa(B^*_\fS \cap B_{\fS,3}) + N \kappa(B^\circ_\fS) \subset N\kappa(B^*_\fS \cap B_{\fS,3}) +  \kappa(B^\circ_\fS)$, as desired.

	Let's show that ${\rm d}$ is injective. Suppose that $d(\alpha_1,\beta_1) = d(\al_2,\beta_2)$ holds for some $(\al_1,\beta_1),(\al_2,\beta_2)\in \Bz$.  
 Define $\gamma_i:=\kappa^{-1}(N\kappa(\alpha_i))\in B_{\fS}$ 
 for $i=1,2$, so that $\kappa(\gamma_i) = N\kappa(\alpha_i)$.
 Define a weighted train track $(\mathcal T_i,w_i):= f^{-1}(\kappa(\al_i))
 $ for $i=1,2$, where $f$ is the map in Lemma \ref{lem-iso-traintrack}.
	For the puncture $p$ in Figure \ref{puncture}, Lemma \ref{lem-loop} implies that
	$w_1(v_i)= w_1(u_j) = 0$ holds for some $1\leq i,j\leq m$.
	We have $
 g^{-1}(\gamma_1)=f^{-1}(N\kappa(\al_1)) = (\mathcal T_1,Nw_1)$, where the map $g$ is as in \eqref{map-g}. 
	Since $Nw_1(v_i)= Nw_1(u_j) = 0$
 , from Lemma \ref{lem-loop} it follows that  $\gamma_1=g(\mathcal{T}_1,Nw_1)$ does not contain peripheral skeins. Similarly, we 
 can see that $\gamma_2$ does not contain peripheral skeins.
 Note that $\kappaz(\alpha_1,\beta_1) = 
 \kappaz(\al_2,\beta_2)$ implies that $\kappa(\gamma_1) + \kappa(\beta_1) = \kappa(\gamma_2) + \kappa(\beta_2)$. Then, by  Lemma \ref{lem-alb}  
 we get $\kappa(\gamma_1)=\kappa(\gamma_2)$ and 
	$\kappa(\beta_1) =  \kappa(\beta_2)$. This shows  that 
	$\kappa(\al_1)=\kappa(\al_2)$ and 
	$\kappa(\beta_1) =  \kappa(\beta_2)$, and therefore we have $\al_1=\al_2$ and $\beta_1=\beta_2$, as desired.
\end{proof}
	
Recall 
the degree map $\deg \colon \Sfz\setminus\{0\}\rightarrow\Gamma_\lambda$  
from Definition \ref{def-lt_and_deg}.

\begin{lemma}\label{lem-d1}
The following hold:
\begin{enumerate}[label={\rm (\alph*)}]\itemsep0,3em
	\item\label{lem-d1-a} For any $(\alpha,\beta)\in B_{\bar\omega}$, we have
	$\deg(
 {\bf F}(\al,\beta)) = \kappaz(\al,\beta)\in\Gamma_{\bar\omega}$.
	
	\item\label{lem-d1-b} The map $
 {\bf F}: \Bz\rightarrow \Zz$ is injective, and   $\im 
 {\bf F}$ is a basis for $\Zz$.
 \end{enumerate}
\end{lemma}
\begin{proof}
	\ref{lem-d1-a} 
 follows from equation \eqref{eq-cent-deg}, so it remains to show \ref{lem-d1-b}.
	
Lemma \ref{lem-dd} and	(a) imply that different elements in $\im 
{\bf F}$ have different degrees. So $
{\bf F}$ is injective, and elements in
 $\im 
 {\bf F}$ are linearly independent. It now suffices to show that 
 $\im 
 \bf F$ spans $\Zz$.
 
 First let $3 \nmid N'$. Using Theorem \ref{thm-center} and the fact that the Frobenius map $\mathcal{F}$ is an algebra homomorphism, one can deduce that $\Zz$ is spanned by elements $\cF(\alpha)\beta$ for $\al\in B_\fS$ and $\beta\in B_\fS^{\circ}$. Write $\alpha = \alpha_1 \beta_1$ for $\alpha_1 \in B_\fS^*$ and $\beta_1 \in B_\fS^\circ$. Then $\mathcal{F}(\alpha) \beta = \mathcal{F}(\alpha_1) \mathcal{F}(\beta_1) \beta$, where $\mathcal{F}(\beta_1) \beta$ is in the span of $B_\fS^\circ$. Hence $\mathcal{F}(\alpha) \beta$ lies in the span of $\im{\bf F}$, as desired.

 Now let $3 \mid N'$. By Theorem \ref{thm-center} and Definition \ref{def-sualgebra3}, $\Zz$ is spanned by elements $\mathcal{F}(\alpha) \beta$ for $\al \in B_{\fS,3}$ and $\beta \in B_\fS^\circ$. If we write $\alpha = \alpha_1 \beta_1$ with $\alpha_1 \in B_\fS^*$ and $\beta_1 \in B_\fS^\circ$, we have $\alpha_1 \in B_\fS^* \cap B_{\fS,3}$ as we saw in the proof of Lemma \ref{lem-dd}. Then we can see that $\cF(\alpha) \beta = \cF(\al_1) \cF(\beta_1) \beta$ is in the span of $\rm {\bf F}$, as desired.
\end{proof}

We use $B_{\mathcal Z,\bar\omega}$ to denote $\im 
{\bf F}$. Lemmas \ref{lem-dd} and \ref{lem-d1} imply the following:
\begin{lemma}\label{lem-cd}
One has:
\begin{enumerate}[label={\rm (\alph*)}]\itemsep0,3em
	\item\label{lem-cd-a} $B_{\mathcal Z,\bar\omega}$ is a basis for $\Zz$.
	
  \item\label{lem-cd-b} The degree map  
  $\deg \colon \Sfz\setminus\{0\}\rightarrow\Gamma_\lambda$ restricts to a bijection from $B_{\mathcal Z,\bar\omega}$ to $\Gamma_{\bar\omega}$.
  \end{enumerate}
\end{lemma}

\subsection{On the rank of $\Sfz$ over $\Zz$}
We will use the techniques in \cite{frohman2021dimension} to calculate the rank
of $\Sfz$ over $\Zz$, denoted as $\text{rank}_{\mathcal Z} \Sfz$.

\def\ran{{\rm rank}_{\mathcal Z} \Sfz}
\def\Kz{K_{\bar\omega}}
\def\bG{\bar\Gamma_\lambda}

In the following, we will assume that $\fS$ is a connected punctured surface, with:
$$
\mbox{$\fS$ is of genus $g$ and has $n>0$ punctures.}
$$
Suppose that $\fS$ is triangulable  with a triangulation $\lambda$.
Define 
$$
\mbox{$\Gamma^{\circ} = \kappa(B_{\fS}^{\circ})$ and $\Gamma^{*} = \kappa(B_{\fS}^{*})$.}
$$
Recall that $\mathsf{B}_\lambda = \bar\Gamma_\lambda$ (Lemma \ref{lem-Gamma-bal}).
Define $\bar \Gamma^{\circ}$ (resp. $\bar \Gamma_{\lambda,3}$) to be the subgroup of $\bG$ generated by $\Gamma^{\circ}$ (resp. $\Gamma_{\lambda,3}$).

From \cite[Propositions 3.30 and 3.34]{kim2011sl3} which provides a bijection between $\bar{\Gamma}_\lambda$ with the set of all `${\rm SL}_3$-laminations' in $\mathfrak{S}$ via the natural extension of the coordinate map $\kappa$, and from the additivity of $\kappa$ (Lemma \ref{lem-additivity}, or \cite[Lemma 3.32]{kim2011sl3}), we have:
\begin{corollary}\label{cor-key}
	Any element in $\bar\Gamma_{\lambda}$ can be written in the form ${\bf k}+{\bf t}$, where ${\bf k}\in \Gamma^{*}$ and ${\bf t}\in \bar \Gamma^{\circ}$.
\end{corollary}

The following lemma is an ${\rm SL}_3$-analog of \cite[Proposition 3.11]{frohman2021dimension}; we note that the proof of the item \ref{summand-a} is same as in \cite{frohman2021dimension}, but the item \ref{summand-b} does not immediately follow from the arguments in \cite{frohman2021dimension}.
\begin{lemma}
\label{summand}
The following hold:
\begin{enumerate}[label={\rm (\alph*)}]\itemsep0,3em
	\item\label{summand-a} The subgroup $\bar \Gamma^{\circ}$ is a direct summand of $\bG$.
	
	\item\label{summand-b} The quotient $\bG/\bar \Gamma_{\lambda,3}$ is isomorphic to $\mathbb Z_3 ^{2g}$.
 \end{enumerate}
\end{lemma}
\begin{proof}
(a) The group $\bar \Gamma^{\circ}$ is a direct summand of $\bG$ if and only if it is {\em primitive} in the sense that 
 if $m x\in \bar \Gamma^{\circ}$, where $m$ is a positive integer and $x\in \bG$, then $x\in \bar \Gamma^{\circ}$.
Lemma \ref{lem-alb} implies that any element $x\in \Gamma_\lambda$ has a unique form
\begin{align}\label{eq-decomposition-x}
    x= x^{\circ} + x^{*},
\end{align}
where $x^{\circ}\in \Gamma^{\circ}$ and $x^{*}\in \Gamma^{*}$.
Suppose $x=x_1 - x_2\in \bar\Gamma_\lambda$, where $x_1,x_2\in\Gamma_\lambda$, such that 
$mx\in \bar \Gamma^{\circ}$ for some positive integer $m$. 
Then $mx = mx_1 - mx_2 = y_2 - y_1$, where $y_1,y_2\in \Gamma^{\circ}$.
Using the decomposition in \eqref{eq-decomposition-x}
for $x_1,x_2$, we have 
$$(m x^{\circ}_1 + y_1) + mx^{*}_1  =( mx^{\circ}_2 +y_2)+ mx^{*}_2,$$
with $x_i^\circ \in \Gamma^\circ$ and $x_i^* \in \Gamma^*$, $i=1,2$. Since $m x^{\circ}_1 + y_1,m x^{\circ}_2 + y_2\in \Gamma^{\circ}$ and $mx^{*}_1,mx^{*}_2\in \Gamma^{*}$, Lemma \ref{lem-alb} implies that $mx_1^* = mx_2^*$, hence we have $x_1^* = x_2^*$.
 Then we 
 get
 $$x=x_1 - x_2 =x_1^{\circ} - x_2^{\circ}\in \bar\Gamma^{\circ}.$$

 (b)
 Note that any element $\al\in B_{\fS}$ can be regarded as an element in $H_1(\overline{\fS},\mathbb Z_3)$, where $\overline{\fS}$ is a closed surface.
 So there is a map $h_1\colon B_{\fS}\rightarrow H_1(\overline{\fS},\mathbb Z_3)$. 
 Observe that $H_1(\overline{\fS},\mathbb Z_3)$ is $\mathbb Z_{3}$-linearly spanned by $\im h_1$.
It is well-known that the (signed) intersection form on $H_1(\overline{\fS},\mathbb{Z}_3)$ is 
a non-degenerate $\mathbb Z_3$-bilinear form, which can be written as:
$$(\cdot,\cdot)\colon H_1(\overline{\fS},\mathbb Z_3)\times H_1(\overline{\fS},\mathbb Z_3)\rightarrow\mathbb Z_3;\quad
(h_1(\alpha),h_1(\beta))\mapsto i_3(\al,\beta).$$
Then we have 
\begin{align}\label{eq-Homology-group}
    h_1(\alpha) = 0\in H_1(\overline{\fS},\mathbb Z_3)\Leftrightarrow \al\in B_{\fS,3}.
\end{align}

 We use $h_2$ to denote the composition $\Gamma_\lambda\xrightarrow{\kappa^{-1}}
 B_{\fS}\xrightarrow{h_1} H_1(\overline{\fS},\mathbb Z_3)$.
 For any ${\bf k},{\bf t}\in\Gamma_\lambda$, Lemma 
 \ref{lem-fil} shows that $\kappa^{-1}({\bf k}+{\bf t})$
 is obtained from $\kappa^{-1}({\bf k})\kappa^{-1}({\bf t})$
 by isotopies and applying relations \eqref{w.cross}-\eqref{wzh.four}.
 Since isotopies and applying relations \eqref{w.cross}-\eqref{wzh.four} are preserved in $H_1(\overline{\fS},\mathbb Z_3)$, then $h_2({\bf k} + {\bf t}) = h_2({\bf k}) + h_2({\bf t})$. This implies that $h_2$ extends to a surjective group homomorphism
 $$h_3\colon \bar\Gamma_\lambda\rightarrow H_1(\overline{\fS},\mathbb Z_3).$$

 We will show that $\ker h_3 = \bar \Gamma_{\lambda,3}$, which completes the proof because $H_1(\overline{\fS},\mathbb Z_3)\simeq \mathbb Z_3^{2g}.$
 Equation \eqref{eq-Homology-group} implies that 
 $\bar \Gamma_{\lambda,3}\subset \ker h_3$.
 For any ${\bf a}\in \ker h_3$, we suppose 
 ${\bf a} = {\bf a}' - {\bf a}''$, where ${\bf a'},{\bf a''}\in\Gamma_\lambda$. 
 We have 
 $$h_3({\bf a}) = h_3({\bf a}') - h_3({\bf a}'')
 = h_3({\bf a}') +2 h_3({\bf a}'') = 
 h_3({\bf a}'+2{\bf a}'') =0.$$
 Then equation \eqref{eq-Homology-group} implies that 
 ${\bf a}'+2{\bf a}''\in \Gamma_{\lambda,3}$.
 Proposition \ref{lemgroup} and Lemma \ref{lem-inter3} imply that $3{\bf a}''\in\Gamma_{\lambda,3}$. 
 Then we have ${\bf a} = {\bf a}'+2{\bf a}'' - 3{\bf a}''
 \in\bar\Gamma_{\lambda,3},$ proving $\ker h_3 \subset \bar{\Gamma}_{\lambda,3}$.

\end{proof}

\begin{remark}\label{rem-rank}
	Since $(3\mathbb Z)^{V_\lambda}\subset \bar\Gamma_\lambda\subset \mathbb Z^{V_\lambda}$, 
 it follows that $\bar\Gamma_\lambda$ is a free abelian group of rank
	$|V_\lambda|$. Using the Euler characteristic of $\fS$, we can get $|V_{\lambda}| = 16 g - 16 + 8
 n$. From Lemma \ref{lem-ll} 
 one can show that 
	 $\bar \Gamma^{\circ}$ is a free abelian group of rank $2
  n$.
\end{remark}

Denote by $\bar\Gamma_{\bar\omega}$ the subgroup of $\bG$ generated by $\Gamma_{\bar\omega}$ (\eqref{eq-GGz}).

\begin{lemma}\label{lem-equal}
 Let $\fS$ be a punctured surface of genus $g$ with $n>0$ punctures. Assume that $\fS$ is triangulable, with a triangulation $\lambda$.
	Then we have $$\left|\frac{\bG}{\bar\Gamma_{\bar\omega}}\right| =K_{\bar\omega},$$ 
 where
 \begin{align}\label{eq-K}
K_{\bar\omega}=\begin{cases}
N^{16g - 16 + 6
n} & \mbox{if } 3\nmid N', \\
3^{2g} N^{16g - 16 + 6
n} & \mbox{if } 3\mid N'.
\end{cases}
\end{align}
\end{lemma}
\begin{proof}
	Here we use the technique in the proof of \cite[Proposition 3.12]{frohman2021dimension}. Their proof almost works here. We only prove the case when $3\nmid N'$.
 For the case when $3\mid N'$, please refer to the proof of 
 \cite[Proposition 3.12]{frohman2021dimension}.
	We have 
	$$\frac{\bG}{\bar\Gamma_{\bar\omega}} = \frac{\bG}{N\bG + \bar \Gamma^{\circ}}\simeq  \frac{\bG/\bar \Gamma^{\circ}}{(N\bG + \bar \Gamma^{\circ})/\bar \Gamma^{\circ}}\simeq \frac{\bG/\bar \Gamma^{\circ}}{N(\bG/\bar \Gamma^{\circ})}.$$
	Lemma \ref{summand} and Remark \ref{rem-rank} imply that
	$\bG/\bar \Gamma^{\circ}$ is a free abelian group of rank $16g - 16 + 6n$.
	Thus $\left|\frac{\bG}{\bar\Gamma_{\bar\omega}}\right| = N^{16g - 16 + 6n}.$
\end{proof}

Note that the degree map $\deg \colon \Sfz\setminus\{0\}\rightarrow\Gamma_\lambda$ is surjective. 
We use $\deg_{\bar\omega}$ to denote the composition 
$\Sfz\setminus\{0\}\rightarrow\Gamma_\lambda\rightarrow\bG\rightarrow \bG/\bar\Gamma_{\bar\omega}$. Then we have the following.
\begin{lemma}\label{lem-de2}
	The degree map $\deg_{\bar\omega}\colon \Sfz\setminus\{0\}\rightarrow \bG/\bar\Gamma_{\bar\omega}$ is a surjective monoid homomorphism.
\end{lemma}
\begin{proof}
	Lemma \ref{lem-leadingterm}\ref{lem_on_deg} implies that $\deg_{\bar\omega}$ is a monoid homomorphism. Obviously, we have $\Gamma_\lambda/\bar\Gamma_{\bar\omega}\subset\im (\deg_{\bar\omega})$, where $\Gamma_\lambda/\bar\Gamma_{\bar\omega}$ is a subset of $\bG/\bar\Gamma_{\bar\omega}$. 
 Since $\bG/\bar\Gamma_{\bar\omega}$ is finite, from \cite[Lemma 3.8]{frohman2021dimension} 
 it follows that $\Gamma_\lambda/\bar\Gamma_{\bar\omega} = \bG/\bar\Gamma_{\bar\omega}$. This completes the proof.
\end{proof}

\begin{lemma}\label{lem-non-zero-rank}
    Let $W_1,\cdots, W_m\in \Sfz\setminus\{0\}$ such that 
    $\deg_{\bar\omega}(W_i)$, $1\leq i\leq m$, are pairwise distinct. Then we have $0\neq \sum_{1\leq i\leq m} W_i\in\Sfz$.
\end{lemma}
\begin{proof}
   From the definition of $\deg_{\bar\omega}$, we have 
   $\deg(W_i)$, $1\leq i\leq m$, are pairwise distinct.
   This implies that $0\neq \sum_{1\leq i\leq m} W_i$.
\end{proof}

\begin{corollary}\label{cor-lower}
	We have $\ran\geq \left| \bG/\bar\Gamma_{\bar\omega}\right|$.
\end{corollary}
\begin{proof}
    
    Here we use the technique in \cite[Corollary 5.2]{frohman2021dimension}. 
    Suppose that $\bG/\bar\Gamma_{\bar\omega} = \{x_1,\cdots,x_m\}$,
    where $m= 
    |\bG/\bar\Gamma_{\bar\omega}|$.
    From Lemma \ref{lem-de2}, there exist $W_1,\cdots,W_m\in \Sfz\setminus\{0\}$ such that $\deg_{\bar\omega}(W_i) = x_i$
    for $1\leq i\leq m$.

   Suppose that $z_1,\cdots,z_m\in \mathcal Z(\Sfz)$ are not all zero. If $z_i\neq 0$, then we have 
   $$\deg_{\bar\omega}(z_i W_i) =\deg_{\bar\omega}(z_i)+\deg_{\bar\omega}(W_i) = \deg_{\bar\omega}(W_i)=x_i.$$
   Here we used $\deg_{\bar\omega}(z_i)=0$; indeed, by Lemma \ref{lem-d1}\ref{lem-d1-b} $z_i$ is a $\mathbb{C}$-linear combination of elements ${\bf F}(\alpha,\beta)$ for some $(\alpha,\beta) \in B_{\bar{\omega}}$, and by Lemma \ref{lem-d1}\ref{lem-d1-a} we have $\deg({\bf F}(\alpha,\beta)) \in \Gamma_{\bar\omega}$, so that it follows that $\deg(z_i) \in \Gamma_{\bar{\omega}}$.
   Since $z_1W_1,\cdots,z_mW_m$ are not all zero, Lemma \ref{lem-non-zero-rank} implies that
   $z_1W_1+\cdots+z_mW_m\neq 0$.
   This shows that $W_1,\cdots, W_m\in\Sfz$ are linearly independent over $\Zz$.
   So we have 
   $\ran\geq m = 
   |\bG/\bar{\Gamma}_{\bar\omega}|.$
\end{proof}

Now our goal is to show $\ran\leq \left| \bG/\bar\Gamma_{\bar\omega}\right|$. We will eventually use \cite[Lemma 2.3]{frohman2021dimension}, and for that we need to build some preliminaries.

An abelian subgroup $\Lambda\subset\mathbb R^n$, of rank $n$, is called a lattice 
A convex polyhedron $Q$ is the convex hull of a finite
number of points in $\mathbb R^n$ and  its $n$-dimensional volume is denoted by
$
{\rm vol}_n(Q)$.
For any $k\in\mathbb R$, define $kQ =\{kx\mid x\in Q\}$.
\begin{lemma}\cite[Lemma 2.4]{frohman2021dimension}\label{lem-dim}
	Suppose that $\Lambda$ and $\Gamma\subset\Lambda$ are lattices in $\mathbb R^n$ and $Q\subset \mathbb R^n$
	is the
	union of a finite number of convex polyhedra with ${\rm vol}_n(Q)>0$. Let $u$ be
	a positive integer. One has
    $$\lim_{k\rightarrow\infty} \frac{|\Lambda\cap k Q|}{|\Gamma\cap (k-u) Q|}
    =|\Lambda/\Gamma|.$$
\end{lemma}

We define a special simplex $Q\subset \mathbb R^{V_\lambda}$ as follows. An element ${\bf k}=(k_v)_{v\in V_\lambda}\in \mathbb R^{V_\lambda}$ is in $Q$ if and only if 
all of the following conditions hold: $k_v\geq 0$,
$\sum_{v\in V_\lambda} k_v\leq 1$, and for each triangle $\tau\in \mathbb{F}_\lambda$ as  illustrated in Figure \ref{quiver}, the values in equation \eqref{eq-kkkk} are non-negative.

\begin{lemma}\label{Lme}
	We have $
 {\rm vol}_n(Q)>0$.
\end{lemma}
\begin{proof} 
 Define ${\bf t}=(t_v)_{v\in V_\lambda}\in Q$ such that, for each triangle $\tau\in \mathbb{F}_\lambda$ as illustrated in Figure \ref{quiver}, we have  $
 t_v = \frac{1}{2
 |V_\lambda|}$ and $
 t_{v_{ij}}= \frac{1}{3|V_\lambda|}$ for $1\leq i\leq 3$ and $1\leq j\leq 2$. Then we have 
	$$\{{\bf k}=(k_v)_{v\in V_\lambda}\in \mathbb R^{V_\lambda} 
 ~:~ | 
 k_v - 
 t_v|\leq \frac{1}{24|V_\lambda|}\text{ for any }v\in V_\lambda\}\subset Q. 
 $$
	This completes the proof. 
\end{proof}

We shall apply Lemma \ref{lem-dim} for $\Lambda = \bG$, $\Gamma = \bar{\Gamma}_{\bar\omega}$ and the above $Q$. We will also need the following:

\begin{lemma}\label{lem-Gammas_and_bar_Gammas}
	For any positive integer $k$, we have 
	$$\Gamma_\lambda\cap k Q = \bar\Gamma_\lambda\cap k Q,\quad \Gamma_{\bar\omega}\cap k Q = \bar\Gamma_{\bar\omega}\cap k Q.$$
\end{lemma}
\begin{proof}
	The definition of our $Q$ implies that $\Gamma_\lambda\cap k Q = \bar\Gamma_\lambda\cap k Q$.
	
	Obviously, we have $\Gamma_{\bar\omega}\cap k Q \subset \bar\Gamma_{\bar\omega}\cap k Q$. Let's show $\bar{\Gamma}_{\bar\omega} \cap kQ \subset \Gamma_{\bar\omega} \cap kQ$.
 Let $x\in \bar\Gamma_{\bar\omega}\cap k Q$. Since $\bar{\Gamma}_{\bar\omega} \cap kQ \subset \bG \cap kQ = \Gamma_\lambda \cap kQ \subset \Gamma_\lambda$, we get $x\in \Gamma_\lambda$.
	We know that $\bar\Gamma_{\bar\omega} = N\bG + \bar\Gamma^{\circ}$ (resp. $\bar\Gamma_{\bar\omega} = N\bar\Gamma_{\lambda,3} + \bar\Gamma^{\circ}$) if $3\nmid N'$ (resp. $3\mid N'$).
	Since $x\in \bar{\Gamma}_{\bar\omega}$, Corollary \ref{cor-key} implies that $x= N{\bf k}_1 +N ({\bf t}_1- {\bf t}_2)+ {\bf t}_3 - {\bf t}_4$, where ${\bf k}_1\in\Gamma^{*}$ and ${\bf t}_i\in\Gamma^{\circ}$ for $1\leq i\leq 4$. Note that ${\bf k}_1 +{\bf t}_1- {\bf t}_2 \in \bar\Gamma_{\lambda,3}$ if $3\mid N'$.
	Since $x\in \Gamma_\lambda$, 
 it follows that $x = {\bf k}_2 + {\bf t}_5$,  where ${\bf k}_2\in\Gamma^{*}$ and ${\bf t}_5\in\Gamma^{\circ}$. Then we have $$N{\bf k}_1 + N{\bf t}_1 + {\bf t}_3 = {\bf k}_2+ N{\bf t}_2+{\bf t}_4 + {\bf t}_5.$$
	The proof of Lemma \ref{lem-dd} (essentially, Lemma \ref{lem-alb}) shows that $N{\bf k}_1={\bf k}_2\in \Gamma^*$.
	Thus $x = N{\bf k}_1 + {\bf t}_5\in \Gamma_{\bar\omega}\cap k Q$ if $3\nmid N'$.
	
	Suppose that $3\mid N'$. To show that $x = N{\bf k}_1 + {\bf t}_5\in \Gamma_{\bar\omega}\cap k Q$, it suffices to show that ${\bf k}_1\in\Gamma_{\lambda,3}$. Since ${\bf k}_1 +{\bf t}_1- {\bf t}_2 \in \bar\Gamma_{\lambda,3}$, 
 it follows that ${\bf k}_1 +{\bf t}_1- {\bf t}_2= {\bf k}_3 - {\bf k}_4 \in \bar\Gamma_{\lambda,3} $. Here ${\bf k}_3,{\bf k}_4\in \Gamma_{\lambda,3}$. We have ${\bf k}_1 = {\bf k}_3 - {\bf k}_4- {\bf t}_1+ {\bf t}_2$. Since  ${\bf t}_1,{\bf t}_2\in \Gamma^{\circ}$, 
 we have
	${\bf t}_1,{\bf t}_2\in \Gamma_{\lambda,3}$. 
 Proposition \ref{lemgroup} implies that
	${\bf k}_1\in\Gamma_{\lambda,3}$.
\end{proof}

As mentioned, we will apply \cite[Lemma 2.3]{frohman2021dimension}, for which we need a filtration of $\Sfz$ by $\mathbb{N}$, compatible with the product.

Recall from \eqref{sum-k} that, for any ${\bf k}=(k_v)_{v\in V_\lambda}\in\mathbb Z^{V_\lambda}$, we define $\text{sum}({\bf k}) = \sum_{v\in V_\lambda} 
k_v$.
For any $k\in\mathbb Z$, define 
$F_k(\Sfz)$ to be the $\mathbb C$-vector subspace of $\Sfz$ spanned by 
$\{\alpha\in B_\fS\mid \text{sum}(\kappa(\al))\leq k\}$.
Then we have 
\begin{align}
\label{Sfz_filtration}
    \Sfz=\bigcup_{k\in\mathbb N} F_k(\Sfz).
\end{align}
Note that, for ${\bf k}, {\bf t}\in \mathbb Z^{V_\lambda}$, ${\bf k}\leq{\bf t}$ (here the order is the one defined in \eqref{eq-linear-order-V}) implies that 
$\text{sum}({\bf k})\leq \text{sum}({\bf t})$.  
Then Lemma \ref{lem-fil} implies that $F_k(\Sfz) F_t(\Sfz)\subset F_{k+t}(\Sfz)$. This is a sought-for degree filtration of $\Sfz$ by $\mathbb{N}$ compatible with the product, which is induced by the degree filtration of $\Sfz$ by $\mathbb{Z}^{V_\lambda}$, or in fact by $\mathbb{N}^{V_\lambda}$, studied in \S\ref{sub-trace}.

We now arrive at the fourth main theorem of the paper, on the computation of the rank of the skein algebra $\Sfz$ over the center, at a root of unity.
\begin{theorem}[Theorem \ref{the.intro.rank}]
\label{thm.rank}
	Suppose $\fS$ is a connected punctured surface with genus $g$ and 
 $n$ punctures ($
 n>0$). Then we have $\ran = \Kz,$ where $\Kz$ is defined in \eqref{eq-K}.
\end{theorem}
\begin{proof}
	When $\fS$ is a once or twice punctured sphere, it is trivial.
	
	Suppose $\fS$ is triangulable with a triangulation $\lambda$.
From Corollary \ref{cor-lower} we have $\ran  \geq |\bG/\bar{\Gamma}_{\bar\omega}|$, and from Lemma \ref{lem-equal} we have $|\bG/\bar{\Gamma}_{\bar\omega}|=K_{\bar\omega}$. So it suffices to show that $\ran \leq |\bG/\bar{\Gamma}_{\bar\omega}|$.

Recall the $\mathbb{N}$-filtration of $\Sfz$ in \eqref{Sfz_filtration}. Since $\ran <\infty$ (Proposition \ref{prop-center1}) and $\dim_\mathbb{C}(F_k(\Sfz))<\infty$, we can apply \cite[Lemma 2.3]{frohman2021dimension}, which says that there is a positive integer $u$ such that
$$
\ran \le \frac{ \dim_\mathbb{C} F_k(\Sfz) }{\dim_\mathbb{C} F_{k-u}(\Sfz) \cap \mathcal{Z}(\Sfz)}
$$
holds for all $k\ge u$.
Note that
\begin{align*}
		&\{\alpha\in B_{\fS}\mid \deg(\alpha)\in \Gamma_\lambda\cap k Q \}\text{ is a basis for $F_k(\Sfz)$},\\
		&\{\alpha\in B_{\mathcal Z,\bar\omega}\mid \deg(\alpha)\in \Gamma_{\bar\omega}\cap k Q \}\text{ is a basis for $F_k(\Sfz)\cap\Zz$},
	\end{align*}
where the latter assertion follows from Lemma \ref{lem-basis-center-filtration} below. 
	Then we have 
	\begin{align*}
		&\dim_{\mathbb C}(F_k(\Sfz)) = |\Gamma_\lambda\cap k Q| = |\bar\Gamma_\lambda\cap k Q|, \\
		&\dim_{\mathbb C}(F_k(\Sfz)\cap\Zz) = |\Gamma_{\bar\omega}\cap k Q| = |\bar\Gamma_{\bar\omega}\cap k Q|,
	\end{align*}
 where the last equality of each line follows from Lemma \ref{lem-Gammas_and_bar_Gammas}. 
 So we have
 $$
 \ran \le \frac{|\bG \cap kQ|}{|\bar{\Gamma}_{\bar\omega} \cap (k-u)Q|}
$$
for all $k\ge u$. Applying Lemmas \ref{lem-dim} and \ref{Lme} to $\Lambda = \bG$ and $\Gamma = \bar{\Gamma}_{\bar\omega}$, we get
\begin{align}
    \label{eq-smaller}
    \ran \le \lim_{k \to \infty} \frac{|\bG \cap kQ|}{|\bar{\Gamma}_{\bar\omega} \cap (k-u)Q|} = |\bG/\bar{\Gamma}_{\bar\omega}|,
\end{align}
as desired.
\end{proof}

\begin{lemma}\label{lem-basis-center-filtration}
For each positive integer $k$, 
    $$\{\alpha\in B_{\mathcal Z,\bar\omega}\mid \deg(\alpha)\in \Gamma_{\bar\omega}\cap k Q \}\text{ is a basis for $F_k(\Sfz)\cap\Zz$}.$$
\end{lemma}
\begin{proof}
From Lemma \ref{lem-cd}\ref{lem-cd-a}, it follows that elements in
$\{\alpha\in B_{\mathcal Z,\bar\omega}\mid \deg(\alpha)\in \Gamma_{\bar\omega}\cap k Q \}$ are linearly independent.
From Lemma \ref{lem-cd}\ref{lem-cd-a} and the definitions of $Q$ and $F_k(\Sfz)$, we have 
$$\{\alpha\in B_{\mathcal Z,\bar\omega}\mid \deg(\alpha) \in \Gamma_{\bar\omega}\cap k Q \}\text{$~\subset~ F_k(\Sfz)\cap\Zz$}.$$

It suffices to show that $F_k(\Sfz)\cap\Zz$ is linearly spanned by $\{\alpha\in B_{\mathcal Z,\bar\omega}\mid \deg(\alpha)\in \Gamma_{\bar\omega}\cap k Q \}$.
Let $0\neq x\in F_k(\Sfz)\cap\Zz$.
Since $x\in \Zz$,
from Lemma \ref{lem-cd}\ref{lem-cd-a}, we can
suppose $x= c_1x_1+\cdots+ c_m x_m$, where $0\neq c_i\in\mathbb C$ and $x_i\in B_{\mathcal Z,\bar\omega}$ for $1\leq i\leq m$.
Without 
loss of generality (and using Lemma \ref{lem-cd}\ref{lem-cd-b}), we can suppose that $\deg(x_1)>\cdots>\deg(x_m)$
(here the order is the one defined in \eqref{eq-linear-order-V}). Set ${\bf t}:=\deg(x_1)$. Then ${\bf t} = \deg(x) \in\Gamma_{\bar\omega}$ (Lemma \ref{lem-d1}\ref{lem-d1-a}). 
Thus
\begin{align*}
    x \in c_1x_1 + \mathcal D_{<{\bf t}}\subset
    c_1'\kappa^{-1}({\bf t}) + \mathcal D_{<{\bf t}},
\end{align*}
where $0\neq c_1'\in\mathbb C$.
Since $x\in F_k(\Sfz)$, we have $\text{sum}({\bf t})\leq k$.
This implies that
$$\text{sum}(\deg(x_m))\leq \cdots\leq \text{sum}(\deg(
x_1))=
\text{sum}({\bf t})\leq k.$$
Thus $x_i \in kQ$, so $x_i\in \{\alpha\in B_{\mathcal Z,\bar\omega}\mid \deg(\alpha)\in \Gamma_{\bar\omega}\cap k Q \}$ for $1\leq i\leq m.$
\end{proof}

This completes the proof of Theorem \ref{thm.rank}.

\begin{conjecture}
    Theorem \ref{thm.rank} holds when $n=0$, i.e. for closed surfaces $\fS$.
\end{conjecture}

\section{On skein modules for $3$-manifolds}\label{sec-3-manifolds}

The work of this section is motivated by \cite{frohman2023sliced,wang2023finiteness}.
Let 
$N'$ be a positive integer such that $3\nmid 
N'$ (so $\dd={\rm gcd}(N',3)=1$, and $N = N'$), and let $M$ be a compact 3-manifold.
When $R=\mathbb C$ and $\bar q$ is a root of unity of order $N=N'$, we will show that $\cS_{\bar q}(M)$ has a finitely generated $\cS_{\pm 1}(M)$-module structure (Propositions \ref{prop-module} and \ref{finite}). Here $\cS_{\pm 1}(M)$ is a commutative algebra (Lemma \ref{lem-com-algebra-sl3}). This module structure is induced by the Frobenius map $\cF: \cS_{\pm 1}(M)\rightarrow \cS_{\bar q}(M)$ between the skein modules obtained in Proposition \ref{Prop-Fro-M}.

As we shall see, any $\rho\in \mathfrak{X}_{{\rm SL}_3(\mathbb{C})}(M)$ (equation \eqref{eq-Sl3-character}), corresponds to an algebra homomorphism from $\cS_{\pm 1}(M)$ to $\mathbb C$ (Corollary \ref{cor-one}).
We define the character-reduced $\SL$-skein module associated to $\rho$ as 
$$\cS_{\bar q}(M)_\rho:=\cS_{\bar q}(M)\otimes_{\cS_{\pm 1}(M)}\mathbb C.$$
As shown in \cite{frohman2023sliced} for the ${\rm SL}_2$-version, $\cS_{\bar q}(M)_\rho$ plays a crucial role to study the representation theory of the skein algebra $\cS_{\bar q}(\partial M)$ (Proposition \ref{prop.skein_module_twisted_over_skein_algebra}, Conjecture \ref{conj-rep}).

\subsection{A commutative algebra structure on $\cS_{\bar\eta}(M)$ and the $\SL$-character variety}

We  
assume that 
$R=\mathbb C$, $\bar q=\bar \omega$, $\omega=\bar \omega^{3}$ with 
$\omega^{\frac{1}{3}} =\bar \omega$, and that 
$\bar\omega^2$ is a root of unity of order $N=N'$ such that $N$ is coprime with $3$.
Then the order of $\omega^2 = (\bar\omega^2)^3$ is also $N$.
 We set $\bar \eta = \bar\omega^{N^2}$ and $\eta = \bar \eta^{3}$ with $\eta^{\frac{1}{3}} = \bar\eta$. 
 Then $\bar\eta^{2} = \bar\omega^{2N^2} = 1$. Thus
 $\bar\eta = \pm 1$ and $\eta = \bar\eta^3=\bar\eta = \pm 1$.

For any  compact 3-manifold $M$,
Proposition \ref{Prop-Fro-M} implies there a $\mathbb C$-linear map
 $\cF: \cS_{\bar\eta}(M)\rightarrow \cS_{\bar\omega}(M)$, the Frobenius map for skein modules. 
 We will show that $\cF$ induces a finite module structure of $\cS_{\bar\omega}(M)$ over 
$\cS_{\bar\eta}(M)$.

For any topological space $T$ such that $\pi_1(T)$ is a finitely generated group,
we use $\text{Hom}(\pi_1(T),{\rm SL}_3(\mathbb C))$ to denote the set of group homomorphisms from $\pi_1(T)$ to ${\rm SL}_3(\mathbb C)$ (one might want to choose a basepoint of $T$; this choice will not matter later).
We define an equivalence relation $\simeq$ on  $\text{Hom}(\pi_1(T),{\rm SL}_3(\mathbb C))$.
 Let $\rho,\rho'\in \text{Hom}(\pi_1(T),{\rm SL}_3(\mathbb C))$. Define $\rho\simeq\rho'$ if and only if $
 {\rm tr}(\rho(x))=
 {\rm tr}(\rho'(x))$
 for all $x\in\pi_1(T)$.
Define
\begin{align}\label{eq-Sl3-character}
\mathfrak{X}_{{\rm SL}_3(\mathbb{C})}(T) = \text{Hom}(\pi_1(T),{\rm SL}_3(\mathbb C))/\simeq.
\end{align}
Then $\mathfrak{X}_{{\rm SL}_3(\mathbb{C})}(T)$ is an algebraic set over the complex field  \cite{sikora2001SLn}. 

Suppose that $M$ is a 3-manifold.
For any  framed knot $K$ in $M$, we define $\text{tr}_K\in \mathcal O(\mathfrak{X}_{{\rm SL}_3(\mathbb{C})}(M))$ as follows. Let $\rho\in \mathfrak{X}_{{\rm SL}_3(\mathbb{C})}(M)$ be a group homomorphism from $\pi_1(M)$ to ${\rm SL}_3(\mathbb C)$. Define
$$\text{tr}_K(\rho) = 
{\rm tr}(\rho(K)).
$$
Here $K$ in the right hand side is regarded as an element in $\pi_1(M)$ by 
forgetting the framing. 

\def\Se{\cS_{\bar\eta}(M)}
\def\Sz{\cS_{\bar\omega}(M)}

\begin{lemma}\label{lem-com-algebra-sl3}
	Let $M$ be a 3-manifold.
	There is a commutative algebra structure on $\Se$ defined as
 follows: for any two webs $l$ and $l'$, we first isotope $l$ and $l'$ such that $l\cap l'=\emptyset$, then define $l\cdot l' = l\cup l'\in \Se$.
\end{lemma}
\begin{proof}
	Relation \eqref{eq-cross} implies that the multiplication is independent of how we isotope $l$ and $l'$. Since the defining relations for ${\rm SL}_3$-skein modules are local relations,  
 this multiplication is  well-defined. 
 The commutativity follows easily.
\end{proof}

\begin{proposition}\label{chara}
There exists a unique surjective algebra homomorphism $$\Phi: \Se\rightarrow \mathcal O(\mathfrak{X}_{{\rm SL}_3(\mathbb{C})}(M))$$ such that $\Phi(K) = {\rm tr}_K$ for any framed knot $K$ in $M$. Furthermore, we have 
	$\kernel \Phi = \sqrt{0}_{\Se}$, where $\sqrt{0}_{\Se}$ is the ideal of $\Se$ consisting of nilpotent elements.
\end{proposition}
\begin{proof}
	We know that $\bar\eta = \pm 1$.
	
	Case 1:  When $\bar\eta = 1$, 
 this Proposition 
 follows from \cite[Corollary 20]{sikora2005skein}. 
	
	Case 2:
	When $\bar\eta = -1$, from \cite[Theorem 4.2]{Wan24} it follows that 
 the algebra $\cS_{-1}(M)$ is isomorphic to the algebra  $\cS_1(M)$. This isomorphism sends any framed knot in $M$ to itself. Thus 
 Case 2 follows from Case 1 together with 
 this isomorphism.
\end{proof}

\def\Homm{\text{Hom}_{\text{Alg}}(\Se,\mathbb C)}

We use $\text{Hom}_{\text{Alg}}(\Se,\mathbb C)$ to denote the set of all algebra homomorphisms from $\Se$ to $\mathbb C$. For element $\rho\in \mathfrak{X}_{{\rm SL}_3(\mathbb{C})}(M)$, we can regard it as an element in $\Homm$ such that $\rho(x) = \Phi(x)(\rho)$ for any
$x\in\Se$. 
Proposition \ref{chara} implies the following Corollary.

\begin{corollary}\label{cor-one}
	The above correspondence between $\mathfrak{X}_{{\rm SL}_3(\mathbb{C})}(M)$ and ${\rm Hom}_{\rm Alg}(\Se,\mathbb C)$ is a bijection.
\end{corollary}

\subsection{A module structure of $\Sz$ over $\Se$}

In this subsection, we will define a module structure of $\Sz$ over $\Se$ using the `transparency' property in the following Proposition. 

\begin{proposition}
\label{prop-tran1}
	Suppose that $l=\cup_{1\leq i\leq m}l_i\in\Se$ is a web in $M$ such that each $l_i$ is a framed knot, and that $T_1\in \Sz$ and $T_2\in\Sz$ are two isotopic webs in $M$ such that $l\cap T_1=l\cap T_2=\emptyset$. 
	Then $$\cF(l)\cup T_1=\cF(l)\cup T_2\in\Sz,$$
	 where $\cF(l)\cup T_j = (\cup_{1\leq i\leq m} l_i^{[P_{N,1}]})\cup T_j$ for $j=1,2$; the expression $\cup_{1\leq i\leq m} l_i^{[P_{N,1}]}$ is given by the threading operation in  \S\ref{subsec.threading}.
\end{proposition}
One could see \cite{HLW} for a proof of this proposition. We will provide a proof in \S\ref{subsec.3d-transparency}.

For a framed knot $a$ in a 3-manifold $M$, we will use $\begin{array}{c} 
\begingroup%
  \makeatletter%
  \providecommand\color[2][]{%
    \errmessage{(Inkscape) Color is used for the text in Inkscape, but the package 'color.sty' is not loaded}%
    \renewcommand\color[2][]{}%
  }%
  \providecommand\transparent[1]{%
    \errmessage{(Inkscape) Transparency is used (non-zero) for the text in Inkscape, but the package 'transparent.sty' is not loaded}%
    \renewcommand\transparent[1]{}%
  }%
  \providecommand\rotatebox[2]{#2}%
  \newcommand*\fsize{\dimexpr\f@size pt\relax}%
  \newcommand*\lineheight[1]{\fontsize{\fsize}{#1\fsize}\selectfont}%
  \ifx\svgwidth\undefined%
    \setlength{\unitlength}{59.52755906bp}%
    \ifx\svgscale\undefined%
      \relax%
    \else%
      \setlength{\unitlength}{\unitlength * \real{\svgscale}}%
    \fi%
  \else%
    \setlength{\unitlength}{\svgwidth}%
  \fi%
  \global\let\svgwidth\undefined%
  \global\let\svgscale\undefined%
  \makeatother%
  \begin{picture}(1,0.38095238)%
    \lineheight{1}%
    \setlength\tabcolsep{0pt}%
    \put(0,0){\includegraphics[width=\unitlength,page=1]{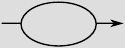}}%
    \put(0.21891212,0.12249803){\color[rgb]{0,0,0}\makebox(0,0)[lt]{\lineheight{1.25}\smash{\begin{tabular}[t]{l}$a^{[P_{N,1}]}$\end{tabular}}}}%
  \end{picture}%
\endgroup%
 \end{array}$ to represent $a^{[P_{N,1}]}\in\cS_{\bar\omega}(M)$.
Let $\beta$ be a web in $M$, we say that $\beta'$ is obtained from $\beta$ by adding a positive kink if one part
$
\raisebox{-.15in}{
	\begin{tikzpicture}
		\tikzset{->-/.style=
			{decoration={markings,mark=at position #1 with
					{\arrow{latex}}},postaction={decorate}}}
		\filldraw[draw=white,fill=gray!20] (-1,-0.5) rectangle (0.6, 0.5);
		\draw [line width =1pt,decoration={markings, mark=at position 0.5 with {\arrow{>}}},postaction={decorate}](-1,0)--(-0.25,0);
		\draw [color = black, line width =1pt](-0.25,0)--(0.6,0);
\end{tikzpicture}}$
 of $\beta$ is replaced by one
$\raisebox{-.15in}{
	\begin{tikzpicture}
		\tikzset{->-/.style=
			{decoration={markings,mark=at position #1 with
					{\arrow{latex}}},postaction={decorate}}}
		\filldraw[draw=white,fill=gray!20] (-1,-0.35) rectangle (0.6, 0.65);
		\draw [line width =1pt,decoration={markings, mark=at position 0.5 with {\arrow{>}}},postaction={decorate}](-1,0)--(-0.25,0);
		\draw [color = black, line width =1pt](0,0)--(0.6,0);
		\draw [color = black, line width =1pt] (0.166 ,0.08) arc (-37:270:0.2);
\end{tikzpicture}}
$.

\begin{proposition}
\label{prop-relation}
	Suppose that $M$ is a 3-manifold. Then the following hold:

 \begin{enumerate}[label={\rm (\arabic*)}]
	\item\label{prop-relation-1} Let $a,b$ be two framed knots in $M$ (we allow $a=b$). We have the following relation 
	$$\begin{array}{c}
\begingroup%
  \makeatletter%
  \providecommand\color[2][]{%
    \errmessage{(Inkscape) Color is used for the text in Inkscape, but the package 'color.sty' is not loaded}%
    \renewcommand\color[2][]{}%
  }%
  \providecommand\transparent[1]{%
    \errmessage{(Inkscape) Transparency is used (non-zero) for the text in Inkscape, but the package 'transparent.sty' is not loaded}%
    \renewcommand\transparent[1]{}%
  }%
  \providecommand\rotatebox[2]{#2}%
  \newcommand*\fsize{\dimexpr\f@size pt\relax}%
  \newcommand*\lineheight[1]{\fontsize{\fsize}{#1\fsize}\selectfont}%
  \ifx\svgwidth\undefined%
    \setlength{\unitlength}{85.03937008bp}%
    \ifx\svgscale\undefined%
      \relax%
    \else%
      \setlength{\unitlength}{\unitlength * \real{\svgscale}}%
    \fi%
  \else%
    \setlength{\unitlength}{\svgwidth}%
  \fi%
  \global\let\svgwidth\undefined%
  \global\let\svgscale\undefined%
  \makeatother%
  \begin{picture}(1,0.56666667)%
    \lineheight{1}%
    \setlength\tabcolsep{0pt}%
    \put(0,0){\includegraphics[width=\unitlength,page=1]{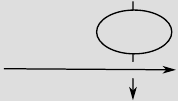}}%
    \put(0.57560943,0.34523383){\color[rgb]{0,0,0}\makebox(0,0)[lt]{\lineheight{1.25}\smash{\begin{tabular}[t]{l}$b^{[P_{N,1}]}$\end{tabular}}}}%
    \put(0,0){\includegraphics[width=\unitlength,page=2]{threading_transparency1.pdf}}%
    \put(0.16676497,0.12834866){\color[rgb]{0,0,0}\makebox(0,0)[lt]{\lineheight{1.25}\smash{\begin{tabular}[t]{l}$a^{[P_{N,1}]}$\end{tabular}}}}%
  \end{picture}%
\endgroup%
\end{array} =  
	\begin{array}{c} 
\begingroup%
  \makeatletter%
  \providecommand\color[2][]{%
    \errmessage{(Inkscape) Color is used for the text in Inkscape, but the package 'color.sty' is not loaded}%
    \renewcommand\color[2][]{}%
  }%
  \providecommand\transparent[1]{%
    \errmessage{(Inkscape) Transparency is used (non-zero) for the text in Inkscape, but the package 'transparent.sty' is not loaded}%
    \renewcommand\transparent[1]{}%
  }%
  \providecommand\rotatebox[2]{#2}%
  \newcommand*\fsize{\dimexpr\f@size pt\relax}%
  \newcommand*\lineheight[1]{\fontsize{\fsize}{#1\fsize}\selectfont}%
  \ifx\svgwidth\undefined%
    \setlength{\unitlength}{85.03937008bp}%
    \ifx\svgscale\undefined%
      \relax%
    \else%
      \setlength{\unitlength}{\unitlength * \real{\svgscale}}%
    \fi%
  \else%
    \setlength{\unitlength}{\svgwidth}%
  \fi%
  \global\let\svgwidth\undefined%
  \global\let\svgscale\undefined%
  \makeatother%
  \begin{picture}(1,0.56666667)%
    \lineheight{1}%
    \setlength\tabcolsep{0pt}%
    \put(0,0){\includegraphics[width=\unitlength,page=1]{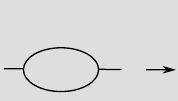}}%
    \put(0.16676497,0.12834866){\color[rgb]{0,0,0}\makebox(0,0)[lt]{\lineheight{1.25}\smash{\begin{tabular}[t]{l}$a^{[P_{N,1}]}$\end{tabular}}}}%
    \put(0,0){\includegraphics[width=\unitlength,page=2]{threading_transparency2.pdf}}%
    \put(0.57560943,0.34523383){\color[rgb]{0,0,0}\makebox(0,0)[lt]{\lineheight{1.25}\smash{\begin{tabular}[t]{l}$b^{[P_{N,1}]}$\end{tabular}}}}%
  \end{picture}%
\endgroup%
 \end{array}\in \cS_{\bar\omega}(M),
	$$

	\item\label{prop-relation-2} Suppose that $\beta$ is a web in $M$ and that a web $\beta'$ is obtained from $\beta$ by adding a positive kink. Then 
	$$(\beta')^{[ P_{N,1}]} = \beta^{[ P_{N,1}]} \in \cS_{\bar\omega}(M).$$
	
	\item\label{prop-relation-3} We have 
	$$
	\left(\raisebox{-.20in}{
		\begin{tikzpicture}
			\tikzset{->-/.style=
				{decoration={markings,mark=at position #1 with
						{\arrow{latex}}},postaction={decorate}}}
			\filldraw[draw=white,fill=gray!20] (0,0) rectangle (1,1);
			\draw [line width =1pt,decoration={markings, mark=at position 0.5 with {\arrow{>}}},postaction={decorate}](0.45,0.8)--(0.55,0.8);
			\draw[line width =1pt] (0.5 ,0.5) circle (0.3);
	\end{tikzpicture}}\right)^{[P_{N,1}]}
	= 3 \ 
	\raisebox{-.15in}{
		\begin{tikzpicture}
			\tikzset{->-/.style=
				{decoration={markings,mark=at position #1 with
						{\arrow{latex}}},postaction={decorate}}}
			\filldraw[draw=white,fill=gray!20] (0,0) rectangle (1,1);
	\end{tikzpicture}}\in \cS_{\bar\omega}(M).
	$$
 \end{enumerate}
\end{proposition}
\begin{proof}
    All statements hold in $\mathscr{S}_{\bar{\eta}}(\fS)$ if we delete all $[P_{N,1}]$; for example, replace $a^{[P_{N,1}]}$ and $b^{[P_{N,1}]}$ by $a$ and $b$. Applying the Frobenius map $\mathcal{F} : \mathscr{S}_{\bar{\eta}}(\fS) \to \mathscr{S}_{\bar{\omega}}(\fS)$ immediately yields the desired result, by Theorem \ref{Fro-surface}.
\end{proof}

\def\si{\textbf{Sink}}
\def\so{\textbf{Source}}

The following 
straightforward Lemma 
holds for general ground ring $R$ and $\hat q$. 
\begin{lemma}\label{lem-fun}
	Suppose that $M,M'$ are two  3-manifolds, and $L:M'\rightarrow M$ is an orientation-preserving embedding. Let $\beta$ be a web in $M$ such that $\beta\cap L(M')=\emptyset$. Then there exists an $R$-linear map
	$$L_{M',\beta}:\cS_{\bar q}(M')\rightarrow \cS_{\bar q}(M)$$
	such that $L_{M',\beta}(\alpha) = L(\alpha)\cup\beta$ for any web $\beta$ in $M'$. 
\end{lemma}

We arrive at the promised module structure of $\Sz$ over $\Se$.

\begin{proposition}\label{prop-module}
	For any  3-manifold $M$, there is a unique module structure of $\Sz$  over $\Se$ such that, for any web $l\in\Se$ consisting of knots and any web $T\in \Sz$ with $l\cap T=\emptyset$, we have 
 $$
 l\cdot T = \cF(l)\cup T,
 $$
 where $\mathcal{F}: \Se \to \Sz$ is the Frobenius map for skein modules established in Proposition \ref{Prop-Fro-M}.
\end{proposition}
\begin{proof}
	Suppose that $\alpha$ is a web in $M$. We use $\si(\alpha)$ (resp. $\so(\alpha)$) to denote the set of all 3-valent sinks (resp. sources) of $\alpha$. Let $f$ be a bijection from $\si(\alpha)$ to $\so(\alpha)$. For each $t\in\si(\alpha)$, let $p_t:[0,1]\rightarrow M$ be an embedding such that $p_t(0) = t,p_t(1)=f(t)$, and $p_t\cap\alpha = \{t,f(t)\}$. Let $U$ be a closed regular neighborhood of $\alpha\cup(\cup_{t\in\si(\alpha)}p_t)$ such that it 
 deformation retracts to $\alpha\cup(\cup_{t\in\si(\alpha)}p_t)$.

	We will define a $\mathbb C$-linear map
	$$\varphi_{\alpha,f,p,U}:\Sz\rightarrow\Sz.$$
	We use $L$ to denote the embedding from $U$ to $M$. For any web $\beta\in\Sz$, we isotope $\beta$ such that $\beta\cap U=\emptyset$. Then define $$
 \varphi_{\alpha,f,p,U}(\beta) = L_{U,\beta}(\cF(\alpha)).
 $$ Here 
	$L_{U,\beta}:\cS_{\bar\omega}(U)\rightarrow\Sz$ is the map in Lemma \ref{lem-fun},   $\alpha\in \cS_{\bar\eta}(U)$, and $\cF:\cS_{\bar\eta}(U)\rightarrow\cS_{\bar\omega}(U)$.
	Using relation \eqref{wzh.four}, we can suppose $\alpha = \sum_{i}c_il_i\in\Se$, where each $l_i$ is a framed knot contained in $U$ and $c_i\in\mathbb C$. 
	Then $$\varphi_{\alpha,f,p,U}(\beta) =\sum_{i} c_i L_{U,\beta}(l_i^{[P_{N,1}]}) = \sum_{i} c_i (l_i^{[P_{N,1}]}\cup\beta).$$
	Proposition \ref{prop-tran1} implies that $\varphi_{\alpha,f,p,U}(\beta)$ is independent of how we isotope $\beta$. The same reason shows that $\varphi_{\alpha,f,p,U}$ is well-defined on the isotopy classes of webs $\alpha$ in $M$. Since relations \eqref{w.cross}-\eqref{wzh.four} are local relations, 
 it follows that $\varphi_{\alpha,f,p,U}$ is a well-defined $\mathbb C$-linear map.

	We will show that $\varphi_{\alpha,f,p,U}$ is independent of $f,p,U$. Suppose that we have $f',p'$, and $U'$.
	We use $L'$ to denote the embedding from $U'$ to $M$, and
	 use $J$ (resp. $J'$) to denote the embedding from $U$ (resp $U'$) to $U\cup U'$. We use $H$ to denote the embedding from $U\cup U'$ to $M$. 
  Note that $H\circ J= L$, $H\circ J' = L'$, and
	$ J_{*}(\cF(\alpha)) = J'_{*}(\cF(\alpha))\in\cS_{\bar\omega}(U\cup U')$.
	Here $\cF(\alpha)\in\zS(U)$ for $ J_{*}(\cF(\alpha))$, and $\cF(\alpha)\in\zS(U')$ for $ J'_{*}(\cF(\alpha))$.
	For any web $\beta\in M$, we isotope $\beta$ such that $\beta\cap(U\cup U')=\emptyset$ (this is doable since both $U$ and $U'$ are closed regular  neighborhoods of graphs). 
	Then we have 
	\begin{equation*}
		\begin{array}{rl}
		    \varphi_{\alpha,f,p,U}(\beta) = L_{U,\beta}(\cF(\alpha)) & = H_{U\cup U',\beta}(J_*(\cF(\alpha))) \\ &= 
		    H_{U\cup U',\beta}(J'_*(\cF(\alpha))) =  L_{U',\beta}(\cF(\alpha))  =
		    \varphi_{\alpha,f',p',U'}(\beta).
		\end{array}
	\end{equation*}
	Since we showed that $\varphi_{\alpha,f,p,U}$ is independent of $f,p,U$,  
 we will use $\varphi_{\alpha}$ to denote it. 
	
	\def\End{\text{End}_{\mathbb C}(\Sz)}
	\def\tF{\widetilde{\cF}}
	
	We use $\text{End}_{\mathbb C}(\Sz)$ to denote the set of all $\mathbb C$-linear maps from $\Sz$ to $\Sz$. Then $\text{End}_{\mathbb C}(\Sz)$ is a $\mathbb C$-algebra with the product given by the 
 composition of maps. 
	Define $$\widetilde{\cF}: \Se\rightarrow \End$$ such that $\tF(\alpha) = \varphi_{\alpha}$ for any web $\alpha$ in $M$.  We will show that $\tF$ is a well-defined algebra homomorphism.
	
	Supppose that $\alpha$ and $\alpha'$ are two isotopic webs in $M$.
	If $\alpha$ and $\alpha'$ contain no sinks or sources, then 
	Proposition \ref{prop-tran1} and Proposition \ref{prop-relation}\ref{prop-relation-1}  
 imply that $\tF(\alpha) = \tF(\alpha').$
	Assume that $\alpha$ and $\alpha'$ contain sinks and sources.
	 As above we choose $f,p,U$ for $\alpha$. Since $\alpha'$ is isotopic to $\alpha$, then there exist $f',p',U'$ for $\alpha'$ such that $\alpha'\cup(\cup_{t\in\si(\alpha')} p_t')$ is obtained from $\alpha\cup(\cup_{t\in\si(\alpha)} p_t)$ by a sequence of isotopies and crossing changes.  For each $t\in\si(\alpha)$, we use path $p_t$ to drag $t$ close enough to $f(t)$, then use the relation \eqref{wzh.four} to kill $t$ and $p(t)$. Then $\alpha = \sum_{i}c_il_i\in\Se$, where each $l_i$ is a framed knot in $U\subset M$ and $c_i\in\mathbb C$. We do the same thing to $\alpha'$, then 
	  $\alpha' = \sum_{i}c_i'l_i'\in\Se$, where each $l_i'$ is a framed knot in $U'\subset M$ and $c_i'\in\mathbb C$. 
	  Since $\alpha'\cup(\cup_{t\in\si(\alpha')} p_t')$ is obtained from $\alpha\cup(\cup_{t\in\si(\alpha)} p_t)$ by a sequence of isotopies and crossing changes, then we can suppose that $c_i'=c_i$ and that $l_i'$ is obtained from $l_i$ by a sequence of isotopies and crossing changes for each $i$. 
	  For any web $\beta\in\Sz$, we isotope $\beta$ such that $\beta\cap(U\cup U')=\emptyset$. Then
	  we have 
	  $$\varphi_{\alpha}(\beta) = \sum_i c_i l_i^{[P_{N,1}]}\cup\beta\text{ and }\varphi_{\alpha'}(\beta) = \sum_i c_i (l_i')^{[P_{N,1}]}\cup\beta.$$
	  	Proposition \ref{prop-tran1} and Proposition \ref{prop-relation}\ref{prop-relation-1} 
    imply that 
	  	$l_i^{[P_{N,1}]}\cup\beta = (l_i')^{[P_{N,1}]}\cup\beta\in\Sz$ for each $i$. Then we have $\varphi_{\alpha}(\beta)=\varphi_{\alpha'}(\beta)$. Thus $\tF(\alpha) = \tF(\alpha')$. This shows that $\tF$ is well defined on the set of isotopy classes of webs in $M$.
	  	
	  	The construction of $\tF$ shows that it preserves the relation \eqref{wzh.four}. Then
   the relations in Proposition \ref{prop-relation} imply that $\tF$ preserves the relations \eqref{w.cross}-\eqref{w.unknot}. 
	  
	  The algebra homomorphism $\widetilde{\cF}: \Se\rightarrow \End$ gives the module structure as desired. 
\end{proof}

\begin{lemma}\label{lem-functor}
	Suppose that $L:M'\rightarrow M$ is an orientation-preserving embedding between 3-manifolds. 
	\begin{enumerate}[label={\rm (\alph*)}]\itemsep0,3em
	\item\label{lem-functor-a} The following diagram commutes:
		\begin{equation}\label{eq-com}
		\begin{tikzcd}
			\cS_{\bar\eta}(M')  \arrow[r, "\cF"]
			\arrow[d, "L_{*,\bar\eta}"]  
			&  \cS_{\bar\omega}(M') \arrow[d, "L_{*,\bar\omega}"] \\
			\cS_{\bar\eta}(M)\arrow[r, "\cF"] 
			& \cS_{\bar\omega}(M), \\
		\end{tikzcd}
	\end{equation}
 where the vertical maps $L_{*,\bar\eta}$ and $L_{*,\bar\omega}$ are natural maps induced by $L$.
	
	\item For any $l\in\cS_{\bar\eta}(M')$ and $T\in\cS_{\bar\omega}(M')$, we have 
	$L_{*,\bar\omega}(l\cdot T) =L_{*,\bar\eta}(l)\cdot L_{*,\bar\omega}(T).$
	
	\item\label{lem-functor-c} If $L_{*,\bar\omega}:\cS_{\bar\omega}(M')\rightarrow \Sz$ is surjective and $\cS_{\bar\omega}(M')$ is finitely generated over $\cS_{\bar\eta}(M')$, then $\cS_{\bar \omega}(M)$ is finitely generated over $\Se$. 
 \end{enumerate}
\end{lemma}
\begin{proof}
	(a) Suppose that $l\in\cS_{\bar\eta}(M')$ is a web. We can suppose that $l$ consists of framed knots because of relation \eqref{wzh.four}. Assume that $l=\cup_{1\leq i\leq m} l_i$, where each $l_i$ is a framed knot in 
 $M'$. 
	Then we have 
	\begin{align*} L_{*,\bar\omega}(\cF(l))  = L_{*,\bar\omega}(\cup_{1\leq i\leq m} l_i^{[P_{N,1}]})
	 & = \cup_{1\leq i\leq m} L_{*,\bar\omega}(l_i)^{[P_{N,1}]} \\
  & =\cF(\cup_{1\le i \le m} L_{*,\bar\eta}(l_i)) = \cF(L_{*,\bar\eta}(l)).
  \end{align*}
	 
	 (b) We can suppose that both $l$ and $T$ are webs in $M$ and that $l$ consists of framed knots. Assume that $l = \cup_{1\leq i\leq m}l_i$, where each $l_i$ is a framed knot. Then we have 
	 $$
	 	L_{*,\bar\omega}(l\cdot T) = L_{*,\bar\omega}((\cup_{1\leq i\leq m} l_i^{[P_{N,1}]})\cup T) = (\cup_{1\leq i\leq m} L_{*,\bar\omega}(l_i)^{[P_{N,1}]})\cup L(T) = L_{*,\bar\eta}(l)\cdot L_{*,\bar\omega}(T).
	 $$
	 
	 (c) Suppose that $\cS_{\bar\omega}(M')$ is generated by $x_1,\cdots,x_m$ over $\cS_{\bar\eta}(M')$.  For any $y\in\Sz$, there exists $x\in\cS_{\bar\omega}(M')$ such that $L_{*,\bar\omega}(x) = y$. 
	 Then there exist $u_1,\cdots,u_m\in\cS_{\bar\eta}(M')$ such that
	 $x= u_1\cdot x_1+\cdots+ u_m\cdot x_m$. Thus we have 
	 $$y = L_{*,\bar\omega}(u_1\cdot x_1+\cdots+ u_m\cdot x_m) = 
	 L_{*,\bar\eta}(u_1)\cdot L_{*,\bar\omega}( x_1)+\cdots+ L_{*,\bar\eta}(u_m)\cdot L_{*,\bar\omega}( x_m).$$
  Therefore, $\cS_{\bar \omega}(M)$ is generated by $L_{*,\bar\omega}( x_1),\cdots,L_{*,\bar\omega}( x_m)$ over $\Se$.
\end{proof}

\begin{proposition}\label{finite}
	For any compact 3-manifold $M$,  
 $\Sz$ is a finitely generated module over
	$\Se$.  
\end{proposition}
\begin{proof}
	Since $M$ is a compact 3-manifold,  
 $M$ is obtained from a handlebody $H_g$ by gluing 2-handles (maybe we need to fill in some sphere boundary components with  3-dimensional balls); see \S\ref{sec-Fro}. Lemma \ref{lem-handle} implies that $L_{*,\bar\omega}:\cS_{\bar\omega}(H_g)\rightarrow \Sz$ is surjective. 
	Note that $H_g$ is isomorphic to the thickening of a surface.  
 From Corollary \ref{cor-Fro}
 it follows that $\cS_{\bar\omega}(H_g)$ is finitely generated over $\cS_{\bar\eta}(H_g)$. Then Lemma \ref{lem-functor}\ref{lem-functor-c} 
 completes the proof.
\end{proof}

\begin{remark}
	Proposition \ref{finite} for stated ${\rm SL}_2$-skein modules is proved in \cite{wang2023finiteness}. 
\end{remark}

\subsection{Character-reduced $\SL$-skein modules}

For any $\rho\in \mathfrak{X}_{{\rm SL}_3(\mathbb{C})}(M)$, we can regard it as an algebra homomorphism from $\Se$ to $\mathbb C$ as in Corollary \ref{cor-one}. Then $\rho$ induces a $\Se$-module structure on $\mathbb C$ such that $x\cdot k = \rho(x)k$ for 
$x\in\Se$ and $k\in\mathbb C$. Then define the following $\mathbb{C}$-vector space: 
\begin{align}
    \label{character-reduced_space}
    \Sz_{\rho} = \Sz\otimes_{\Se}\mathbb C.
\end{align}

Proposition \ref{finite} implies the following. 
\begin{corollary}\label{cor-finite1}
	For any compact 3-manifold $M$ and any $\rho\in \mathfrak{X}_{{\rm SL}_3(\mathbb{C})}(M)$, the vector space $\Sz_{\rho}$ is of finite dimension  over $\mathbb C$. 
\end{corollary}

\begin{conjecture}\label{conj}
    Suppose  $M$ is a closed 3-manifold. If $\rho: \pi_1(M)\rightarrow \SL(\mathbb C)$ is irreducible, then $\Sz_{\rho} = \mathbb C.$ (See Remark \ref{rem-finnal}.) 
\end{conjecture}

Recall that,
for any  surface $\fS$ without boundary, there is an algebra homomorphism $\cF:\cS_{\bar\eta}(\fS)\rightarrow \cS_{
\bar\omega}(\fS)$, the Frobenius map for skein algebras (Theorem \ref{Fro-surface}).
From Proposition 
\ref{prop-center} 
it follows that $\im \cF\subset \mathcal Z(\cS_{\bar\eta}(\fS))$ (note that we are assuming $3\nmid N'$).

We have a natural identification
$\mathfrak{X}_{{\rm SL}_3(\mathbb{C})}(\fS\times[-1,1])= \mathfrak{X}_{{\rm SL}_3(\mathbb{C})}(\fS)$ because $\pi_1(\fS\times [-1,1])$ is naturally isomorphic to
$\pi_1(\fS)$.

\begin{definition}\label{Def-cl-sh}
    Suppose that $\fS$ is a closed surface, and that $(V, \mu)$ is a finite dimensional representation of the skein algebra $\cS_{\bar\omega}(\fS)$. Assume that
$\mu(\cF(x)) = r_x \, {\rm Id}_V$ holds for 
each $x\in\cS_{\bar\eta}(\fS)$, where $r_x\in\mathbb C$ (note that every finite dimensional irreducible representation $\mu$ of $\cS_{\bar\eta}(\fS)$ satisfies this assumption). Then the map $r_{\mu} \colon \cS_{\bar\eta}(\fS)\rightarrow\mathbb C$ given by $x\mapsto r_x$ is an algebra homomorphism. Corollary \ref{cor-one} implies that $r_\mu$ corresponds to a point $\rho_\mu\in \mathfrak{X}_{{\rm SL}_3(\mathbb{C})}(\fS)$.
We call $\rho_\mu$ the {\bf classical shadow} of the representation $(V,\mu)$. 
\end{definition}

The classical shadow in Definition \ref{Def-cl-sh} is equivalent to the one defined in \eqref{eq-X}  if $\im \cF = \mathcal Z(\cS_{
\bar\omega}(\fS))$. 
It is the $\SL$ version for the classical shadow defined by
Bonahon and Wong for ${\rm SL}_2$ \cite{bonahon2016representations}.

\begin{remark}
    We believe that $\im \cF = \mathcal Z(\cS_{
    \bar\omega}(\fS))$ when $\fS$ is a closed surface. Note that $\mathcal Z(\cS_{
    \bar\omega}(\fS))$ is 
    strictly bigger than $\im \cF$ when each component of $\fS$ contains at least one puncture (Theorem \ref{thm-center} and Lemma \ref{lem-d1}\ref{lem-d1-b}). This is why we require $\fS$ to be a closed surface when we define the 
    ``classical shadow" in Definition \ref{Def-cl-sh}.
\end{remark}

For any compact 3-manifold $M$ with $\partial M\neq\emptyset$,  there is a natural left action of the skein algebra $\cS_{\bar\omega}(\partial M)$ on the $\mathbb C$-vector space $\Sz$, given as follows. There is an embedding  $L\colon \partial M\times [-1,1]\rightarrow M$ such that $L(\partial M\times\{1\}) = \partial M$.
For any web $\alpha\in \cS_{\bar\omega}(\partial M)$ and any web $\beta\in \Sz$, define 
\begin{align}
\label{boundary_action}
\alpha\cdot\beta = L(\alpha)\cup\beta\in \Sz. 
\end{align}
Here we isotope $\beta$ such that $(\im L)\cap\beta=\emptyset$.

Proposition \ref{prop-tran1} implies the following.
\begin{corollary}\label{cor-action} 
     Let $M$ 
     be a compact 3-manifold with $\partial M\neq\emptyset$. Suppose that $\alpha$ is a web in $\cS_{\bar\omega}(\partial M)$, $\beta$ is a web in $\Se$, and $\gamma$ is a web in $\Sz$.
    Then we have $$\alpha\cdot(\beta\cdot\gamma) = \beta\cdot(\al\cdot\gamma)\in \Sz.$$
\end{corollary}

The embedding $L\colon \partial M\times [-1,1]\rightarrow M$ induces a group homomorphism $\bar L:\pi_1(\partial M)=\pi_1(\partial M\times [-1,1])\rightarrow \pi_1(M).$
Thus we can see that each $\rho\in \mathfrak{X}_{{\rm SL}_3(\mathbb{C})}(M)$ induces a point in $\rho_\partial\in \mathfrak{X}_{{\rm SL}_3(\mathbb{C})}(\partial M)$; here $\rho_\partial$
is the composition 
$$
\rho_\partial : \pi_1(\partial M)\xrightarrow{\bar L}\pi_1(M)\xrightarrow{\rho} \SL(\mathbb C).
$$

\begin{proposition}
\label{prop.skein_module_twisted_over_skein_algebra}
    Suppose that $M$ is a compact 3-manifold with $\partial M\neq\emptyset$. For any $\rho\in \mathfrak{X}_{{\rm SL}_3(\mathbb{C})}(M)$, the vector space $\Sz_\rho$ defined in \eqref{character-reduced_space}
    has the structure of a finite dimensional representation of the skein algebra $\cS_{\bar\omega}(\partial M)$ whose classical shadow 
    coincides with $\rho_\partial$ in the sense of Definition \ref{Def-cl-sh}.
\end{proposition}
\begin{proof}
    Corollary \ref{cor-action} implies that the $\cS_{\bar\omega}(\partial M)$-module structure on $\Sz$ induces a $\cS_{\bar\omega}(\partial M)$-module structure on $\Sz_\rho$. Corollary \ref{cor-finite1} implies that
    $\Sz_\rho$ is a finite dimensional representation of $\cS_{\bar\omega}(\partial M)$.

    Note that $L\colon \partial M\times [-1,1]\rightarrow M$ naturally induces two $\mathbb C$-linear maps
$$L_{*,\bar\eta}\colon \cS_{\bar\eta} (\partial M)\rightarrow 
\cS_{\bar\eta} ( M)\text{ ~and~ }L_{*,\bar\omega}\colon \cS_{\bar\omega} (\partial M)\rightarrow 
\cS_{\bar\omega} ( M),$$
which appeared in 
the commutative diagram in \eqref{eq-com} of Lemma \ref{lem-functor}\ref{lem-functor-a}.

    For a web $x\in \cS_{\bar\eta}(\partial M)$ consisting of a single framed knot 
    and for a web $y\in \Sz$, we
    have 
    \begin{align*}
        \cF(x)\cdot( y\otimes_{\Se} 1) &= (L_{*,\bar\omega}(\cF(x))\cup y)\otimes_{\Se} 1 \quad (\because \mbox{ \eqref{boundary_action}}) \\
        & =  (\cF(L_{*,\bar\eta}(x))\cup y)\otimes_{\Se} 1 \quad (\because\mbox{Lemma \ref{lem-functor}\ref{lem-functor-a})} \\
        &= ((L_{*,\bar\eta}(x)) \cdot y) \otimes_{\Se} 1 \quad (\because\mbox{Proposition \ref{prop-module}}) \\
        &={\rm tr}(\rho(\bar{L}(x))) (y \otimes_{\Se} 1) \quad (\because\mbox{\eqref{character-reduced_space}, Corollary \ref{cor-one}}) \\
        &={\rm tr}(\rho_\partial(x)) (y \otimes_{\Se} 1) \\
        &=\rho_\partial(x) (y\otimes_{\Se} 1). \quad (\because\mbox{ Corollary \ref{cor-one}
        }) 
    \end{align*}
    So the action of $\mathcal{F}(x)$ on $\Sz_\rho$ equals $\rho_\partial(x)\cdot {\rm Id}$, for each single framed knot $x \in \cS_{\bar\eta}(\partial M)$. Since $\mathcal{F}$ is an algebra homomorphism and since $\cS_{\bar\eta}(\partial M)$ is generated by framed knots (by \eqref{wzh.four} and \eqref{eq-cross}), we are done.

\end{proof}

\begin{conjecture}\label{conj-rep}
    Suppose that $H_g$ is the genus $g$ handlebody, and $\rho\in \mathfrak{X}_{{\rm SL}_3(\mathbb{C})}(H_g)$ is represented by an irreducible representation. Then $\cS_{\bar\omega}(H_g)_\rho$ is an irreducible representation of $\cS_{\bar\omega}(\partial H_g)$ with dimension $N^{8g-8}$ whose classical shadow is $\rho_\partial$.
\end{conjecture}

\begin{remark}\label{rem-finnal}
	Conjecture \ref{conj} for  ${\rm SL}_2$-skein modules is proved in \cite{frohman2023sliced}. 
	The stated ${\rm SL}_2$-skein module version is proved in \cite{wang2023representation}.
 Conjecture \ref{conj-rep} for ${\rm SL}_2$-skein theory is proved in \cite{frohman2023sliced}.

\end{remark}

\section{The proofs for Theorem \ref{thm-com-trace} and Proposition \ref{prop-center}}
\label{sec.independent_proofs}

Theorem \ref{thm-com-trace} and Proposition \ref{prop-center}, for which we cited an upcoming joint paper of the authors with Thang L\^e \cite{HLW}, are used crucially in the present paper. In this section we provide proofs for these, independently on \cite{HLW}, to make the paper more self-contained. With the tools developed here, we also provide a proof of Proposition \ref{prop-tran1} independently on \cite{HLW}.

In the present paper, so far we formulated the statements and arguments using only ${\rm SL}_3$-skein algebras (and modules), based on knots and 3-valent graphs with framing, and their skein relations \eqref{w.cross}-\eqref{wzh.four}, which are local relations happening over the interior of the surface $\fS$. We recall from the literature on ${\rm SL}_3$-skein algebras that proofs of many important properties in fact make use of the `stated' ${\rm SL}_3$-skein algebra of a pb surface $\fS$ with boundary (\cite{higgins2020triangular,HW,kim2011sl3,le2021stated,le2023quantum,wang2023stated}), as hinted in Remark \ref{rem-sl3}; in this section we will uses these stated ${\rm SL}_3$-skein algebras. 

We cut the pb surface $\fS$ along an ideal arc running between punctures to get another pb surface. This cutting induces a `splitting' map, which is a 
homomorphism between the stated ${\rm SL}_3$-skein algebras of these two pb surfaces \cite{higgins2020triangular,le2021stated}, and we recall from literature that the Frobenius maps are compatible with the splitting maps \cite{higgins2020triangular,wang2023stated}, and that the quantum trace maps are compatible with the splitting maps \cite{le2023quantum}. These compatibilities with the cutting process often boils the problem down to checking statements for small surfaces like a triangle $\mathbb{P}_3$.

\subsection{Stated $\SL$-skein algebra} 
We begin by recalling the notion of stated ${\rm SL}_3$-tangles, which are basic objects used in the definition of stated ${\rm SL}_3$-skein algebras.

Let $\fS$ be a pb surface (Definition \ref{def:surface}). For an element $(x,t)\in\fS\times[-1,1]$, we refer $t$ as its height. For two elements $(x_1,t_1),(x_2,t_2)\in\fS\times[-1,1]$, we say $(x_1,t_1)$ is higher than $(x_2,t_2)$ if $t_1>t_2$. For  subsets $A \subset \fS \times [-1,1]$ and $B\subset \fS$, we say that $A$ (or an element of $A$) lies {\bf over} $B$ if $A \subset B \times [-1,1]$.

\begin{definition}\label{def.SL3-tangle}
    Let $\fS$ be a pb surface. 
    An {\bf $\SL$-tangle} $\alpha$ in $\fS\times [-1,1]$ is a disjoint union of oriented closed paths and a directed finite graph properly embedded in $\fS\times [-1,1]$, such that the valence of the vertices of the graph may be 1 or 3, where 1-valent vertices must lie over $\partial \fS$ and 3-valent vertices over the interior of $\fS$, and the conditions (W2)--(W5) of Def.\ref{def-web} should hold. 
    Define $\partial \alpha=\alpha\cap(\partial\fS\times[-1,1])$, whose elements are called {\bf endpoints} of $\alpha$.
    We require the following conditions for 
    $\partial \alpha$:
    \begin{enumerate}[label={\rm (\arabic*)}]
        \item 
        for each connected component $b$ of $\partial \fS$, the endpoints of $\alpha$ lying over $b$, if there are any, have mutually distinct heights;
        
        \item 
        the framing of $\alpha$ at each endpoint of $\alpha$ is parallel to the $[-1,1]$ factor and points toward the positive direction of $[-1,1]$.
    \end{enumerate}

    A {\bf state} of $\alpha$ is a map $s\colon
    \partial \alpha\rightarrow\{1,2,3\}$. An $\SL$-tangle with a state is called a {\bf stated $\SL$-tangle}.

    If a (stated) $\SL$-tangle consists of a 
    connected non-empty graph with no 3-valent vertices, so that it is just a properly embedded closed interval with framing (and a state), we call it a {\bf (stated) arc}. 
\end{definition}

We will call the (stated) $\SL$-tangle as {\bf (stated) tangle}. Note that the definition of the (stated) tangle in $\fS\times [-1,1]$ is the same with the web in  Def.\ref{def-web} when $\partial\fS=\emptyset$.

For any $i\in\{1,2,3\}$, define $$
\bar i:=4-i,
$$
the inversion of the index.
The stated $\SL$-skein algebra of $\fS$, denoted as $\cS_{\bar q}(\fS)$, is
the quotient module of the $R$-module freely generated by the set 
of all isotopy classes of stated
tangles in $\fS\times[-1,1]$ subject to  relations \eqref{w.cross}-\eqref{wzh.four} (happening over the interior of $\fS$) and the following extra relations (happening over a neighborhood of the boundary of $\fS$):

\beq\label{wzh.five}
   \raisebox{-.30in}{
\begin{tikzpicture}
\tikzset{->-/.style=
{decoration={markings,mark=at position #1 with
{\arrow{latex}}},postaction={decorate}}}
\filldraw[draw=white,fill=gray!20] (-1,-0.7) rectangle (0.2,1.3);
\draw [line width =1pt](-1,1)--(0,0.3);
\draw [line width =1pt](-1,0.3)--(0,0.3);
\draw [line width =1pt](-1,-0.4)--(0,0.3);
\draw [line width =1.5pt](0.2,1.3)--(0.2,-0.7);
\filldraw[fill=white,line width =0.8pt] (-0.5 ,0.63) circle (0.07);
\filldraw[fill=white,line width =0.8pt] (-0.5 ,0.3) circle (0.07);
\filldraw[fill=white,line width =0.8pt] (-0.5 ,-0.08) circle (0.07);
\end{tikzpicture}}
   = 
   q^{\frac{7}{2}} \sum_{\sigma \in S_3} (-q)^{-\ell(\sigma)}\,  \raisebox{-.30in}{
\begin{tikzpicture}
\tikzset{->-/.style=
{decoration={markings,mark=at position #1 with
{\arrow{latex}}},postaction={decorate}}}
\filldraw[draw=white,fill=gray!20] (-1,-0.7) rectangle (0.2,1.3);
\draw [line width =1pt](-1,1)--(0.2,1);
\draw [line width =1pt](-1,0.3)--(0.2,0.3);
\draw [line width =1pt](-1,-0.4)--(0.2,-0.4);
\draw [line width =1.5pt,decoration={markings, mark=at position 1 with {\arrow{>}}},postaction={decorate}](0.2,1.3)--(0.2,-0.7);
\filldraw[fill=white,line width =0.8pt] (-0.5 ,1) circle (0.07);
\filldraw[fill=white,line width =0.8pt] (-0.5 ,0.3) circle (0.07);
\filldraw[fill=white,line width =0.8pt] (-0.5 ,-0.4) circle (0.07);
\node [right] at(0.2,1) {$\sigma(3)$};
\node [right] at(0.2,0.3) {$\sigma(2)$};
\node [right] at(0.2,-0.4){$\sigma(1)$};
\end{tikzpicture}},
\eeq
\beq \label{wzh.six}
\raisebox{-.20in}{
\begin{tikzpicture}
\tikzset{->-/.style=
{decoration={markings,mark=at position #1 with
{\arrow{latex}}},postaction={decorate}}}
\filldraw[draw=white,fill=gray!20] (-0.7,-0.7) rectangle (0,0.7);
\draw [line width =1.5pt,decoration={markings, mark=at position 1 with {\arrow{>}}},postaction={decorate}](0,0.7)--(0,-0.7);
\draw [color = black, line width =1pt] (0 ,0.3) arc (90:270:0.5 and 0.3);
\node [right]  at(0,0.3) {$i$};
\node [right] at(0,-0.3){$j$};
\filldraw[fill=white,line width =0.8pt] (-0.5 ,0) circle (0.07);
\end{tikzpicture}}   = \delta_{\bar j,i }\,  (-q)^{i-3}q^{-\frac{1}{3}}\ \raisebox{-.20in}{
\begin{tikzpicture}
\tikzset{->-/.style=
{decoration={markings,mark=at position #1 with
{\arrow{latex}}},postaction={decorate}}}
\filldraw[draw=white,fill=gray!20] (-0.7,-0.7) rectangle (0,0.7);
\draw [line width =1.5pt](0,0.7)--(0,-0.7);
\end{tikzpicture}},
\eeq
\beq \label{wzh.seven}
\raisebox{-.20in}{
\begin{tikzpicture}
\tikzset{->-/.style=
{decoration={markings,mark=at position #1 with
{\arrow{latex}}},postaction={decorate}}}
\filldraw[draw=white,fill=gray!20] (-0.7,-0.7) rectangle (0,0.7);
\draw [line width =1.5pt](0,0.7)--(0,-0.7);
\draw [color = black, line width =1pt] (-0.7 ,-0.3) arc (-90:90:0.5 and 0.3);
\filldraw[fill=white,line width =0.8pt] (-0.55 ,0.26) circle (0.07);
\end{tikzpicture}}
= \sum_{i=1}^3  (-q)^{i-1}q^{\frac{1}{3}} \, \raisebox{-.20in}{
\begin{tikzpicture}
\tikzset{->-/.style=
{decoration={markings,mark=at position #1 with
{\arrow{latex}}},postaction={decorate}}}
\filldraw[draw=white,fill=gray!20] (-0.7,-0.7) rectangle (0,0.7);
\draw [line width =1.5pt,decoration={markings, mark=at position 1 with {\arrow{>}}},postaction={decorate}](0,0.7)--(0,-0.7);
\draw [line width =1pt](-0.7,0.3)--(0,0.3);
\draw [line width =1pt](-0.7,-0.3)--(0,-0.3);
\filldraw[fill=white,line width =0.8pt] (-0.3 ,0.3) circle (0.07);
\filldraw[fill=black,line width =0.8pt] (-0.3 ,-0.3) circle (0.07);
\node [right]  at(0,0.3) {$i$};
\node [right]  at(0,-0.3) {$\bar{i}$};
\end{tikzpicture}},
\eeq
\beq\label{wzh.eight}
\raisebox{-.20in}{

\begin{tikzpicture}
\tikzset{->-/.style=

{decoration={markings,mark=at position #1 with

{\arrow{latex}}},postaction={decorate}}}
\filldraw[draw=white,fill=gray!20] (-0,-0.2) rectangle (1, 1.2);
\draw [line width =1.5pt,decoration={markings, mark=at position 1 with {\arrow{>}}},postaction={decorate}](1,1.2)--(1,-0.2);
\draw [line width =1pt](0.6,0.6)--(1,1);
\draw [line width =1pt](0.6,0.4)--(1,0);
\draw[line width =1pt] (0,0)--(0.4,0.4);
\draw[line width =1pt] (0,1)--(0.4,0.6);
\draw[line width =1pt] (0.4,0.6)--(0.6,0.4);
\filldraw[fill=white,line width =0.8pt] (0.2 ,0.2) circle (0.07);
\filldraw[fill=white,line width =0.8pt] (0.2 ,0.8) circle (0.07);
\node [right]  at(1,1) {$i$};
\node [right]  at(1,0) {$j$};
\end{tikzpicture}
} =q^{\frac{1}{3}}\left(\delta_{{j<i} }(q^{-1}-q)\raisebox{-.20in}{

\begin{tikzpicture}
\tikzset{->-/.style=

{decoration={markings,mark=at position #1 with

{\arrow{latex}}},postaction={decorate}}}
\filldraw[draw=white,fill=gray!20] (-0,-0.2) rectangle (1, 1.2);
\draw [line width =1.5pt,decoration={markings, mark=at position 1 with {\arrow{>}}},postaction={decorate}](1,1.2)--(1,-0.2);
\draw [line width =1pt](0,0.8)--(1,0.8);
\draw [line width =1pt](0,0.2)--(1,0.2);
\filldraw[fill=white,line width =0.8pt] (0.2 ,0.8) circle (0.07);
\filldraw[fill=white,line width =0.8pt] (0.2 ,0.2) circle (0.07);
\node [right]  at(1,0.8) {$i$};
\node [right]  at(1,0.2) {$j$};
\end{tikzpicture}
}+q^{-\delta_{i,j}}\raisebox{-.20in}{

\begin{tikzpicture}
\tikzset{->-/.style=

{decoration={markings,mark=at position #1 with

{\arrow{latex}}},postaction={decorate}}}
\filldraw[draw=white,fill=gray!20] (-0,-0.2) rectangle (1, 1.2);
\draw [line width =1.5pt,decoration={markings, mark=at position 1 with {\arrow{>}}},postaction={decorate}](1,1.2)--(1,-0.2);
\draw [line width =1pt](0,0.8)--(1,0.8);
\draw [line width =1pt](0,0.2)--(1,0.2);
\filldraw[fill=white,line width =0.8pt] (0.2 ,0.8) circle (0.07);
\filldraw[fill=white,line width =0.8pt] (0.2 ,0.2) circle (0.07);
\node [right]  at(1,0.8) {$j$};
\node [right]  at(1,0.2) {$i$};
\end{tikzpicture}
}\right),
\eeq
where $i,j\in\{1,2,3\}$,  
$\delta_{j<i}= \left \{
 \begin{array}{rr}
     1,                    & j<i\\
     0,                                 & \text{otherwise}
 \end{array}
 \right.,
\delta_{i,j}= \left \{
 \begin{array}{rr}
     1,                    & i=j\\
     0,                                 & \text{otherwise}
 \end{array}
 \right.$. Each shaded rectangle in the above relations is the projection of a cube in $\fS\times[-1,1]$. The lines contained in the shaded rectangle represent parts of stated $n$-webs with framing  pointing to  readers. The thick line in the edge of shaded rectangle is part of  
 a connected component of $\partial\fS$ and its orientation represents the heights of the endpoints of the stated tangles contained in it such that the arrow goes from lower points to higher points.

For a boundary puncture $p$ of a pb surface $\fS$, 
corner arcs $C(p)_{ij}$ and $\cev{C}(p)_{ij}$ are stated arcs depicted as in Figure \ref{Fig;badarc}.
For a boundary puncture $p$ 
that is not 
the boundary puncture of a component of $\fS$ that is isomorphic to a monogon $\mathbb{P}_1$, i.e. a closed disc minus one boundary puncture 
set 
$$C_p=\{C(p)_{ij}\mid i<j\},\quad\cev{C}_p=\{\cev{C}(p)_{ij}\mid i<j\}.$$  
Each element of $C_p\cup \cev{C}_p$ is called a {\bf bad arc} at $p$. 
\begin{figure}[h]
    \centering
    \includegraphics[width=150pt]{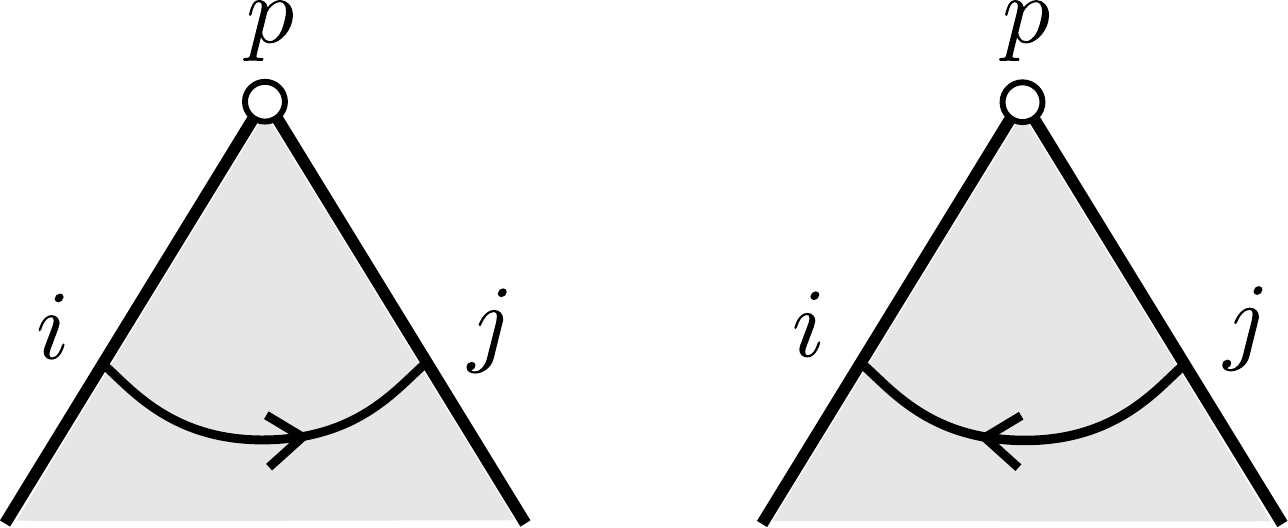}
    \caption{The left is $C(p)_{ij}$ and the right is $\cev{C}(p)_{ij}$.}\label{Fig;badarc}
\end{figure}

For a pb surface $\fS$, $$\overline \cS_{\bar q}(\fS) = \cS_{\bar q}(\fS)/I^{\text{bad}}$$ 
is called the \textbf{reduced stated $\SL$-skein algebra}, defined in \cite{le2023quantum}, where $I^{\text{bad}}$ is the two-sided ideal of $\cS_{\bar q}(\fS)$ generated by all bad arcs. It is these reduced stated ${\rm SL}_3$-skein algebras that we use for the pb surfaces possibly with boundary. Note that when $\partial\fS=\emptyset$, the reduced stated ${\rm SL}_3$-skein algebra $\overline{\cS}_{\bar q}(\fS)$ coincides with the original ${\rm SL}_3$-skein algebra $\cS_{\bar q}(\fS)$ defined in \S\ref{sec2}.

\subsection{The splitting map}
\def\cut{\mathsf{Cut}}
\def\pr{{\bf pr}}
\def\rS{\overline{\cS}_{\bar q}}

Let $e$ be an unoriented path in a pb surface $\fS$ that runs between punctures of $\fS$ (we may call $e$ an ideal arc; see Definition \ref{def.ideal_triangulation} of \S\ref{subsec.coordinate_for_webs}). We do not necessarily require that the two `endpoints' of $e$ be distinct punctures, but here we do require that $e$ (except its `endpoints') lie in the interior of $\fS$.
After cutting $\fS$ along $e$, we get a new pb surface $\cut_e(\fS)$, which has 
boundary components $e_1,e_2$ 
corresponding to $e$ such that 
${\fS}= \cut_e(\fS)/(e_1=e_2)$. We use $\pr_e$ to denote the projection from $\cut_e(\fS)$ to $\fS$.  Suppose that $\alpha$ is  a stated tangle 
in $\fS \times [-1,1]$, 
such that $\alpha$ is transverse to 
$e\times [-1,1]$ and the points of $\alpha$ lying over $e$ have mutually distinct heights.
Let $s$ be a map from $e\cap\alpha$ to $\{1,2,3\}$. 
Denote by $\alpha(s)$ the stated tangle in $\cut_e(\fS)$ defined as the preimage ${\bf pr}_e^{-1}(\alpha)$, where the state value of an endpoint lying over $e_1 \cup e_2$ is induced by $s$ via ${\bf pr}_e$.
Then the splitting map 
$$
\mathbb{S}_e : \mathscr{S}_{\bar q}(\fS) \to \mathscr{S}_{\bar q}(\cut_e(\fS))
$$
is defined by 
\begin{align}\label{eq-def-splitting}
    \mathbb S_e(\alpha) =\sum_{s\colon 
    e\cap \alpha \to \{1,2,3\}} 
    \alpha(s). 
\end{align}
It is shown in \cite{higgins2020triangular,le2021stated}  that $\mathbb{S}_e$ is a well-defined $R$-algebra homomorphism. Note that it is the boundary relations \eqref{wzh.five}--\eqref{wzh.eight} that guarantee the well-definedness of $\mathbb{S}_e$. This splitting map is a crucial tool used in the literature to prove various properties of the (stated) skein algebras.

We can observe that $\mathbb S_e$ sends bad arcs to bad arcs.
So it induces an algebra homomorphism from $\rS(\fS)$
to $\rS(\cut_e(\fS))$, which is still denoted as $\mathbb S_e$.
When there is no confusion we will omit the subscript $e$ 
from $\mathbb S_e$.

\subsection{The quantum trace map for the reduced stated $\SL$-skein algebra}

Recall that Theorem \ref{thm-com-trace}, a sought-for statement, is about compatibility between Frobenius homomorphisms and the quantum trace maps. Although that theorem was only for triangulable pb surfaces without boundary, we will approach this theorem by considering the statement for triangulable pb surfaces possibly with boundary, as outlined in the beginning of this section. In the present subsection we first recall the quantum trace maps for pb surfaces with boundary.

A pb surface is said to be {\bf triangulable} if none of its connected components is a monogon $\mathbb{P}_1$ (i.e. closed disc minus one puncture on boundary), a bigon $\mathbb{P}_2$ (i.e. closed disc minus two punctures on boundary), or a sphere with less than three punctures.
Suppose that $\fS$ is a triangulable pb surface with a triangulation $\lambda$, where a triangulation is defined in Definition \ref{def.ideal_triangulation} (which works for triangulable pb surface with boundary). 
We define a weighted quiver $\mathsf{H}_{\lambda}$ associated to $\lambda$ such that for any (ideal) triangle $\tau\in\mathbb F_\lambda$  the intersection $\mathsf{H}_{\lambda}\cap\tau^{\circ}$ looks like the picture in Figure \ref{quiver} and the weights of all these arrows are $2$, where $\tau^{\circ}$ is the interior of $\tau$.
If $v_{i1}$ and $v_{i2}$ are contained in an boundary edge, we add an arrow from 
 $v_{i1}$ to $v_{i2}$ in Figure \ref{quiver} and the weight of this arrow is $1$. 
 For any arrow $\mathsf{ar}$ in $\mathsf{H}_\lambda$, we use $w(\mathsf{ar})$ to denote its weight.
We use $V_{\lambda}$ to denote the set of vertices of $\mathsf{H}_{\lambda}$. 

Define an anti-symmetric 
integer-valued matrix $ Q_{\lambda} : V_{\lambda}\times V_{\lambda}\rightarrow \mathbb Z$ 
by 
\begin{align}
\label{Q_lambda1}
	\begin{split}
 Q_{\lambda}(v,u) = \sum_{\mathsf{ar}\in\{\text{arrows going from $v$ to $u$}\}}
 w(\mathsf{ar})-
	\sum_{\mathsf{ar}\in\{\text{arrows going from $u$ to $v$}\}}
 w(\mathsf{ar})
 \end{split}
\end{align}
We define $\Xl$ as in equation \eqref{quantum_torus_algebra}.
As in Theorem \ref{thm-trace},
there is an algebra embedding \cite{kim2011sl3,le2023quantum}
 \begin{align}
 \label{tr_X_general}
 {\rm tr}_{\lambda}^{X}:
	\overline{\cS}_{\bar q}(\fS)\rightarrow \Xl,    
 \end{align}
 the {\bf quantum trace map} for a general triangulable pb surface.

\def\rdS{\overline{\cS}_{\bar q}(\fS)}

Suppose that $\fS$ is a triangulable pb surface with a triangulation $\lambda$. 
Let $e$ be an edge 
of $\lambda$ that is not isotopic to an ideal arc wholly contained in the boundary of $\fS$.
We use $\fS'$ to denote the pb surface obtained from $\fS$ by cutting along the ideal arc $e$.
We use ${\bf pr}_e$ to denote the projection from
$\fS'$ to $\fS$. Suppose ${\bf pr}_e^{-1}(e)
=e'\cup e''$.
Then $\lambda'=(\lambda\setminus\{e\})\cup\{e',e''\}$ is a  triangulation for $\fS'$. We say $\lambda'$ is induced from $\lambda$.
There is 
a natural algebra embedding \cite{kim2011sl3,le2023quantum}
\begin{align}\label{eq-cutting}
    \mathcal S_e\colon
    \mathcal X_{\hat q}
    (\fS,\lambda)\rightarrow
    \mathcal X_{\hat q}
    (\fS',\lambda')
\end{align}
given on the generators by
$$
\mathcal S_e(x_v)=
\begin{cases}
    x_v & v\notin e,\\
    [x_{v'}x_{v''}] & v\in e\text{ and ${\bf pr}_e^{-1}(v) = \{v',v''\}$},
\end{cases}
$$
for any $v\in V_\lambda$, where $[\sim]$ is the Weyl-ordering \eqref{Weyl-ordering}.
The quantum trace maps \eqref{tr_X_general} are compatible with the splitting homomorphisms (\eqref{eq-def-splitting} and \eqref{eq-cutting}), in the sense that the following diagram commutes \cite{le2023quantum}:
\begin{equation}\label{eq-com-trace-splitting}
\begin{tikzcd}
\rdS \arrow[r, "\mathbb S_e"]
\arrow[d, "\tr"]  
&  \overline{\cS}_{\bar q}(\fS') \arrow[d, "{\rm tr}_{
\lambda'}^X"] \\
 \mathcal{X}_{\hat q}(\fS,\lambda)
 \arrow[r, "\mathcal S_e"] 
&  
\mathcal{X}_{\hat q}(\fS',
\lambda')
\end{tikzcd}
\end{equation}

The following 
lemma
is perhaps well known. We give a brief proof here, since we did not find it explicitly written in the literature, while we will use this lemma crucially. 
\begin{lemma}\label{lem-inj-splitting}
    The splitting map
    $\mathbb S_e\colon \rdS\rightarrow \overline{\cS}_{\bar q}(\fS')$ is injective.
\end{lemma}
\begin{proof}
    The maps $\tr$, ${\rm tr}_{\lambda_e}^X$, and 
    $\mathcal S_e$ in diagram \eqref{eq-com-trace-splitting} are injective. Then the commutativity of 
    diagram \eqref{eq-com-trace-splitting} completes the proof. 
\end{proof}

\begin{remark}\label{rem-inject-splitting}
    Higgins proved in \cite{higgins2020triangular} the injectivity of the splitting map $\mathbb{S}_e$ for the stated skein algebras, but not that for the reduced versions. 
\end{remark}

Eventually we will prove Theorem \ref{thm-com-trace}, which is on the compatibility of the Frobenius homomorphisms and the quantum trace maps, or in fact a generalized version for triangulable pb surfaces possibly with boundary (Theorem \ref{thm-com-trace-general}), with the help of splitting maps for cutting along ideal arcs. For this we need to study the Frobenius homomorphisms for stated ${\rm SL}_3$-skein algebras for pb surfaces with boundary, which we shall do in the next subsection. Then we will cut a surface into disjoint union of ideal triangles $\mathbb{P}_3$, so that we will just need to prove the theorem for $\mathbb{P}_3$. In the remainder of this subsection we study the quantum trace map for $\mathbb{P}_3$. Since there is only one triangulation $\lambda$ of $\mathbb{P}_3$, we will denote the quantum trace ${\rm tr}^X_\lambda$ as ${\rm tr}^X_{\mathbb{P}_3}$.

Let $\alpha$ be a counterclockwise corner arc in a triangle $\mathbb P_3$ (see Figure \ref{clockwise} for the clockwise corner arc); so there are three kinds of $\alpha$ because there are three corners. For $i,j\in\{1,2,3\}$, we use $\alpha_{ij}$ to denote the stated arc in $\mathbb P_3$ such that the state of $\alpha(0)$ (resp. $\alpha(1)$) is $i$ (resp. $j$).
\begin{lemma}[\cite{le2021stated}]\label{lem-algebra-P3-generator}
    The algebra $\overline\cS_{\bar q}(\mathbb P_3)$ is generated by all these $\alpha_{ij}$.
\end{lemma}
Thus it suffices to study the values ${\rm tr}_{\mathbb P_3}^{X}$ at the elements $\alpha_{ij} \in \overline{\cS}_{\bar q}(\mathbb{P}_3)$. We now recall a method developed by Schrader and Shapiro \cite{schrader} and further by L\^e and Tao \cite{le2023quantum} of studying the values of stated corner arcs under the ${\rm SL}_n$ quantum trace map of $\mathbb{P}_3$. In principle this is not completely necessary as the case of ${\rm SL}_3$ is simple enough, but it makes the argument short and more systematic, and we believe it will be useful when trying to generalize to ${\rm SL}_n$.

\begin{figure}
	\centering
	\includegraphics[width=5cm]{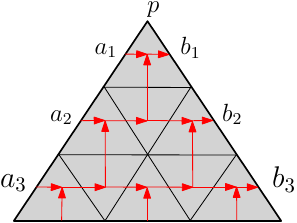}
	\caption{The dual quiver for the one in Figure \ref{quiver}.}\label{dualquiver}
\end{figure}

Choose a distinguished vertex of $\mathbb{P}_3$, i.e. choose one of the three boundary punctures, and call it $p$. Subdivide the triangle $\mathbb{P}_3$ into a `3-triangulation' consisting of $9$ small triangles; the non-boundary edges of this 3-triangulation is drawn by black edges in Figure \ref{dualquiver}, and by purple arrows in Figure \ref{quiver}. Consider a graph dual to this 3-triangulation, and give orientations on the edges as in Figure \ref{quiver}, presented as red arrows there. The resulting oriented graph, consisting of 9 vertices on the boundary of $\mathbb{P}_3$, 9 vertices in the interior, and 18 oriented edges, is referred to as a `network' \cite{schrader}; let's call it $\mathsf D_{\mathbb P_3}$.
 The vertices of $\mathsf D_{\mathbb P_3}$ contained in the two boundary edges of $\mathbb P_3$ 
 adjacent to $p$ are labeled as $a_i$, $b_i$ for $i=1,2,3$, as shown in Figure \ref{dualquiver}.

 For any $i,j\in\{1,2,3\}$, we use 
 $\mathsf P(\mathbb P_3,p,i,j)$ to denote the set of paths in $\mathsf D_{\mathbb P_3}$ going from $a_i$ to $b_j$, where a path is a concatenation of edges of $\mathsf{D}_{\mathbb{P}_3}$ respecting the orientations. For example, there is unique path from $a_2$ to $b_1$, which looks like \raisebox{-0,3\height}{\scalebox{0.8}{
\begingroup%
  \makeatletter%
  \providecommand\color[2][]{%
    \errmessage{(Inkscape) Color is used for the text in Inkscape, but the package 'color.sty' is not loaded}%
    \renewcommand\color[2][]{}%
  }%
  \providecommand\transparent[1]{%
    \errmessage{(Inkscape) Transparency is used (non-zero) for the text in Inkscape, but the package 'transparent.sty' is not loaded}%
    \renewcommand\transparent[1]{}%
  }%
  \providecommand\rotatebox[2]{#2}%
  \newcommand*\fsize{\dimexpr\f@size pt\relax}%
  \newcommand*\lineheight[1]{\fontsize{\fsize}{#1\fsize}\selectfont}%
  \ifx\svgwidth\undefined%
    \setlength{\unitlength}{29.76377953bp}%
    \ifx\svgscale\undefined%
      \relax%
    \else%
      \setlength{\unitlength}{\unitlength * \real{\svgscale}}%
    \fi%
  \else%
    \setlength{\unitlength}{\svgwidth}%
  \fi%
  \global\let\svgwidth\undefined%
  \global\let\svgscale\undefined%
  \makeatother%
  \begin{picture}(1,0.85714286)%
    \lineheight{1}%
    \setlength\tabcolsep{0pt}%
    \put(0,0){\includegraphics[width=\unitlength,page=1]{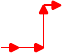}}%
  \end{picture}%
\endgroup%
}}.

 Note that $\mathsf P(\mathbb P_3,p,i,j)=\emptyset$ when $i<j$, and 
 $\mathsf P(\mathbb P_3,p,i,j)$ consists of one element when $i=j$, which we may denote by $\gamma_{ii}$.
 Recall that $V_{\mathbb P_3}$ is the vertex set of the (purple, 3-triangulation) quiver in Figure \ref{quiver}.
 For each path $\gamma\in \mathsf P(\mathbb P_3,p,i,j)$, we define an integer-valued vector
 ${\bf k}^\gamma  =(k^\gamma_v)_{v\in V_{\mathbb{P}_3}} \in\mathbb Z^{V_{\mathbb P_3}}$ 
 as
 \begin{align}\label{eq-def-k-gamma}
 k^\gamma_v =
 \begin{cases}
     1 & \text{if $v$ is on your left when you walk along $\gamma$,}\\
     0 & \text{otherwise}.
 \end{cases}
 \end{align}
 For example, for the unique path $\gamma$ from $a_2$ to $b_1$, there is exactly one $v \in V_{\mathbb{P}_3}$ with $k^\gamma_v=1$, for the vertex $v$ shared by $a_1$ and $a_2$. For the unique path $\gamma_{22}$ from $a_2$ to $b_2$, there are two $v' \in V_{\mathbb{P}_3}$ with $k^{\gamma_{22}}_{v'}=1$.

Define 
\begin{align}\label{eq-k-sum-k1}
    {\bf k}=
    \sum_{i=1,2,3} \, \sum_{\gamma\in \mathsf P(\mathbb P_3,p,i,i)} 
    {\bf k}^\gamma \, \in \mathbb{Z}^{V_{\mathbb{P}_3}}.
\end{align}
Note that 
${\bf k} = \sum_{i=1}^3 {\bf k}^{\gamma_{ii}} = {\bf k}^{\gamma_{22}} + {\bf k}^{\gamma_{33}}$, as ${\bf k}^{\gamma_{11}}=0$.

\begin{remark}
 L{\^e} and Yu defined a vector 
 in $\mathbb{Z}^{V_{\mathbb{P}_3}}$ in 
 \cite[equation (160)]{le2023quantum}, which they denoted by ${\bf k}_1$.
 If we identify the vertex $p$ in Figure \ref{dualquiver} with the vertex $v_1$ in 
  \cite[Figure 10]{le2023quantum}, then we have
     ${\bf k}={\bf k}_1$.
 \end{remark}

The following lemma is the ${\rm SL}_3$ version of the ${\rm SL}_n$ version established in \cite{le2023quantum} based on the development of \cite{schrader}; it gives a neat description of the value of the quantum trace of a stated corner arc in a triangle.
\begin{lemma}[{\cite[Lemma 10.5]{le2023quantum}}]
    Let $\alpha$ be a counterclockwise corner arc in $\mathbb{P}_3$, surrounding the puncture $p$. For $i,j\in\{1,2,3\}$, let $\alpha_{ij} \in \overline{\cS}_{\bar q}(\mathbb{P}_3)$ be the stated corner arc whose state values at the initial and terminal endpoints are $i$ and $j$, respectively. Then we have
    \begin{align}\label{eq-trace-P3-arc}
{\rm tr}_{\mathbb P_3}^{X}(\alpha_{ij})
=\sum_{\gamma\in\mathsf{P}(\mathbb P_3,p,i,j)}
x^{3{\bf k}^\gamma -{\bf k}},
\end{align}
where ${\bf k}^\gamma$ and ${\bf k}$ are defined in \eqref{eq-def-k-gamma} and \eqref{eq-k-sum-k1} respectively. 
\end{lemma}

We perform some more computation to obtain the following, which we will use in our proof of Theorem \ref{thm-com-trace}.

 \begin{lemma}\label{AP-lem-image-arc-trace}
     For $\alpha_{ij}\in \overline\cS_{\bar q}(\mathbb P_3)$, we have 
     $${\rm tr}_{\mathbb P_3}^{X}(\alpha_{ij})
     =\begin{cases}
         0 & i<j,\\
         x^{{\bf t}_1} + x^{{\bf t}_2}
         \text{ for some
         ${\bf t}_1,{\bf t}_2\in \mathbb Z^{V_{\mathbb P_3}}$ such that 
         $\langle {\bf t}_1,{\bf t}_2 \rangle_{Q_\lambda}=18$
         }
         & i=3, j=2,\\
         x^{{\bf t}_{ij}}\text{ for some
         ${\bf t}_{ij}\in \mathbb Z^{V_{\mathbb P_3}}$ }
         & \text{otherwise,}
     \end{cases}$$
     where $x^{{\bf t}_1}$, $x^{{\bf t}_2}$ and $x^{\bf t}$ are 
     Weyl-ordered Laurent monomials defined as in \eqref{x_k}, and $\langle \cdot,\cdot\rangle_{Q_\lambda}$ is as in \eqref{inner_prod_Q} for $\lambda$ being the unique triangulation of $\mathbb{P}_3$.
 \end{lemma}

\begin{proof}
{\bf Case 1}: When $i<j$, then $\gamma\in\mathsf{P}(\mathbb P_3,p,i,j)=\emptyset$. This shows that 
${\rm tr}_{\mathbb P_3}^{X}(\alpha_{ij})=0$.

{\bf Case 2}: When $i=3\text{ and }j=2$, then
$\gamma\in\mathsf{P}(\mathbb P_3,p,i,j)$ consists of two elements. We label these two elements as $\gamma_1$ and $\gamma_2$ such that 
${\bf k}^{\gamma_2}$ has more nonzero entries than ${\bf k}^{\gamma_1}$ (it has exactly one more).
Define 
\begin{align}\label{eq-def-t1-t2}
    {\bf t}_1=3{\bf k}^{\gamma_1}-{\bf k}
    \text{ and }
    {\bf t}_2=3{\bf k}^{\gamma_2}-{\bf k}.
\end{align}
Equation \eqref{eq-trace-P3-arc} implies that 
$${\rm tr}_{\mathbb P_3}^{X}(\alpha_{ij})=
x^{{\bf t}_1} + x^{{\bf t}_2}.$$
It is a trivial check (from \eqref{quantum_torus_relations_vectors}) that 
$\langle {\bf t}_1,{\bf t}_2\rangle_{Q_\lambda} = 18$ (which means that $x^{{\bf t}_1}x^{{\bf t}_2} =q^2 x^{{\bf t}_2}x^{{\bf t}_1}$).

{\bf Case 3}: When $(i,j)$ does not belong to Case 1 or 2, then $\gamma\in\mathsf{P}(\mathbb P_3,p,i,j)$ consists of a single element, denoted 
by $\gamma_{ij}$.
Define 
\begin{align}\label{eq-def-tij}
    {\bf t}_{ij}=3{\bf k}^{\gamma_{ij}}-{\bf k}.
\end{align}
Equation \eqref{eq-trace-P3-arc} implies that 
$${\rm tr}_{\mathbb P_3}^{X}(\alpha_{ij})=
x^{{\bf t}_{ij}}.$$
\end{proof}

\subsection{The Frobenius map for the reduced stated $\SL$-skein algebra}
In the 
remainder of  this section, we will assume that $R=\mathbb C$ and $\hat q=\hat \omega$ is a root of unity. 
 We have $\bar\omega = \hat \omega^{6}$ and $\omega=\hat\omega^{18}$ such that $\bar\omega^{\frac{1}{6}} = \hat\omega$ and $\omega^{\frac{1}{18}}=\hat\omega$. Suppose the order of $\hat \omega^2$ is $N''$.  
 Define $N'=\frac{N''}{\gcd(N'',6)}$ and $N= \frac{N'}{\gcd(N',3)}$.
 Then $N'$ (resp. $N$) is the order of $\bar\omega^2$ (resp. $\omega^2$). 
 Set $\hat\eta = \hat\omega^{N^2}$. Then we have $\bar\eta = \hat\eta^{6}=\bar\omega^{N^2}$ and
$\eta=\hat \eta^{18}=\omega^{N^2} = \pm 1$.

As mentioned before, in order to approach Theorem \ref{thm-com-trace} via cutting the surfaces along ideal arcs, we need Frobenius homomorphisms for stated ${\rm SL}_3$-skein algebras for pb surfaces possibly with boundary, which we recall from literature as follows.

\begin{theorem}\cite{higgins2020triangular,higgins2024miraculous,HW,HLW,wang2023stated}\label{AP-thm-Fro}
    Let $\fS$ be a pb surface. 
    There exists an algebra homomorphism 
    $$\cF\colon
    \cS_{\bar \eta}(\fS)\rightarrow
    \cS_{\bar \omega}(\fS),$$
    called the Frobenius homomorphism,
    satisfying the following properties:
    \begin{enumerate}[label={\rm (\alph*)}]\itemsep0,3em
    
	\item\label{AP-thm-Fro-a} Let $\alpha = \cup_{1\leq t\leq k}\alpha_t$ be a stated tangle (Definition \ref{def.SL3-tangle}) in the thickened surface $\widetilde{\fS}=\fS \times [-1,1]$ such that each
	$\alpha_t$  is a stated arc.
 Then 
 $$\mathcal{F}(l) = \cup_{1\le t \le m} \alpha_t^{(N)},$$
	where $\alpha_t^{(N)}$ denotes the union of $N$ parallel copies of $\alpha_t$ taken in the framing direction.
    
	\item\label{AP-thm-Fro-b} Let $l = \cup_{1\leq t\leq m}l_t$ be a stated tangle in 
    $\widetilde{\fS}$ such that each
	$l_t$  is a framed knot. Then 
 $$\mathcal{F}(l) = \cup_{1\le t \le m} l_t^{[P_{N,1}]},$$
 where the right hand side is given by the threading operation in \S\ref{subsec.threading}.

 \item\label{AP-thm-Fro-c} The Frobenius homomorphism $\cF$ commutes with  the splitting maps for stated ${\rm SL}_3$-skein algebras.

 \item\label{AP-thm-Fro-d} The algebra homomorphism $\cF\colon
    \cS_{\bar \eta}(\fS)\rightarrow
    \cS_{\bar \omega}(\fS)$ between the stated ${\rm SL}_3$-skein algebras induces an algebra homomorphism 
    $$\cF\colon
    \overline{\cS}_{\bar \eta}(\fS)\rightarrow
    \overline{\cS}_{\bar \omega}(\fS)$$
    between the reduced stated ${\rm SL}_3$-skein algebras,
    which commutes with the splitting maps for the reduced stated ${\rm SL}_3$-skein algebras.

    \item\label{AP-Fro-inj-nonreduced} $\cF\colon
    \cS_{\bar \eta}(\fS)\rightarrow
    \cS_{\bar \omega}(\fS)$ is injective when every component of $\fS$ contains at least one puncture.
 \end{enumerate}
\end{theorem}

One point that we want to make in this section is that we come up with proofs that are independent of the upcoming paper \cite{HLW}. So it is worth mentioning that each item of the above theorem is proved in the literature, indepedent of \cite{HLW}. Indeed, note that the items \ref{AP-thm-Fro-a} and \ref{AP-Fro-inj-nonreduced} are
shown in \cite{higgins2020triangular,wang2023stated,HW}, (b) in \cite{higgins2024miraculous}, (c) in \cite{higgins2020triangular,wang2023stated,HW}, and (d) in \cite{wang2023stated,HW}.

\begin{remark}
\label{rem.on_condition_of_N_prime}
    In \cite{higgins2020triangular,higgins2024miraculous,wang2023stated} the authors require $\gcd(N',6) = 1$ (which implies $\bar\eta=1$ in particular),  whereas in \cite{HLW} Theorem \ref{AP-thm-Fro} is proved without assuming this condition.
    Arguments in \cite{wang2023stated} work for all roots of unity. In particular, there are no restrictions for the ground ring and the quantum parameter in \cite[Lemmas 7.3, 7.6, and 7.7]{wang2023stated}.
\end{remark}

\begin{remark}
Item \ref{AP-Fro-inj-nonreduced} is not used in this section, but we leave it because we think it's still good to have it.
\end{remark}

\subsection{Proof for Theorem \ref{thm-com-trace}}

We now assemble what we have gathered so far, to prove Theorem \ref{thm-com-trace}, which is on the compatibility of the Frobenius homomorphisms of ${\rm SL}_3$-skein algebras and the quantum trace maps. As said, the strategy is to prove the following stronger version, going through reduced stated ${\rm SL}_3$-skein algebras of pb surfaces with boundary, using splitting maps.
\begin{theorem}[compatibility of Frobenius homomorphisms and splitting maps]\label{thm-com-trace-general}
    Theorem \ref{thm-com-trace} holds for any triangulable pb surface $\fS$ possibly with boundary.
\end{theorem}

We begin with the base case when $\fS$ is a triangle $\mathbb{P}_3$.
\begin{lemma}\label{AP-lem-key-P3}
    Theorem \ref{thm-com-trace-general} holds when $\fS$ is a 
 triangle $\mathbb{P}_3$. That is, the following diagram commutes
    \begin{equation}\label{eq-tr-cF-triangle}
		\begin{tikzcd}
			\overline\cS_{\hat \eta}(\mathbb P_3)  \arrow[r, "\cF"]
			\arrow[d, "{\rm tr}_{\mathbb P_3}^X"]  
			&  \overline\cS_{\hat \omega}(\mathbb P_3)  \arrow[d, "{\rm tr}_{\mathbb P_3}^X"] \\
			\mathcal X_{\hat \eta}(\mathbb P_3) \arrow[r, "F"] 
			& \mathcal X_{\hat \omega}(\mathbb P_3), 
		\end{tikzcd}
	\end{equation}
 where $F$ is defined in \eqref{eq-Fro-quantum-tori}, sending each generator $x_v \in \mathcal{X}_{\hat{\eta}}(\mathbb{P}_3)$ (for $v\in V_{\mathbb{P}_3}$) to $x_v^N \in \mathcal{X}_{\hat{\omega}}(\mathbb{P}_3)$.
\end{lemma}
\begin{proof}
    From Lemma \ref{lem-algebra-P3-generator}, it follows that it suffices to show that
    $$F({\rm tr}_{
    \mathbb{P}_3}^X(\alpha_{ij}))
    ={\rm tr}_{
    \mathbb{P}_3}^X(\cF(\alpha_{ij})),$$
    where $\alpha_{ij}$ is a counterclockwise stated corner arc with states $i$ and $j$ at the initial and terminal endpoints. Theorem \ref{AP-thm-Fro}\ref{AP-thm-Fro-a} shows that $\cF(\alpha_{ij}) =\alpha_{ij}^{(N)}=\alpha_{ij}^N$. Then we have
    ${\rm tr}_{
    \mathbb{P}_3}^X(\cF(\alpha_{ij}))
    =({\rm tr}_{    \mathbb{P}_3}^X(\alpha_{ij}))^N,$ because ${\rm tr}^X_{\mathbb{P}_3}$ is a homomorphism. 
    Thus it suffices to show that
    \begin{align}
        \label{Fro-com-trace_P3_eq}
        F({\rm tr}_{
    \mathbb{P}_3}^X(\alpha_{ij}))=
    ({\rm tr}_{
    \mathbb{P}_3}^X(\alpha_{ij}))^N.
    \end{align}
Lemma \ref{AP-lem-image-arc-trace} tells us that if $i\neq 3$ or $j \neq 2$, then ${\rm tr}^X_{\mathbb{P}_3}(\alpha_{ij})$ is either $0$ or a Laurent monomial $x^{{\bf t}_{ij}}$, so \eqref{Fro-com-trace_P3_eq} is easy to check; one may want to use basic facts about Weyl-ordered Laurent monomials, such as $(x^v)^m = x^{mv}$ for any $v\in \mathbb{Z}^{\mathbb{P}_3}$ and $m\in \mathbb{Z}$, which follows from \eqref{quantum_torus_relations_vectors}.

When $i=3$ and $j=2$, Lemma \ref{AP-lem-image-arc-trace} says that there exist ${\bf t}_1,{\bf t}_2\in \mathbb Z^{V_{\mathbb P_3}}$ with   $\langle{\bf t}_1, {\bf t}_2 \rangle_{Q_\lambda}=18$, such that 
    ${\rm tr}_{
    \mathbb{P}_3}^X(\alpha_{32})
    =x^{{\bf t}_1} + x^{{\bf t}_2}$ holds, for both of the two vertical maps ${\rm tr}^X_{\mathbb{P}_3}$ in the diagram \eqref{eq-tr-cF-triangle}. For the left vertical map ${\rm tr}^X_{\mathbb{P}_3}$, we have
    $$F({\rm tr}_{
    \mathbb{P}_3}^X(\alpha_{ij}))=
    F(x^{{\bf t}_1} + x^{{\bf t}_2})=
    x^{N{\bf t}_1} + x^{N{\bf t}_2}. 
    $$
    Here, $F(x^{{\bf t}_i}) = x^{N {\bf t}_i}$, $i=1,2$, is a straightforward exercise about Weyl-ordered Laurent monomials. On the other hand, for the right vertical map ${\rm tr}^X_{\mathbb{P}_3}$, note that
    $$
    ({\rm tr}_{\lambda}^X(\alpha_{ij}))^N = (x^{{\bf t}_1} + x^{{\bf t}_2})^N = x^{N{\bf t}_1} + x^{ N {\bf t}_2},
    $$
    where the last equality is an easy exercise using $x^{{\bf t}_1} x^{{\bf t}_2} = \omega^2 x^{{\bf t}_2} x^{{\bf t}_1}$ and the fact that $\omega^2$ is a primitive $N$-th root of unity.
\end{proof}

Before we finally present our proof of Theorem \ref{thm-com-trace-general}, we need to establish the following, which is not written in the literature. 

\begin{proposition}
    Suppose that $\fS$ is a triangulable pb surface. 
    The Frobenius map
    $$\cF\colon
    \overline{\cS}_{\bar \eta}(\fS)\rightarrow
    \overline{\cS}_{\bar \omega}(\fS)$$ between the reduced stated ${\rm SL}_3$-skein algebras is injective. 
\end{proposition}
\begin{proof}
    First we look at the case when
    $\fS$ is an ideal triangle $\mathbb{P}_3$. The quantum trace maps and $F$ in the diagram \eqref{eq-tr-cF-triangle} are all injective. The commutativity of the diagram \eqref{eq-tr-cF-triangle} shows that 
    $\cF\colon
    \overline{\cS}_{\bar \eta}(\mathbb P_3)\rightarrow
    \overline{\cS}_{\bar \omega}(\mathbb P_3)$ is injective. 
    Using the facts that the Frobenius map commutes with the splitting map, that the splitting map is injective (Lemma \ref{lem-inj-splitting}), and that the Frobenius map is injective for $\mathbb P_3$, we can show that
    $\cF\colon
    \overline{\cS}_{\bar \eta}(\fS)\rightarrow
    \overline{\cS}_{\bar \omega}(\fS)$ is injective by cutting $\fS$ into a
    disjoint union of ideal triangles.
\end{proof}

\begin{proof}[Proof for Theorem 
\ref{thm-com-trace-general}]
  Note that both the quantum trace map and the Frobenius map commute with the splitting map. 
  By cutting $\fS$ into a 
  disjoint union of triangles,
   Lemma \ref{AP-lem-key-P3} implies Theorem 
   \ref{thm-com-trace-general} because the splitting map, the quantum trace map, and the Frobenius map are all injective. 
\end{proof}

In particular, Theorem \ref{thm-com-trace} is now proved, independently of \cite{HLW}.

\begin{remark}
    In our proof above, we used Theorem \ref{AP-thm-Fro}, which is proved in the literature so far with the assumption ${\rm gcd}(N',6)=1$, which is dropped only in \cite{HLW}; see Remark \ref{rem.on_condition_of_N_prime}. Other than Theorem \ref{AP-thm-Fro}, the arguments we gave in this subsection does not depend on this assumption.
\end{remark}

\subsection{The root of unity transparency}

In the remainder of this section, we provide a proof of Proposition \ref{prop-center}, independently of \cite{HLW}. Recall that this proposition says that certain elements in the image of the Frobenius map $\mathcal{F} : \mathscr{S}_{\bar{\eta}}(\fS) \to \mathscr{S}_{\bar{\omega}}(\fS)$ is in the center of the codomain algebra.

Let $\fS$ be a pb surface with $\partial\fS\neq\emptyset$.
Let $\alpha$ be a (stated) arc in the thickened surface $\widetilde{\fS}=\fS\times[-1,1]$.
For a nonnegative integer $m$, 
recall that $\alpha^{(m)}$ denote the disjoint union of $m$ parallel copies of $\alpha$ (taken
in the direction of the framing). 
Define $\alpha^{(0)}$ to be the empty tangle.
We use 
$
\begin{tikzpicture}
\tikzset{->-/.style=
{decoration={markings,mark=at position #1 with
{\arrow{latex}}},postaction={decorate}}}
\draw [color = black, line width =1pt](0.6,0) --(0,0);
\draw[line width =1pt,decoration={markings, mark=at position 0.8 with {\arrow{>}}},postaction={decorate}](1,0) --(1.5,0);
\node at(0.8,0) {\small $m$};
\draw[color=black] (0.8,0) circle(0.2);
\end{tikzpicture}
$
to denote the 
$m$ parallel copies of
$
\begin{tikzpicture}
\tikzset{->-/.style=
{decoration={markings,mark=at position #1 with
{\arrow{latex}}},postaction={decorate}}}
\draw [color = black, line width =1pt](0.6,0) --(1,0);
\draw[line width =1pt,decoration={markings, mark=at position 0.8 with {\arrow{>}}},postaction={decorate}](1,0) --(1.5,0);
\end{tikzpicture}
$.

\begin{lemma}\cite[Lemmas 7.2 and 7.3]{wang2023stated}\label{lem-transparency-P4}
    In $\cS_{\bar\omega}(\mathbb P_4)$, we have the following equations:
    \begin{equation}\label{eqq14}
\raisebox{-.10in}{
\begin{tikzpicture}
\tikzset{->-/.style=
{decoration={markings,mark=at position #1 with
{\arrow{latex}}},postaction={decorate}}}
\filldraw[draw=black,fill=gray!20] (-0.7,-0.5) rectangle (1.3,0.5);
\draw [color = black, line width =1pt](0,0) --(-0.7,0);
\draw [color = black, line width =1pt](0,0) --(0.4,0);
\draw[line width =1pt,decoration={markings, mark=at position 0.8 with {\arrow{>}}},postaction={decorate}](0.8,0) --(1.3,0);
\draw[line width =1pt,decoration={markings, mark=at position 0.8 with {\arrow{>}}},postaction={decorate}](0,0.1) --(0,0.5);
\draw [color = black, line width =1pt](0,-0.5) --(0,-0.1);
\filldraw[draw=black,fill=gray!20] (0.6,0) circle(0.25);
\node at(0.6,0) {\small $N$};
\end{tikzpicture}}
=
 \omega^{\frac{2N}{3}}
\raisebox{-.10in}{
\begin{tikzpicture}
\tikzset{->-/.style=
{decoration={markings,mark=at position #1 with
{\arrow{latex}}},postaction={decorate}}}
\filldraw[draw=black,fill=gray!20] (-0.7,-0.5) rectangle (1.3,0.5);
\draw [color = black, line width =1pt](-0.1,0) --(-0.7,0);
\draw [color = black, line width =1pt](0.1,0) --(0.4,0);
\draw[line width =1pt,decoration={markings, mark=at position 0.8 with {\arrow{>}}},postaction={decorate}](0.8,0) --(1.3,0);
\draw[line width =1pt,decoration={markings, mark=at position 0.8 with {\arrow{>}}},postaction={decorate}](0,0) --(0,0.5);
\draw [color = black, line width =1pt](0,-0.5) --(0,0);
\filldraw[draw=black,fill=gray!20] (0.6,0) circle(0.25);
\node at(0.6,0) {\small $N$};
\end{tikzpicture}}\;
\end{equation}
and 
\begin{equation}\label{eqq14}
\raisebox{-.10in}{
\begin{tikzpicture}
\tikzset{->-/.style=
{decoration={markings,mark=at position #1 with
{\arrow{latex}}},postaction={decorate}}}
\filldraw[draw=black,fill=gray!20] (-0.7,-0.5) rectangle (1.3,0.5);
\draw [color = black, line width =1pt](0,0) --(-0.7,0);
\draw [color = black, line width =1pt](0,0) --(0.4,0);
\draw[line width =1pt,decoration={markings, mark=at position 0.8 with {\arrow{>}}},postaction={decorate}](0.8,0) --(1.3,0);
\draw[color = black, line width =1pt](0,0.1) --(0,0.5);
\draw [line width =1pt,decoration={markings, mark=at position 0.5 with {\arrow{<}}},postaction={decorate}](0,-0.5) --(0,-0.1);
\filldraw[draw=black,fill=gray!20] (0.6,0) circle(0.25);
\node at(0.6,0) {\small $N$};
\end{tikzpicture}}
=
 \omega^{-\frac{2N}{3}}
\raisebox{-.10in}{
\begin{tikzpicture}
\tikzset{->-/.style=
{decoration={markings,mark=at position #1 with
{\arrow{latex}}},postaction={decorate}}}
\filldraw[draw=black,fill=gray!20] (-0.7,-0.5) rectangle (1.3,0.5);
\draw [color = black, line width =1pt](-0.1,0) --(-0.7,0);
\draw [color = black, line width =1pt](0.1,0) --(0.4,0);
\draw[line width =1pt,decoration={markings, mark=at position 0.8 with {\arrow{>}}},postaction={decorate}](0.8,0) --(1.3,0);
\draw[line width =1pt,decoration={markings, mark=at position 0.92 with {\arrow{>}}},postaction={decorate}](0,0.5) --(0,-0.5);
%
\filldraw[draw=black,fill=gray!20] (0.6,0) circle(0.25);
\node at(0.6,0) {\small $N$};
\end{tikzpicture}},
\end{equation}
where the states of the arcs are arbitrary. 
\end{lemma}

Suppose that
$
\begin{tikzpicture}
\tikzset{->-/.style=
{decoration={markings,mark=at position #1 with
{\arrow{latex}}},postaction={decorate}}}
\draw [color = black, line width =1pt](0.6,0) --(1,0);
\draw[line width =1pt,decoration={markings, mark=at position 0.8 with {\arrow{>}}},postaction={decorate}](1,0) --(1.5,0);
\end{tikzpicture}
$ is a part of a framed knot $\alpha$ in $\widetilde{\fS}=\fS\times[-1,1]$.
We use 
$
\raisebox{-.10in}{
\begin{tikzpicture}
\tikzset{->-/.style=
{decoration={markings,mark=at position #1 with
{\arrow{latex}}},postaction={decorate}}}
\draw [color = black, line width =1pt](0.4,0) --(-0.2,0);
\draw[line width =1pt,decoration={markings, mark=at position 0.8 with {\arrow{>}}},postaction={decorate}](1,0) --(1.5,0);
\draw[color=black,fill=white] (0.2,-0.3) rectangle (1,0.3);
\node at(0.6,0) {\small $P_{N,1}$};
\end{tikzpicture}}
$
to denote $\alpha^{[P_{N,1}]}$ (see 
\S\ref{subsec.threading}).

\begin{proposition}\label{prop-transparency-pb}
Let $\fS$ be a pb surface with $\partial\fS=\emptyset$.
    In $\cS_{\bar\omega}(\fS)$, we have the following equations:
    \begin{equation}\label{transparency-fS-1}
\raisebox{-.10in}{
\begin{tikzpicture}
\tikzset{->-/.style=
{decoration={markings,mark=at position #1 with
{\arrow{latex}}},postaction={decorate}}}
\filldraw[draw=white,fill=gray!20] (-0.7,-0.5) rectangle (2,0.5);
\draw [color = black, line width =1pt](0,0) --(-0.7,0);
\draw [color = black, line width =1pt](0,0) --(0.4,0);
\draw[line width =1pt,decoration={markings, mark=at position 0.85 with {\arrow{>}}},postaction={decorate}](0.8,0) --(2,0);
\draw[line width =1pt,decoration={markings, mark=at position 0.8 with {\arrow{>}}},postaction={decorate}](0,0.1) --(0,0.5);
\draw [color = black, line width =1pt](0,-0.5) --(0,-0.1);
\filldraw[draw=black,fill=gray!20](0.4,-0.3) rectangle (1.2,0.3);
\node at(0.8,0) {\small $P_{N,1}$};
\end{tikzpicture}}
=
 \omega^{\frac{2N}{3}}
\raisebox{-.10in}{
\begin{tikzpicture}
\tikzset{->-/.style=
{decoration={markings,mark=at position #1 with
{\arrow{latex}}},postaction={decorate}}}
\filldraw[draw=white,fill=gray!20] (-0.7,-0.5) rectangle (2,0.5);
\draw [color = black, line width =1pt](-0.1,0) --(-0.7,0);
\draw [color = black, line width =1pt](0.1,0) --(0.4,0);
\draw[line width =1pt,decoration={markings, mark=at position 0.8 with {\arrow{>}}},postaction={decorate}](0.8,0) --(2,0);
\draw[line width =1pt,decoration={markings, mark=at position 0.8 with {\arrow{>}}},postaction={decorate}](0,0) --(0,0.5);
\draw [color = black, line width =1pt](0,-0.5) --(0,0);
\filldraw[draw=black,fill=gray!20](0.4,-0.3) rectangle (1.2,0.3);
\node at(0.8,0) {\small $P_{N,1}$};
\end{tikzpicture}}\;
\end{equation}
and 
\begin{equation}\label{transparency-fS-2}
\raisebox{-.10in}{
\begin{tikzpicture}
\tikzset{->-/.style=
{decoration={markings,mark=at position #1 with
{\arrow{latex}}},postaction={decorate}}}
\filldraw[draw=white,fill=gray!20] (-0.7,-0.5) rectangle (2,0.5);
\draw [color = black, line width =1pt](0,0) --(-0.7,0);
\draw [color = black, line width =1pt](0,0) --(0.4,0);
\draw[line width =1pt,decoration={markings, mark=at position 0.85 with {\arrow{>}}},postaction={decorate}](0.8,0) --(2,0);
\draw[color = black, line width =1pt](0,0.1) --(0,0.5);
\draw [line width =1pt,decoration={markings, mark=at position 0.5 with {\arrow{<}}},postaction={decorate}](0,-0.5) --(0,-0.1);
\filldraw[draw=black,fill=gray!20](0.4,-0.3) rectangle (1.2,0.3);
\node at(0.8,0) {\small $P_{N,1}$};
\end{tikzpicture}}
=
 \omega^{-\frac{2N}{3}}
\raisebox{-.10in}{
\begin{tikzpicture}
\tikzset{->-/.style=
{decoration={markings,mark=at position #1 with
{\arrow{latex}}},postaction={decorate}}}
\filldraw[draw=white,fill=gray!20] (-0.7,-0.5) rectangle (2,0.5);
\draw [color = black, line width =1pt](-0.1,0) --(-0.7,0);
\draw [color = black, line width =1pt](0.1,0) --(0.4,0);
\draw[line width =1pt,decoration={markings, mark=at position 0.8 with {\arrow{>}}},postaction={decorate}](0.8,0) --(2,0);
\draw[line width =1pt,decoration={markings, mark=at position 0.92 with {\arrow{>}}},postaction={decorate}](0,0.5) --(0,-0.5);
%
\filldraw[draw=black,fill=gray!20](0.4,-0.3) rectangle (1.2,0.3);
\node at(0.8,0) {\small $P_{N,1}$};
\end{tikzpicture}},
\end{equation}
where each horizontal arc is a part of knot $\alpha$ and the vertical arc is a part of web $\beta$ with $\alpha\neq\beta$. 
\end{proposition}
\begin{proof}
We only 
prove equation \eqref{transparency-fS-1} since 
a similar argument works for \eqref{transparency-fS-2}.

    Because of the relations \eqref{wzh.four}, we can suppose that $\beta$ is a framed knot. 
    There exists a regular open neighborhood $U$ of $\alpha\cup \beta$ such that $U$ is isomorphic 
    to the once punctured torus $T_{1,1}$ times an open interval $(-1,1)$ and the two webs in \eqref{transparency-fS-1} look like
    $$\raisebox{-.15in}{
\begin{tikzpicture}
\tikzset{->-/.style=
{decoration={markings,mark=at position #1 with
{\arrow{latex}}},postaction={decorate}}}
\filldraw[draw=white,fill=gray!20] (-0.7,-0.5) rectangle (2,0.5);
\draw [color = red, line width =1pt](-0.5,-0.5) --(1.8,-0.5);
\draw [color = red, line width =1pt](-0.5,0.5) --(1.8,0.5);
\draw [color = blue, line width =1pt](-0.7,-0.4) --(-0.7,0.4);
\draw [color = blue, line width =1pt](2,-0.4) --(2,0.4);
\draw [color = black, line width =1pt](0,0) --(-0.7,0);
\draw [color = black, line width =1pt](0,0) --(0.4,0);
\draw[line width =1pt,decoration={markings, mark=at position 0.85 with {\arrow{>}}},postaction={decorate}](0.8,0) --(2,0);
\draw[line width =1pt,decoration={markings, mark=at position 0.8 with {\arrow{>}}},postaction={decorate}](0,0.1) --(0,0.5);
\draw [color = black, line width =1pt](0,-0.5) --(0,-0.1);
\filldraw[draw=black,fill=gray!20](0.4,-0.3) rectangle (1.2,0.3);
\node at(0.8,0) {\small $P_{N,1}$};
\end{tikzpicture}}\text{ and }\raisebox{-.15in}{
\begin{tikzpicture}
\tikzset{->-/.style=
{decoration={markings,mark=at position #1 with
{\arrow{latex}}},postaction={decorate}}}
\filldraw[draw=white,fill=gray!20] (-0.7,-0.5) rectangle (2,0.5);
\draw [color = red, line width =1pt](-0.5,-0.5) --(1.8,-0.5);
\draw [color = red, line width =1pt](-0.5,0.5) --(1.8,0.5);
\draw [color = blue, line width =1pt](-0.7,-0.4) --(-0.7,0.4);
\draw [color = blue, line width =1pt](2,-0.4) --(2,0.4);
\draw [color = black, line width =1pt](-0.1,0) --(-0.7,0);
\draw [color = black, line width =1pt](0.1,0) --(0.4,0);
\draw[line width =1pt,decoration={markings, mark=at position 0.8 with {\arrow{>}}},postaction={decorate}](0.8,0) --(2,0);
\draw[line width =1pt,decoration={markings, mark=at position 0.8 with {\arrow{>}}},postaction={decorate}](0,0) --(0,0.5);
\draw [color = black, line width =1pt](0,-0.5) --(0,0);
\filldraw[draw=black,fill=gray!20](0.4,-0.3) rectangle (1.2,0.3);
\node at(0.8,0) {\small $P_{N,1}$};
\end{tikzpicture}},
$$
where the two red lines (or two blue lines) in each picture are glued together; see Figure \ref{picture-U}. This is a standard trick used in the literature; see e.g. \cite[Fig.2]{bullock} to see that $U$ appearing in Figure \ref{picture-U} is isomorphic to $T_{1,1} \times (-1,1)$.
Then $\alpha$ and $\beta$ in $T_{1,1}$ look like 
$\raisebox{-.10in}{
\begin{tikzpicture}
\tikzset{->-/.style=
{decoration={markings,mark=at position #1 with
{\arrow{latex}}},postaction={decorate}}}
\filldraw[draw=white,fill=gray!20] (-0.7,-0.5) rectangle (2,0.5);
\draw [color = red, line width =1pt](-0.5,-0.5) --(1.8,-0.5);
\draw [color = red, line width =1pt](-0.5,0.5) --(1.8,0.5);
\draw [color = blue, line width =1pt](-0.7,-0.4) --(-0.7,0.4);
\draw [color = blue, line width =1pt](2,-0.4) --(2,0.4);
\draw [color = black, line width =1pt](0,0) --(-0.7,0);
\draw [color = black, line width =1pt](0,0) --(0.8,0);
\draw[line width =1pt,decoration={markings, mark=at position 0.85 with {\arrow{>}}},postaction={decorate}](0.8,0) --(2,0);
%
\end{tikzpicture}}
\text{ and }
\raisebox{-.10in}{
\begin{tikzpicture}
\tikzset{->-/.style=
{decoration={markings,mark=at position #1 with
{\arrow{latex}}},postaction={decorate}}}
\filldraw[draw=white,fill=gray!20] (-0.7,-0.5) rectangle (2,0.5);
\draw [color = red, line width =1pt](-0.5,-0.5) --(1.8,-0.5);
\draw [color = red, line width =1pt](-0.5,0.5) --(1.8,0.5);
\draw [color = blue, line width =1pt](-0.7,-0.4) --(-0.7,0.4);
\draw [color = blue, line width =1pt](2,-0.4) --(2,0.4);
\draw[line width =1pt,decoration={markings, mark=at position 0.8 with {\arrow{>}}},postaction={decorate}](0.55,-0.5) --(0.55,0.5);
%
\end{tikzpicture}}
$ respectively. 
Suppose that the red (resp. blue) line in $T_{1,1}$ is denoted as $e_1$ (resp. $e_2$). Then $e_1$ and $e_2$ are two ideal arcs in $T_{1,1}$, which is viewed as a pb surface.

\begin{figure}
	\centering
	\includegraphics[width=5cm]{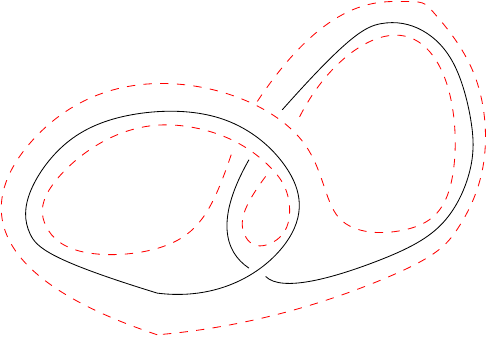}
	\caption{The illustration for  the regular open neighborhood $U$ of $\alpha\cup \beta$.}\label{picture-U}
\end{figure}

We get $\mathbb P_4$ if we cut $T_{1,1}$ along 
$e_1$ and $e_2$ (in any order). Then we have 
\begin{align*}
    &\mathbb S_{e_1}(\mathbb S_{e_2}(\raisebox{-.15in}{
\begin{tikzpicture}
\tikzset{->-/.style=
{decoration={markings,mark=at position #1 with
{\arrow{latex}}},postaction={decorate}}}
\filldraw[draw=white,fill=gray!20] (-0.7,-0.5) rectangle (2,0.5);
\draw [color = red, line width =1pt](-0.5,-0.5) --(1.8,-0.5);
\draw [color = red, line width =1pt](-0.5,0.5) --(1.8,0.5);
\draw [color = blue, line width =1pt](-0.7,-0.4) --(-0.7,0.4);
\draw [color = blue, line width =1pt](2,-0.4) --(2,0.4);
\draw [color = black, line width =1pt](0,0) --(-0.7,0);
\draw [color = black, line width =1pt](0,0) --(0.4,0);
\draw[line width =1pt,decoration={markings, mark=at position 0.85 with {\arrow{>}}},postaction={decorate}](0.8,0) --(2,0);
\draw[line width =1pt,decoration={markings, mark=at position 0.8 with {\arrow{>}}},postaction={decorate}](0,0.1) --(0,0.5);
\draw [color = black, line width =1pt](0,-0.5) --(0,-0.1);
\filldraw[draw=black,fill=gray!20](0.4,-0.3) rectangle (1.2,0.3);
\node at(0.8,0) {\small $P_{N,1}$};
\end{tikzpicture}}))\\
    =&\mathbb S_{e_1}(\mathbb S_{e_2}(\cF(\alpha)\beta))\\
    =& \mathbb S_{e_1}(\mathbb S_{e_2}(\cF(\alpha)))
    \mathbb S_{e_1}(\mathbb S_{e_2}(\beta))\\
    =&\cF(\mathbb S_{e_1}(\mathbb S_{e_2}(\alpha)))
    \mathbb S_{e_1}(\mathbb S_{e_2}(\beta))
    \quad (\because\mbox{Theorem \ref{AP-thm-Fro}\ref{AP-thm-Fro-c})}\\
    =&\sum_{a,b\in\{1,2,3\}} \raisebox{-.25in}{
\begin{tikzpicture}
\tikzset{->-/.style=
{decoration={markings,mark=at position #1 with
{\arrow{latex}}},postaction={decorate}}}
\filldraw[draw=black,fill=gray!20] (-0.7,-0.5) rectangle (1.3,0.5);
\draw [color = black, line width =1pt](0,0) --(-0.7,0);
\draw [color = black, line width =1pt](0,0) --(0.4,0);
\draw[line width =1pt,decoration={markings, mark=at position 0.8 with {\arrow{>}}},postaction={decorate}](0.8,0) --(1.3,0);
\draw[line width =1pt,decoration={markings, mark=at position 0.8 with {\arrow{>}}},postaction={decorate}](0,0.1) --(0,0.5);
\draw [color = black, line width =1pt](0,-0.5) --(0,-0.1);
\filldraw[draw=black,fill=gray!20] (0.6,0) circle(0.25);
\node at(0.6,0) {\small $N$};
\node[left] at(-0.7,0) {\small $a$};
\node[right] at(1.3,0) {\small $a$};
\node[below] at(0,-0.5) {\small $b$};
\node[above] at(0,0.5) {\small $b$};
\end{tikzpicture}}\in\cS_{\bar\omega}(\mathbb P_4)
    \quad(\because\mbox{Theorem \ref{AP-thm-Fro}\ref{AP-thm-Fro-a}).}\\
\end{align*}
Similarly, we can show 
\begin{align*}
    \mathbb S_{e_1}(\mathbb S_{e_2}(\raisebox{-.15in}{
\begin{tikzpicture}
\tikzset{->-/.style=
{decoration={markings,mark=at position #1 with
{\arrow{latex}}},postaction={decorate}}}
\filldraw[draw=white,fill=gray!20] (-0.7,-0.5) rectangle (2,0.5);
\draw [color = red, line width =1pt](-0.5,-0.5) --(1.8,-0.5);
\draw [color = red, line width =1pt](-0.5,0.5) --(1.8,0.5);
\draw [color = blue, line width =1pt](-0.7,-0.4) --(-0.7,0.4);
\draw [color = blue, line width =1pt](2,-0.4) --(2,0.4);
\draw [color = black, line width =1pt](-0.1,0) --(-0.7,0);
\draw [color = black, line width =1pt](0.1,0) --(0.4,0);
\draw[line width =1pt,decoration={markings, mark=at position 0.8 with {\arrow{>}}},postaction={decorate}](0.8,0) --(2,0);
\draw[line width =1pt,decoration={markings, mark=at position 0.8 with {\arrow{>}}},postaction={decorate}](0,0) --(0,0.5);
\draw [color = black, line width =1pt](0,-0.5) --(0,0);
\filldraw[draw=black,fill=gray!20](0.4,-0.3) rectangle (1.2,0.3);
\node at(0.8,0) {\small $P_{N,1}$};
\end{tikzpicture}}))
    =\sum_{a,b\in\{1,2,3\}} \raisebox{-.25in}{
\begin{tikzpicture}
\tikzset{->-/.style=
{decoration={markings,mark=at position #1 with
{\arrow{latex}}},postaction={decorate}}}
\filldraw[draw=black,fill=gray!20] (-0.7,-0.5) rectangle (1.3,0.5);
\draw [color = black, line width =1pt](-0.1,0) --(-0.7,0);
\draw [color = black, line width =1pt](0.1,0) --(0.4,0);
\draw[line width =1pt,decoration={markings, mark=at position 0.8 with {\arrow{>}}},postaction={decorate}](0.8,0) --(1.3,0);
\draw[line width =1pt,decoration={markings, mark=at position 0.8 with {\arrow{>}}},postaction={decorate}](0,0) --(0,0.5);
\draw [color = black, line width =1pt](0,-0.5) --(0,0);
\filldraw[draw=black,fill=gray!20] (0.6,0) circle(0.25);
\node at(0.6,0) {\small $N$};
\node[left] at(-0.7,0) {\small $a$};
\node[right] at(1.3,0) {\small $a$};
\node[below] at(0,-0.5) {\small $b$};
\node[above] at(0,0.5) {\small $b$};
\end{tikzpicture}} \in\cS_{\bar\omega}(\mathbb P_4).
\end{align*}
Lemma \ref{lem-transparency-P4} implies that
\begin{align*}
    \mathbb S_{e_1}(\mathbb S_{e_2}(\raisebox{-.15in}{
\begin{tikzpicture}
\tikzset{->-/.style=
{decoration={markings,mark=at position #1 with
{\arrow{latex}}},postaction={decorate}}}
\filldraw[draw=white,fill=gray!20] (-0.7,-0.5) rectangle (2,0.5);
\draw [color = red, line width =1pt](-0.5,-0.5) --(1.8,-0.5);
\draw [color = red, line width =1pt](-0.5,0.5) --(1.8,0.5);
\draw [color = blue, line width =1pt](-0.7,-0.4) --(-0.7,0.4);
\draw [color = blue, line width =1pt](2,-0.4) --(2,0.4);
\draw [color = black, line width =1pt](-0.1,0) --(-0.7,0);
\draw [color = black, line width =1pt](0.1,0) --(0.4,0);
\draw[line width =1pt,decoration={markings, mark=at position 0.8 with {\arrow{>}}},postaction={decorate}](0.8,0) --(2,0);
\draw[line width =1pt,decoration={markings, mark=at position 0.8 with {\arrow{>}}},postaction={decorate}](0,0) --(0,0.5);
\draw [color = black, line width =1pt](0,-0.5) --(0,0);
\filldraw[draw=black,fill=gray!20](0.4,-0.3) rectangle (1.2,0.3);
\node at(0.8,0) {\small $P_{N,1}$};
\end{tikzpicture}}))
    =\omega^{\frac{2N}{3}}\mathbb S_{e_1}(\mathbb S_{e_2}(\raisebox{-.15in}{
\begin{tikzpicture}
\tikzset{->-/.style=
{decoration={markings,mark=at position #1 with
{\arrow{latex}}},postaction={decorate}}}
\filldraw[draw=white,fill=gray!20] (-0.7,-0.5) rectangle (2,0.5);
\draw [color = red, line width =1pt](-0.5,-0.5) --(1.8,-0.5);
\draw [color = red, line width =1pt](-0.5,0.5) --(1.8,0.5);
\draw [color = blue, line width =1pt](-0.7,-0.4) --(-0.7,0.4);
\draw [color = blue, line width =1pt](2,-0.4) --(2,0.4);
\draw [color = black, line width =1pt](-0.1,0) --(-0.7,0);
\draw [color = black, line width =1pt](0.1,0) --(0.4,0);
\draw[line width =1pt,decoration={markings, mark=at position 0.8 with {\arrow{>}}},postaction={decorate}](0.8,0) --(2,0);
\draw[line width =1pt,decoration={markings, mark=at position 0.8 with {\arrow{>}}},postaction={decorate}](0,0) --(0,0.5);
\draw [color = black, line width =1pt](0,-0.5) --(0,0);
\filldraw[draw=black,fill=gray!20](0.4,-0.3) rectangle (1.2,0.3);
\node at(0.8,0) {\small $P_{N,1}$};
\end{tikzpicture}}))\in\cS_{\bar\omega}(\mathbb P_4).
\end{align*}
Remark \ref{rem-inject-splitting}, the injectivity of the splitting map for (non-reduced) stated ${\rm SL}_3$-skein algebras, implies that 
\begin{align}\label{transparency-torus}
    \raisebox{-.15in}{
\begin{tikzpicture}
\tikzset{->-/.style=
{decoration={markings,mark=at position #1 with
{\arrow{latex}}},postaction={decorate}}}
\filldraw[draw=white,fill=gray!20] (-0.7,-0.5) rectangle (2,0.5);
\draw [color = red, line width =1pt](-0.5,-0.5) --(1.8,-0.5);
\draw [color = red, line width =1pt](-0.5,0.5) --(1.8,0.5);
\draw [color = blue, line width =1pt](-0.7,-0.4) --(-0.7,0.4);
\draw [color = blue, line width =1pt](2,-0.4) --(2,0.4);
\draw [color = black, line width =1pt](0,0) --(-0.7,0);
\draw [color = black, line width =1pt](0,0) --(0.4,0);
\draw[line width =1pt,decoration={markings, mark=at position 0.85 with {\arrow{>}}},postaction={decorate}](0.8,0) --(2,0);
\draw[line width =1pt,decoration={markings, mark=at position 0.8 with {\arrow{>}}},postaction={decorate}](0,0.1) --(0,0.5);
\draw [color = black, line width =1pt](0,-0.5) --(0,-0.1);
\filldraw[draw=black,fill=gray!20](0.4,-0.3) rectangle (1.2,0.3);
\node at(0.8,0) {\small $P_{N,1}$};
\end{tikzpicture}}=\omega^{\frac{2N}{3}}
\raisebox{-.15in}{
\begin{tikzpicture}
\tikzset{->-/.style=
{decoration={markings,mark=at position #1 with
{\arrow{latex}}},postaction={decorate}}}
\filldraw[draw=white,fill=gray!20] (-0.7,-0.5) rectangle (2,0.5);
\draw [color = red, line width =1pt](-0.5,-0.5) --(1.8,-0.5);
\draw [color = red, line width =1pt](-0.5,0.5) --(1.8,0.5);
\draw [color = blue, line width =1pt](-0.7,-0.4) --(-0.7,0.4);
\draw [color = blue, line width =1pt](2,-0.4) --(2,0.4);
\draw [color = black, line width =1pt](-0.1,0) --(-0.7,0);
\draw [color = black, line width =1pt](0.1,0) --(0.4,0);
\draw[line width =1pt,decoration={markings, mark=at position 0.8 with {\arrow{>}}},postaction={decorate}](0.8,0) --(2,0);
\draw[line width =1pt,decoration={markings, mark=at position 0.8 with {\arrow{>}}},postaction={decorate}](0,0) --(0,0.5);
\draw [color = black, line width =1pt](0,-0.5) --(0,0);
\filldraw[draw=black,fill=gray!20](0.4,-0.3) rectangle (1.2,0.3);
\node at(0.8,0) {\small $P_{N,1}$};
\end{tikzpicture}}\in\cS_{\bar\omega}(T_{1,1}).
\end{align}

Note that the embedding from $U$ to $\widetilde{\fS}=\fS\times [-1,1]$ induces a $\mathbb C$-linear map from 
$\cS_{\bar\omega}(T_{1,1})$ to 
$\cS_{\bar\omega}(\fS)$ that maps relation \eqref{transparency-torus} to \eqref{transparency-fS-1}.
This completes the proof.
\end{proof}

\begin{remark}
    Proposition \ref{prop-transparency-pb} still holds if we remove the  condition $\partial\fS\neq\emptyset$. 
    When $\partial\fS\neq\emptyset$, we can suppose $\beta$ is a knot or a stated arc. 
    When $\beta$ is a knot, the proof of Proposition \ref{prop-transparency-pb} still works.
    When $\beta$ is a stated arc, the open regular neighborhood of $\alpha\cup\beta$ is isomorphic to $\mathbb A_2\times(-1,1)$, where $\mathbb A_2$ is obtained from the closed annulus $\bigodot$ (see Figure \ref{fig1}) by removing one point from each 
    of the two boundary components. 
    Cutting $\mathbb{A}_2$ along any ideal arc $e$ in $\mathbb{A}_2$ connecting the two punctures yields $\mathbb{P}_4$.
    Cutting $\mathbb A_2$ along $e$ and using the same techniques in the proof of  Proposition \ref{prop-transparency-pb}, we can show equations 
    \eqref{transparency-fS-1} and \eqref{transparency-fS-2} when $\beta$ is a stated arc. 
\end{remark}

\subsection{Proof of Proposition \ref{prop-center}}

We are now ready to prove Proposition \ref{prop-center}, which asserts that for the Frobenius map $\mathcal{F} : \mathscr{S}_{\bar{\eta}}(\fS) \to \mathscr{S}_{\bar{\omega}}(\fS)$, the set $\im_{\bar\omega} \cF$ lies in the center of $\mathscr{S}_{\bar{\omega}}(\fS)$, where
    $$\im_{\bar\omega} \cF=\begin{cases}
	\cF (\Sfe) & \mbox{if } 3\nmid N',\\
	\cF (\Sfe_3) & \mbox{if } 3\mid N',
\end{cases}$$ 
where $\Sfe_3$ is as defined in Definition \ref{def-sualgebra3}.

\begin{proof}[Proof of Proposition \ref{prop-center}]

{\bf Case 1} when $3\nmid N'$.
Then we have $N'=N$ and $\omega^{\frac{2N}{3}}
=\bar\omega^{2N}= \bar\omega^{2N'}=1.$
Proposition \ref{prop-transparency-pb} implies that
in $\cS_{\bar\omega}(\fS)$ we have the following relations
\begin{equation}\label{transparency-fS-1-1}
\raisebox{-.10in}{
\begin{tikzpicture}
\tikzset{->-/.style=
{decoration={markings,mark=at position #1 with
{\arrow{latex}}},postaction={decorate}}}
\filldraw[draw=white,fill=gray!20] (-0.7,-0.5) rectangle (2,0.5);
\draw [color = black, line width =1pt](0,0) --(-0.7,0);
\draw [color = black, line width =1pt](0,0) --(0.4,0);
\draw[line width =1pt,decoration={markings, mark=at position 0.85 with {\arrow{>}}},postaction={decorate}](0.8,0) --(2,0);
\draw[line width =1pt,decoration={markings, mark=at position 0.8 with {\arrow{>}}},postaction={decorate}](0,0.1) --(0,0.5);
\draw [color = black, line width =1pt](0,-0.5) --(0,-0.1);
\filldraw[draw=black,fill=gray!20](0.4,-0.3) rectangle (1.2,0.3);
\node at(0.8,0) {\small $P_{N,1}$};
\end{tikzpicture}}
=
\raisebox{-.10in}{
\begin{tikzpicture}
\tikzset{->-/.style=
{decoration={markings,mark=at position #1 with
{\arrow{latex}}},postaction={decorate}}}
\filldraw[draw=white,fill=gray!20] (-0.7,-0.5) rectangle (2,0.5);
\draw [color = black, line width =1pt](-0.1,0) --(-0.7,0);
\draw [color = black, line width =1pt](0.1,0) --(0.4,0);
\draw[line width =1pt,decoration={markings, mark=at position 0.8 with {\arrow{>}}},postaction={decorate}](0.8,0) --(2,0);
\draw[line width =1pt,decoration={markings, mark=at position 0.8 with {\arrow{>}}},postaction={decorate}](0,0) --(0,0.5);
\draw [color = black, line width =1pt](0,-0.5) --(0,0);
\filldraw[draw=black,fill=gray!20](0.4,-0.3) rectangle (1.2,0.3);
\node at(0.8,0) {\small $P_{N,1}$};
\end{tikzpicture}}
\text{ and }
\raisebox{-.10in}{
\begin{tikzpicture}
\tikzset{->-/.style=
{decoration={markings,mark=at position #1 with
{\arrow{latex}}},postaction={decorate}}}
\filldraw[draw=white,fill=gray!20] (-0.7,-0.5) rectangle (2,0.5);
\draw [color = black, line width =1pt](0,0) --(-0.7,0);
\draw [color = black, line width =1pt](0,0) --(0.4,0);
\draw[line width =1pt,decoration={markings, mark=at position 0.85 with {\arrow{>}}},postaction={decorate}](0.8,0) --(2,0);
\draw[color = black, line width =1pt](0,0.1) --(0,0.5);
\draw [line width =1pt,decoration={markings, mark=at position 0.5 with {\arrow{<}}},postaction={decorate}](0,-0.5) --(0,-0.1);
\filldraw[draw=black,fill=gray!20](0.4,-0.3) rectangle (1.2,0.3);
\node at(0.8,0) {\small $P_{N,1}$};
\end{tikzpicture}}
=
\raisebox{-.10in}{
\begin{tikzpicture}
\tikzset{->-/.style=
{decoration={markings,mark=at position #1 with
{\arrow{latex}}},postaction={decorate}}}
\filldraw[draw=white,fill=gray!20] (-0.7,-0.5) rectangle (2,0.5);
\draw [color = black, line width =1pt](-0.1,0) --(-0.7,0);
\draw [color = black, line width =1pt](0.1,0) --(0.4,0);
\draw[line width =1pt,decoration={markings, mark=at position 0.8 with {\arrow{>}}},postaction={decorate}](0.8,0) --(2,0);
\draw[line width =1pt,decoration={markings, mark=at position 0.92 with {\arrow{>}}},postaction={decorate}](0,0.5) --(0,-0.5);
%
\filldraw[draw=black,fill=gray!20](0.4,-0.3) rectangle (1.2,0.3);
\node at(0.8,0) {\small $P_{N,1}$};
\end{tikzpicture}}.
\end{equation}
Theorem \ref{AP-thm-Fro}\ref{AP-thm-Fro-b} and equation \eqref{transparency-fS-1-1} imply that
$\im\cF\subset\Zz.$

{\bf Case 2} when $3\mid N'$.
Then we have $N'=3N$ and the order of 
$\omega^{\frac{2N}{3}}$ is $3$.
Let $\alpha,\beta$ be two 
web diagrams in $B_{\fS}$ (Definition \ref{def-B_S}).
Equations \eqref{transparency-fS-1} and \eqref{transparency-fS-2}, together with the definition of the intersection number $i_3(\alpha,\beta)$ in equation \eqref{i3}, imply that
\begin{align*}
    \cF(\alpha)\beta = (\omega^{\frac{2N}{3}})^{i_3(\alpha,\beta)}\beta\cF(\alpha).
\end{align*}
Thus we have $ \cF(\alpha)\beta = \beta\cF(\alpha)$
if $\alpha\in B_{\fS,3}$, by definition of $B_{\fS,3}$ (Definition \ref{def-sualgebra3}). 
This shows that $\cF (\Sfe_3)\subset\Zz.$

\end{proof}

\subsection{Proof of Proposition \ref{prop-tran1}}\label{subsec.3d-transparency} 
Recall that in \S\ref{sec-3-manifolds} we assumed $3\nmid N'$. Thus we have \eqref{transparency-fS-1-1}, which immediately implies Proposition \ref{prop-tran1}, because one can `pass through' $\mathcal{F}(l) = (\cup_{1\le i \le m} l_i^{[P_{N,1}]})$ when isotoping $T_1$ into $T_2$.

\bibliography{ref.bib}

\begin{thebibliography}{10}

\bibitem{bloomquist2020chebyshev}
Wade Bloomquist and Thang~TQ L{\^e}.
\newblock {The Chebyshev--Frobenius homomorphism for stated skein modules of
  3-manifolds}.
\newblock {\em Mathematische Zeitschrift}, 301:1063--1105, 2022.

\bibitem{bonahon1996}
Francis Bonahon.
\newblock {Shearing hyperbolic surfaces, bending pleated surfaces and
  Thurston’s symplectic form}.
\newblock {\em Ann. Fac. Sci. Toulouse Math. (6)}, 5(2):233--297, 1996.

\bibitem{bonahon2023central}
Francis Bonahon and Vijay Higgins.
\newblock {Central elements in the ${\rm SL}_d$-skein algebra of a surface}.
\newblock {\em Mathematische Zeitschrift}, 308:article number 1, 2024.

\bibitem{bonahon2016representations}
Francis Bonahon and Helen Wong.
\newblock {Representations of the Kauffman bracket skein algebra I: invariants
  and miraculous cancellations}.
\newblock {\em Inventiones Mathematicae}, 204:195--243, 2016.

\bibitem{bonahon2017representations}
Francis Bonahon and Helen Wong.
\newblock {Representations of the Kauffman bracket skein algebra, II: Punctured
  surfaces}.
\newblock {\em Algebraic \& Geometric topology}, 17(6):3399--3434, 2017.

\bibitem{bonahon2019representations}
Francis Bonahon and Helen Wong.
\newblock {Representations of the Kauffman bracket skein algebra III: closed
  surfaces and naturality}.
\newblock {\em Quantum Topology}, 10:325--398, 2019.

\bibitem{bullock}
Doug Bullock.
\newblock A finite set of generators for the kauffman bracket skein algebra.
\newblock {\em Math. Z.}, 231:91--101, 1999.

\bibitem{cautis2014webs}
Sabin Cautis, Joel Kamnitzer, and Scott Morrison.
\newblock Webs and quantum skew howe duality.
\newblock {\em Mathematische Annalen}, 360:351--390, 2014.

\bibitem{cremaschi2024monomial}
Tommaso Cremaschi and Daniel~C Douglas.
\newblock {Monomial web basis for the SL(N) skein algebra of the twice
  punctured sphere}.
\newblock {\em arXiv preprint arXiv:2407.04178}, 2024.

\bibitem{douglas2024tropical}
Daniel~C Douglas and Zhe Sun.
\newblock {Tropical Fock--Goncharov coordinates for-webs on surfaces I:
  construction}.
\newblock In {\em Forum of Mathematics, Sigma}, volume~12, page~e5. Cambridge
  University Press, 2024.

\bibitem{fockgoncharov06}
Vladimir~V. Fock and Alexander~B. Goncharov.
\newblock Moduli spaces of local systems and higher teichm\"uller theory.
\newblock {\em Publ. Math. Inst. Hautes \'Etudes Sci.}, 103:1--211, 2006.

\bibitem{frohman2019unicity}
Charles Frohman, Joanna Kania-Bartoszynska, and Thang L{\^e}.
\newblock {Unicity for representations of the Kauffman bracket skein algebra}.
\newblock {\em Inventiones Mathematicae}, 215:609--650, 2019.

\bibitem{frohman2021dimension}
Charles Frohman, Joanna Kania-Bartoszynska, and Thang L{\^e}.
\newblock {Dimension and trace of the Kauffman bracket skein algebra}.
\newblock {\em Transactions of the American Mathematical Society, Series B},
  8(18):510--547, 2021.

\bibitem{frohman2023sliced}
Charles Frohman, Joanna Kania-Bartoszynska, and Thang~TQ L{\^e}.
\newblock {Sliced skein algebras and geometric Kauffman bracket}.
\newblock {\em arXiv preprint arXiv:2310.06189}, 2023.

\bibitem{frohman20223}
Charles Frohman and Adam~S Sikora.
\newblock {SU(3)-skein algebras and webs on surfaces}.
\newblock {\em Mathematische Zeitschrift}, 300(1):33--56, 2022.

\bibitem{GJS24}
Iordan Ganev, David Jordan, and Pavel Safronov.
\newblock {The quantum Frobenius for character varieties and multiplicative
  quiver varieties}.
\newblock {\em J. Eur. Math. Soc.}, page published online first, 2024.

\bibitem{goodearl2004introduction}
Kenneth~R Goodearl and Robert~B Warfield.
\newblock {\em An introduction to noncommutative Noetherian rings}.
\newblock Cambridge university press, 2004.

\bibitem{gordan1873ueber}
Paul Gordan.
\newblock {Ueber die Aufl{\"o}sung linearer Gleichungen mit reellen
  Coefficienten}.
\newblock {\em Mathematische Annalen}, 6(1):23--28, 1873.

\bibitem{higgins2020triangular}
Vijay Higgins.
\newblock {Triangular decomposition of ${\rm SL}_3$ skein algebras}.
\newblock {\em Quantum Topology}, 14:1--63, 2023.

\bibitem{higgins2024miraculous}
Vijay Higgins.
\newblock {Miraculous cancellations and the quantum Frobenius for $ SL_3 $
  skein modules}.
\newblock {\em arXiv preprint arXiv:2409.00351}, 2024.

\bibitem{HW}
Hiroaki Karuo and Zhihao Wang.
\newblock {Center of stated SL($n$)-skein algebras}.
\newblock {\em arXiv preprint arXiv:2408.12520}, 2024.

\bibitem{kim2011sl3}
Hyun~Kyu Kim.
\newblock {${\rm SL}_3$-laminations as bases for ${\rm PGL}_3$ cluster
  varieties for surfaces}.
\newblock {\em arXiv preprint arXiv:2011.14765}, 2020.

\bibitem{HLW}
Hyun~Kyu Kim, Thang~TQ L{\^e}, and Zhihao Wang.
\newblock {Frobenius maps for stated $SL_n$-skein algebras}.
\newblock {\em in preparation}, 2022.

\bibitem{korinman2021unicity}
Julien Korinman.
\newblock Unicity for representations of reduced stated skein algebras.
\newblock {\em Topology and its Applications}, 293:107570, 2021.

\bibitem{kuperberg}
Greg Kuperberg.
\newblock {Spiders for rank 2 Lie algebras}.
\newblock {\em Commun. Math. Phys.}, 180:109--151, 1996.

\bibitem{le2021stated}
Thang~TQ L{\^e} and Adam~S Sikora.
\newblock {Stated SL(n)-skein modules and algebras}.
\newblock {\em arXiv preprint arXiv:2201.00045}, 2021.

\bibitem{le2023quantum}
Thang~TQ L{\^e} and Tao Yu.
\newblock {Quantum traces for $ SL_n $-skein algebras}.
\newblock {\em arXiv preprint arXiv:2303.08082}, 2023.

\bibitem{GL93}
George Lusztig.
\newblock {\em Introduction to quantum groups}, volume 110 of {\em Progress in
  Mathematics}.
\newblock Birkh\"auser, 1993.

\bibitem{milne2012algebraic}
James~S Milne.
\newblock {\em Algebraic geometry}.
\newblock Allied Publishers, 2012.

\bibitem{PW}
Brian Parshall and Jian-pan Wang.
\newblock Quantum linear groups.
\newblock {\em Memoirs of the American Mathematical Society}, 89(439):1--157,
  1991.

\bibitem{prz}
J.~Przytycki.
\newblock {Fundamentals of Kauffman bracket skein modules}.
\newblock {\em Kobe J. Math.}, 16:45--66, 1999.

\bibitem{przytycki1998fundamentals}
J{\'o}zef~H Przytycki.
\newblock {Fundamentals of Kauffman bracket skein modules}.
\newblock {\em arXiv preprint math/9809113}, 1998.

\bibitem{queffelec2018sutured}
Hoel Queffelec and David Rose.
\newblock {Sutured annular Khovanov-Rozansky homology}.
\newblock {\em Transactions of the American Mathematical Society},
  370(2):1285--1319, 2018.

\bibitem{saito2005lecture}
Toshio Saito, Martin Scharlemann, and Jennifer Schultens.
\newblock {Lecture notes on generalized Heegaard splittings}.
\newblock {\em arXiv preprint math/0504167}, 2005.

\bibitem{schrader}
Gus Schrader and Alexander Shapiro.
\newblock Continuous tensor categories from quantum groups i: algebraic
  aspects.
\newblock {\em arXiv preprint arXiv:1708.08107}, 2017.

\bibitem{sikora2001SLn}
Adam Sikora.
\newblock {${\rm SL}_n$-character varieties as spaces of graphs}.
\newblock {\em Transactions of the American Mathematical Society},
  353(7):2773--2804, 2001.

\bibitem{sikora2005skein}
Adam~S Sikora.
\newblock {Skein theory for $SU(n)$-quantum invariants}.
\newblock {\em Algebraic \& Geometric Topology}, 5(3):865--897, 2005.

\bibitem{turaev}
V.G. Turaev.
\newblock {\em {Algebra of loops on surfaces, algebra of knots, and
  quantization}}, volume~9 of {\em Adv. Ser. Math. Phys.}, pages 59--95.
\newblock World Sci. Publ., Teaneck, NJ, 1989.

\bibitem{wang2023finiteness}
Zhihao Wang.
\newblock {Finiteness and dimension of stated skein modules over Frobenius}.
\newblock {\em arXiv preprint arXiv:2309.10920}, 2023.

\bibitem{wang2023stated}
Zhihao Wang.
\newblock {On stated $ SL (n) $-skein modules}.
\newblock {\em arXiv preprint arXiv:2307.10288}, 2023.

\bibitem{Wan24}
Zhihao Wang.
\newblock {Stated $SL_n$-skein modules, roots of unity, and TQFT}.
\newblock {\em arXiv preprint arXiv:2401.09995}, 2024.

\bibitem{wang2023representation}
Zhihao Wang.
\newblock Representation-reduced stated skein modules and algebras.
\newblock {\em Journal of Algebra}, 661:831--852, 2025.

\bibitem{yu2023center}
Tao Yu.
\newblock Center of the stated skein algebra.
\newblock {\em arXiv preprint arXiv:2309.14713}, 2023.

\end{thebibliography}

\end{document}